\documentclass[10pt]{article}
\usepackage{mathpazo}
\usepackage{yk}
\usepackage{mathrsfs}
\usepackage{mathtools}
\usepackage[numbers]{natbib}
\usepackage{geometry}
\usepackage{color}
\usepackage{hyperref}
\usepackage{lineno}

\newtheorem{theorem}{Theorem}[section]
\newtheorem{assumption}[theorem]{Assumption}

\newtheorem{corollary}[theorem]{Corollary}
\newtheorem{proof}[theorem]{Proof}

\newtheorem{lemma}[theorem]{Lemma}

\newtheorem{proposition}[theorem]{Proposition}
\newtheorem{remark}[theorem]{Remark}

\DeclareMathOperator*{\argmin}{arg\,min}

\usepackage{algorithm}
\usepackage{algorithmic} 

\usepackage{booktabs}
\RequirePackage{graphicx}

\def\bA{\mathbf{A}}

\def\bG{\mathbf{G}}

\def\bI{\mathbf{I}}

\def\bM{\mathbf{M}}

\def\bS{\mathbf{S}}
\def\bT{\mathbf{T}}

\def\bW{\mathbf{W}}
\def\bX{\mathbf{X}}
\def\bY{\mathbf{Y}}
\def\bZ{\mathbf{Z}}

\def\bPhi{{\boldsymbol\Phi}}

\def\indc{\mathbf{1}}

\def\bb{\mathbf{b}}

\def\be{\mathbf{e}}

\def\be{\mathbf{e}}

\def\br{\mathbf{r}}

\def\bv{\mathbf{v}}
\def\bw{\mathbf{w}}
\def\bx{\mathbf{x}}
\def\by{\mathbf{y}}

\def\balpha{{\boldsymbol\alpha}}

\def\beps{{\boldsymbol\epsilon}}

\def\btheta{{\boldsymbol\theta}}

\def\btau{{\boldsymbol\tau}}

\def\bbE{\mathbb{E}}

\def\bbP{\mathbb{P}}

\def\bbR{\mathbb{R}}

\def\Ber{{\sf Ber}}

\def\cB{\mathcal{B}}

\def\cD{\mathcal{D}}
\def\cE{\mathcal{E}}

\def\cH{\mathcal{H}}
\def\cI{\mathcal{I}}

\def\cN{\mathcal{N}}

\def\cP{\mathcal{P}}

\def\cU{\mathcal{U}}

\def\Cov{{\sf Cov}}

\def\diag{{\sf diag}}

\def\exp{{\sf exp}}

\def\eps{{\epsilon}}

\def\ie{{\it i.e.}}

\def\KL{{\sf KL}}

\def\med{{\sf med}}

\def\norm#1{\left\|#1\right\|}

\def\Poi{{\sf Poi}}
\def\Pr{{\bbP}}

\def\reals{{\bbR}}

\def\sign{{\sf sign}}

\def\tr{{\sf tr}}

\def\TV{{\sf TV}}

\def\Var{{\sf Var}}
\def\Cov{{\sf Cov}}
\def\Bias{{\sf Bias}}

\def\SE{{\sf SE}}

\newcommand{\TL}{\mathrm{TL}}
\newcommand{\DVCM}{\mathrm{DVCM}}
\newcommand{\GDVCM}{\mathrm{GDVCM}}
\newcommand{\LR}{\mathrm{LR}}
\newcommand{\GLR}{\mathrm{GLR}}

\newcommand{\mse}{\mathrm{MSE}}

\usepackage{caption}
\newtheorem{assumptionprime}{Assumption}

\newtheorem{assumptiondoubleprime}{Assumption}

\usepackage{url}

\title{Minimax optimal adaptive structured transfer learning through semi-parametric domain-varying coefficient model}

\author{Hanxiao Chen and Debarghya Mukherjee \\ Department of Mathematics and Statistics, Boston University}

\date{\today}

\begin{document}

\maketitle

\begin{abstract}
Transfer learning aims to improve inference in a target domain by leveraging information from related source domains, but its effectiveness critically depends on how cross-domain heterogeneity is modeled and controlled. When the conditional mechanism linking covariates and responses varies across domains, indiscriminate information pooling can lead to negative transfer, degrading performance relative to target-only estimation.
We study a multi-source, single-target transfer learning problem under conditional distributional drift and propose a semiparametric domain–varying coefficient model (DVCM), in which domain-relatedness is encoded through an observable domain identifier. This framework generalizes classical varying-coefficient models to structured transfer learning and interpolates between invariant and fully heterogeneous regimes.
Building on this model, we develop an adaptive transfer learning estimator that selectively borrows strength from informative source domains while provably safeguarding against negative transfer. Our estimator is computationally efficient and easy to implement; we also show that it is minimax rate-optimal and derive its asymptotic distribution, enabling valid uncertainty quantification and hypothesis testing despite data-adaptive pooling and shrinkage. 
Our results precisely characterize the interplay among domain heterogeneity, the smoothness of the underlying mean function, and the number of source domains and are corroborated by comprehensive numerical experiments and two real-data applications.
\end{abstract}


\section{Introduction}

Transfer learning, domain adaptation, and multi-task learning are modern machine-learning methodologies that aim to improve prediction or estimation in a target domain by leveraging data from related source domains, especially when labeled target data are scarce or costly to obtain. These methods have been successfully applied across a wide range of areas \citep{yang2020transfer, shao2014transfer, cai2020transferdrug}. The key challenge is that domains typically differ, via shifts in covariate distributions (\emph{covariate shift}), response distributions (\emph{label shift}), or the conditional mechanism linking them (\emph{concept/posterior drift}). Therefore, any gain hinges on how “relatedness” between source and target is quantified and enforced. Most of the existing approaches in transfer learning align distributions or representations, share parameters across tasks, or reweight instances to emphasize target-relevant information. 
However, if information from source domains is incorporated blindly, without adequate preventive measures against uninformative or mismatched sources, it may degrade the estimator’s performance relative to what could be achieved using only the target data, a phenomenon commonly referred to as \textit{negative transfer} in the literature \cite{pan2009survey}.
Therefore, the goal of transfer learning is to construct an \emph{adaptive} estimator that efficiently borrows information from related sources while ensuring performance never degrades relative to a target-only estimator. A further challenge, often underemphasized in theoretical work, is to provide valid uncertainty quantification (e.g., in terms of asymptotic distribution) for these estimators to draw valid statistical inference.

Modern transfer learning problems increasingly involve data collected across multiple heterogeneous environments or domains, where the relationship between covariates and responses is not strictly invariant but instead exhibits systematic variation. Such heterogeneity may arise from differences in population characteristics, experimental conditions, data-collection protocols, or other contextual factors. A central challenge in these settings is to leverage information from related environments to improve inference in a target environment, while avoiding degradation in performance when the source environments are only weakly informative. To formalize this setting, we consider an environment-indexed framework. Each environment is associated with a \emph{domain identifier} $U_e \in \cU$, where $\cU$ is a compact set generated from some (unknown) distribution $P_U$. Within environment $e$, the covariate–response pairs are generated as
$$
X_{ie} \sim P_{X\mid U}(\cdot \mid U_e), 
\qquad
Y_{ie} \mid (X_{ie},U_e) \sim P_{U_e}(\cdot \mid X_{ie}),
$$
so that the conditional mechanism relating $Y$ to $X$ may vary across environments. The primary goal of transfer learning in this setting is to improve inference for the target environment $u_0$ by borrowing information from related environments, while safeguarding against the risk of negative transfer. However, to efficiently transfer information from the source domains, it is essential to model how the conditional distribution varies smoothly with respect to the domain identifier, so that information can be shared across nearby environments.

Motivated by this framework, we model the cross-environment heterogeneity through a \emph{Domain–Varying Coefficient Model} (DVCM), inspired by the classical varying–coefficient model (VCM; see, e.g.,\cite{hastie1993varying, cai2000efficient}). Specifically, we assume that
$$
\bbE[Y \mid X,U] = g^{-1}\!\big(X^\top \theta(U)\big),
$$
where $g$ is a known link function and $\theta(\cdot)$ is an unknown coefficient function of the domain index $U$. This specification represents a structured and interpretable restriction of the general conditional distribution of $Y$ given $(X, U)$, allowing it to vary smoothly across environments. The key distinction from a standard generalized linear model (GLM) is that GLMs assume fixed coefficients across observations, whereas DVCMs allow coefficients to vary with $U$, thereby providing a flexible yet parsimonious mechanism for capturing systematic heterogeneity across domains.

In our generalized linear DVCM framework, we assume access to data from $K{+}1$ domains, indexed
by $k \in \{0,1,\dots,K\}$, where $k=0$ denotes the target domain and
$k \in \{1,\dots,K\}$ denote the source domains.
From domain $k$, we observe $n_k$ response–predictor pairs
$\mathcal{D}_k = \{(Y_{ki},X_{ki})\}_{i=1}^{n_k}$.
Each domain is additionally associated with a \emph{domain identifier} $U_k$, which is constant
within a domain (i.e., $U_{ki} \equiv U_k$ for all $i$) but varies across domains.
Under this model, the conditional mean of the response satisfies
\begin{equation}
\label{eq:dgp_dvcm}
\bbE[Y_{ki} \mid X_{ki}, U_k] = g^{-1}\!\big(X_{ki}^\top \theta(U_k)\big).
\end{equation}
In practice, the domain identifier $U_k$ may represent calendar time (e.g., year or quarter of
data collection) in temporal processes, geographic or institutional indicators (e.g., city,
region, hospital) in spatial or multi-site studies, or cohort characteristics such as tenure or
exposure duration.
In biomedical and sensing applications, $U_k$ may encode instrument-specific or batch effects
(e.g., scanner model, assay batch, or sensor platform), which are typically constant within a
domain but vary across domains.
Similarly, under stratified sampling designs, $U_k$ may consist of the indicators defining the
$k$th stratum.

Under the above-mentioned generalized linear DVCM model, our goal is to estimate the target-domain coefficient $\theta(u_0)$ by borrowing information from the source domains while guarding against \emph{negative transfer}. As mentioned previously, the potential gains from transfer hinge on (i) the smoothness of the coefficient function $\theta(\cdot)$ and (ii) the similarity between $U_k$ and $u_0$. Even when $\theta(\cdot)$ is not very smooth, transfer can be beneficial if some source identifier
$U_k$ lies sufficiently close to $u_0$, while conversely, when no source is particularly close
to $u_0$, a high degree of smoothness of $\theta(\cdot)$ can still enable effective transfer by
ensuring that $\theta(U_k)$ remains close to $\theta(u_0)$.

Building on this insight, we propose a \emph{minimax-optimal, computationally efficient, and adaptive} estimator $\hat \btheta_{\TL}(u_0)$ for $\btheta(u_0)$ that exploits source information when relevant and remains robust to negative transfer, i.e., its risk is never worse than a target-only estimator of $\btheta(u_0)$, and is conceptually simple and easy to implement. As a baseline, in the absence of source data, a target-only least-squares (or GLM) estimator attains the optimal rate for $\theta(u_0)$, but it leverages neither the smoothness of $\theta$ nor cross-domain similarities. To incorporate both, we first form a nonparametric pilot $\hat\theta_{\DVCM}(u_0)$ (e.g., via local polynomial regression) by pooling sources (and a split of the target) whose $U_k$ are near $u_0$. We then fit a GLM on the target domain with an \emph{adaptive ridge penalty} that shrinks the GLM estimator toward $\hat\theta_{\DVCM}(u_0)$. The key challenge, therefore, lies in designing the penalty in a careful, data-driven manner so as to guard against potential negative transfer. To this end, we construct a penalty based on \textit{inverse-variance reweighting}, which effectively serves this purpose.
In Section \ref{sec:method}, we provide a practical recipe for constructing such a penalty, and establish sharp theoretical guarantees for its performance. 

Although our methodology is applicable for vector-valued $U$, we conduct our theoretical study under an univariate ($\dim(U) = 1$) domain indicator for the simplicity of presentation. One of our key theoretical contributions is to rigorously establish that $\hat \btheta_{\TL}(u_0)$ achieves the following minimax-optimal rate (Theorems \ref{thm:rate-TL} and \ref{thm:minimax-TL}): 
$$
\inf_{\hat \btheta} \sup_{\btheta, \{\cP_k\}_{k = 0}^K} \bbE\left[\|\hat \btheta - \btheta(u_0)\|_2^2\right] \asymp \left\{n_0^{-1} \wedge \max\left((K/\gamma)^{-2\beta}, (n/\gamma)^{-\frac{2\beta}{2\beta+1}}, n^{-1}\right)\right\} \,,
$$
where $K$ is the number of source domains, $n=\sum_{k=0}^K n_k$ is the total sample size, $\beta$ denotes the smoothness of the coefficients of $\btheta$, and $\gamma$ quantifies the proximity of source and the target identifiers (e.g., variability among $U_i$'s, smaller $\gamma$ means closer domain indices; see Section \ref{sec:linear} for details). Two implications are immediate from the rate: (i) it is never worse than $n_0^{-1}$, so the estimator is immune to negative transfer; and (ii) it improves when either $\gamma$ is small (i.e., $U_k$ is close to $U_0$) or $\beta$ is large (i.e., $\btheta$ is smoother). 
Furthermore, the term $(K/\gamma)^{-2\beta}$ is unavoidable. When $K$ is small and $\gamma$ is large, only limited information can be effectively pooled from the source domains, even in the presence of infinite source data. A large $\gamma$ indicates that $\theta(u_0)$ and $\theta(u_k)$ are not sufficiently close, while a small $K$ implies inadequate local information to reliably infer $\theta(u_0)$ from $\theta(u_k)$. This limitation is analogous to the behavior of the bias term encountered in standard nonparametric regression.
Our analysis can be extended to multivariate $U$ in a straightforward manner, with no additional insight, beyond routine bookkeeping.

While minimax optimality characterizes the fundamental estimation difficulty, rates alone do not provide uncertainty quantification. In our setting, deriving a valid inference is particularly delicate: the proposed estimator combines nonparametric pooling with a data-adaptive shrinkage matrix \(Q\), so both the pilot estimator and the penalty are random and depend on the full sample.
Consequently, the estimator is not a simple linear functional of the data, and standard asymptotic arguments do not apply directly. We show that (Theorem \ref{thm:Dconv-LR-TF-0-infty} and the corollaries follow), under appropriate undersmoothing conditions, the adaptive estimator nevertheless admits a centered asymptotically normal distribution with a feasible variance estimator: 
$$
\hat \Sigma_{\TL}^{-1/2} (\hat \btheta_{\TL} - \btheta(U_0)) \overset{\mathscr{L}}{\implies} \cN(0, \bI_p)
$$
where $\hat \Sigma_{\TL}$ is an estimator of the variance of $\hat \btheta_{\TL}$, which depends on $(\beta, \gamma, K, n, n_0)$ (precisely quantified in Section \ref{sec:linear-infer}). This enables confidence intervals and Wald-type tests for \(\theta(u_0)\) that properly account for the data-adaptive pooling and shrinkage mechanism.
We summarize our contribution below:    
\begin{enumerate}
\item {\bf Methodological contribution: }We propose a domain varying coefficient model, bridging standard GLM and VCM that relates source and target domains via observable domain identifiers. We develop a computationally efficient methodology to construct a minimax optimal estimator $\hat \btheta_{\TL}$, which is provably safe (no negative transfer), and adaptively borrows strength from related source domains.

\item {\bf Theoretical contribution: }On theoretical front, we rigorously establish that $\hat \btheta_{\TL}$ is minimax rate optimal. Furthermore, we also establish the asymptotic normality of $\hat \btheta_{\TL}$, which aids in inference and constructing an asymptotically valid confidence interval. Details can be found in Section \ref{sec:theory}.

\item {\bf Application: } Last but not least, we demonstrate the efficiency of our estimator through extensive simulation, as well as on two real socio-economic datasets: i) US adult-income dataset, and ii) SLID-Ontario dataset (details can be found in Sections \ref{sec:sim} and \ref{sec:real-data}). 
\end{enumerate}
The organization of this paper is as follows: we conclude the Introduction section with a brief discussion of the related literature and introduce the notations used throughout the rest of the paper. In Section \ref{sec:method}, we present our two-step methodology for constructing $\hat \btheta_{\TL}$. In Section \ref{sec:theory}, we establish theoretical guarantees of our proposed estimator. In Section \ref{sec:sim}, we present extensive numerical experiments demonstrating the efficacy of our methodology. In Section \ref{sec:real-data}, we apply our method to two real datasets. Finally, we conclude by outlining promising directions for future research in Section \ref{sec:conclusion}.\footnote{Code implementing our methodology and reproducing all experiments is available at \url{github.com/hanxiao-chen/Transfer_Learning_DVCM}}.
\\\\
\noindent
{\bf Positioning in the existing literature: }Transfer learning with parametric regression has gained significant attention recently. For a single source domain, \cite{chen2015data} proposes data-enriched linear regression with a penalized difference between source and target coefficients, while \cite{obst2022improved} analyzes a fine-tuning approach with a significance test for positive transfer. With multiple sources, selecting informative domains is the key. \cite{li2022transfer} proposes a minimax-efficient strategy for high-dimensional linear regression, followed by de-biasing with target data, extended by \cite{tian2022transfer} and \cite{li2023transfer} to high-dimensional generalized linear models. \cite{chen2024transfer} incorporate dependence among observations, and \cite{zhao2023residual} introduce an importance-weighted method using residuals for transfer learning. Beyond parametric settings, nonparametric transfer learning has also seen advancements in both classification \citep{cai2021transferC,reeve2021adaptive,Maity2023A,fan2023robust,scott2019generalized} and regression \citep{cai2022transferR,wang2023minimax}, as well as other extensions \citep{cai2024transferfunctionalmean, cai2024transferbandits, auddy2024minimax}. Recent extensions further expand the scope to reinforcement learning, functional data analysis, matrix estimation, outlier detection, heavy-tailed data, and bootstrap \citep{chen2025data, chen2025transfer, qin2025adaptive, zhao2025trans, jalan2025optimal, kalan2025transfer, yan2026transfer, shang2025bootstrap,chai2025deep,chai2025transition}. Notably, \cite{bussas2017varying} proposed a Bayesian varying coefficient model with Gaussian process priors for geospatial transfer learning, though they did not provide convergence rates. To our knowledge, this is the only work that employs a VCM for transfer learning. Our approach differs by explicitly specifying the functional class of $\btheta(\cdot)$, which enables the derivation of matching minimax lower and upper bounds. 

It is worth clarifying how our framework differs from several existing transfer-learning paradigms. First, high-dimensional parametric transfer methods typically assume a fixed coefficient vector across domains and exploit sparsity or shared support structure (e.g., Lasso-based transfer), focusing on variable selection and parameter shrinkage. In contrast, we explicitly model domain heterogeneity through a smooth coefficient function $\theta(U)$, allowing systematic variation across environments rather than enforcing invariance. 
Second, many modern approaches rely on representation alignment or feature adaptation, seeking a domain-invariant representation of $X$ through deep architectures. While powerful in practice, such methods are often algorithmic and do not yield transparent statistical characterizations of the bias--variance tradeoff under domain drift. Our DVCM framework instead imposes a structured, interpretable restriction: cross-domain similarity is encoded through smoothness in the domain index, leading to precise minimax characterizations and adaptive guarantees. Thus, rather than performing penalized fine-tuning in an abstract parameter space, our approach leverages an explicit geometric structure on domains, which enables both negative-transfer control and rigorous inference.
\\\\
\noindent
{\bf Notations.} 
For a vector $\bx = [x_1,\ldots,x_p]^\top$, define $\|\bx\|_q = (\sum_{i=1}^p |x_i|^q)^{1/q}$, $\bx^{\otimes 2} = \bx \bx^\top$, and $\|\bx\|_A^2 = \bx^\top A \bx$, where $A$ is positive semidefinite. The spectral norm of $A$ is $\|A\|_2$. 
For $\bx \in \bbR^p$ and $\by \in \bbR^q$, $\bx \otimes \by = [x_1\by^\top, \ldots, x_p\by^\top]^\top$.  
Let $\hat\btheta$ estimate $\btheta$; define $\mse_A(\hat\btheta)=\bbE\|\hat\btheta-\btheta\|_A^2$ and $M(\hat\btheta)=\bbE[(\hat\btheta-\btheta)^{\otimes2}]$.  
Denote by $\be_{i,j}$ the length-$j$ vector with $1$ in the $i$th position.  
For matrices $A,B$, $A\succeq B$ means $A-B$ is positive semidefinite; $\lambda_{\min}(A)$ and $\lambda_{\max}(A)$ are its smallest and largest eigenvalues.  
For sequences $a_n,b_n>0$, write $a_n\lesssim b_n$ if $a_n\le Cb_n$ for some $C>0$, and $a_n\asymp b_n$ if both $a_n\lesssim b_n$ and $a_n\gtrsim b_n$.  
Also, $a_n=o(b_n)$ means $|a_n/b_n|\to0$, and $a_n=O(b_n)$ means $\sup_n|a_n/b_n|<\infty$.    
For real numbers $x,y$, let $x\wedge y=\min(x,y)$ and $x\vee y=\max(x,y)$; for integer $n$, $[n]=\{1,\ldots,n\}$.  
For the $k$th domain, define the index set $\cI_k=[n_k]$ and the pooled source data $\cD_S=\bigcup_{k=1}^K\cD_k$.  
Let $\Gamma=\{U_k:0\le k\le K\}\cup\{X_{ki}:0\le k\le K, i\in\cI_k\}$ denote all covariates.  
We use $\bA_l=[I_p, \mathbf{0}_{p\times lp}]$ when defining nonparametric estimators, where $I_p$ is the identity matrix and $\mathbf{0}_{p\times lp}$ is a zero matrix.

\vspace{-2mm}
\section{Methodology}
\label{sec:method}
In this section, we present our methodology for estimating $\btheta(U_0)$, the coefficient on the target domain. 
Our proposed methodologies for the linear DVCM and generalized linear DVCM are presented in Sections \ref{subsec:TL_linear} and \ref{subsec:method-TL-GLM}, respectively.  

\subsection{Transfer learning for linear DVCM}
\label{subsec:TL_linear} 
Recall that we use the index $0$ for the target domain and $1 \le k \le K$ for the source domains. The observed data from $k$-th domain is denoted by $\{(X_{ki}, Y_{ki})\}_{i \in \cI_k}$ and the domain identifier is $U_k$ for $k \in \{0, 1, \dots, K\}$. 
As per our data-generating model, Equation \eqref{eq:dgp_dvcm}, the observed sample from $k^{th}$ domain, under the linearity assumption, is assumed to follow 
\begin{equation*}
    Y_{ki} = X_{ki}^\top\btheta(U_k)  + \varepsilon_{ki}
\end{equation*}
where the noise $\varepsilon_{ki}$'s are independent within and across the domains. Our parameter of interest is $\btheta(U_0) = \btheta(u_0)$ ($u_0$ being the realization of $U_0$), the coefficient of the target domain. A simple estimator for $\btheta(u_0)$ is the ordinary least squares estimator computed using only the target-domain data $\cD_0$:
\vspace{-2mm}
\begin{equation}
\textstyle
\label{eqn:est-GLR-linear}
\hat\btheta_{\LR}(u_0) 
= \argmin_{\balpha} \ \sum_{i\in\cI_0} \left(Y_{0i}- X_{0i}^\top \balpha\right)^2 
= \left(\bX_0^\top \bX_0\right)^{-1}\bX_0^\top \by_0 \,.
\end{equation}
Although this target-only estimator is rate-optimal (and efficient under Gaussian errors), it ignores potentially useful information from the source domains. Nevertheless, it serves as our \emph{target-only baseline}, and any transfer-learning-based estimator should not underperform relative to this estimator. We now describe our adaptive transfer learning procedure, which is summarized in Algorithm \ref{alg:TL-LVCM}.
\\\\
\noindent
\underline{\bf Step I: Nonparametric initialization.}
To borrow information from the source domains, we first construct a nonparametric point estimator of $\btheta(u_0)$, denoted by $\hat \btheta_{\DVCM}(u_0)$, using all available data via local polynomial regression. Specifically, let $W$ be a smoothing kernel (typically a symmetric probability density function; see Section~\ref{sec:theory} for precise assumptions) with bandwidth parameter $h>0$. The estimator $\hat\btheta_{\DVCM}(u_0)$ is defined as: 
\begin{equation}
\textstyle
\label{eqn:est-DVCM-linear}
\hat\btheta_{\DVCM}(u_0)= \bA_l \cdot \argmin_{\balpha\in \bbR^{(l+1)p}} \sum_{k=0}^K \sum_{i\in \cI_k}\left[Y_{ki}-\left(\bPhi_l\!\left(\frac{U_k-u_0}{h}\right)^\top \otimes X_{ki}^\top\right)\balpha\right]^2W\!\left(\frac{U_k-u_0}{h}\right),
\end{equation}
where $\bPhi_l(x) = \left(1, x, x^2/2!, \ldots, x^l/l!\right)^\top$ is the $l$-th order polynomial feature map, $\otimes$ denotes the Kronecker product, and $\bA_l = [I_{p}, \mathbf 0_{p\times lp}]$ selects the first $p$ coordinates of the minimizer. This estimator admits the closed-form expression
\begin{equation}
\label{eq:DVCM-closed-form}
\hat\btheta_{\DVCM}(u_0)=\bA_l(\bZ^\top \bW \bZ)^{-1}\bZ^\top \bW \by ,
\end{equation}
where
\begin{equation}
\label{eq:def_Z_W}
\textstyle
\bZ_{ki} = \bPhi_l\!\left(\frac{U_k-u_0}{h}\right)\otimes X_{ki}, \qquad \bW=S_h^{-1}\diag\!\left\{W\!\left(\tfrac{U_k-u_0}{h}\right)\right\}_{k,i}, \qquad S_h = \sum_{k,i}W\!\left(\frac{U_k-u_0}{h}\right).
\end{equation}
The key advantage of $\hat\btheta_{\DVCM}(u_0)$ is that it aggregates source information using local weights that are larger for domains with identifiers $U_k$ close to $u_0$ (i.e., the weight decreases as $|U_k-u_0|$ grows). Consequently, it effectively borrows information from relevant source domains. As will be shown in Section~\ref{sec:theory} (Equation \eqref{eqn:optimal_bw}), the optimal bandwidth $h$ automatically balances two factors: (i) the smoothness of $\btheta(\cdot)$, and (ii) the proximity and spread of the $U_k$'s. 
If $\btheta(\cdot)$ is sufficiently smooth or the $U_k$'s are clustered near $u_0$, then a substantial amount of information is borrowed from the sources.
\\\\
\noindent
\underline{\bf Step II: Fine-tuning.} Although $\hat\btheta_{\DVCM}(u_0)$ borrows information adaptively, it may perform worse than the baseline $\hat\btheta_{\LR}(u_0)$ when $\btheta(\cdot)$ is not sufficiently smooth and the $U_k$'s are far from $u_0$, or when the bandwidth $h$ is misspecified. Such situations can lead to negative transfer, as illustrated in our numerical studies. To guard against this phenomenon, we fine-tune $\hat\btheta_{\DVCM}(u_0)$ using the target data through a ridge-regularized regression: 
\allowdisplaybreaks
\begin{align}
\label{eqn:TL-linear}
\hat\btheta_{\TL}(u_0)
&=
\argmin_{\balpha\in\bbR^p}
\frac{1}{2n_0}
\sum_{i\in \cI_0}
\left(
Y_{0i}-X_{0i}^\top \balpha
\right)^2
+
\frac{1}{2}
\|\balpha-\hat\btheta_{\DVCM}(u_0)\|_Q^2
\notag \\
&=
\left(
\frac{1}{n_0}\bX_0^\top \bX_0 + Q
\right)^{-1}
\left(
\frac{1}{n_0}\bX_0^\top \by_0
+
Q\,\hat\btheta_{\DVCM}(u_0)
\right),
\end{align}
where $Q$ is a symmetric positive definite matrix. The choice of $Q$ is crucial: only a proper selection ensures adaptivity and protection against negative transfer. One oracle choice is 
\begin{equation}
\textstyle
\label{eqn:Q-linear}
Q
=
\delta
\frac{\sigma^2(u_0)}{n_0}
M\!\left(\hat\btheta_{\DVCM}(u_0)\right)^{-1},
\end{equation}
for some constant $\delta \in (1/2,2)$ (see Theorem \ref{thm:rate-TL}). Here $\sigma^2(u_0)$ denotes the noise variance in the target domain, and $n_0$ is the sample size of the target domain. We assume that the mean squared error matrix $M(\hat\btheta_{\DVCM}(u_0))$ is invertible, which holds whenever the covariance matrix of $\hat\btheta_{\DVCM}(u_0)$ is positive definite. This choice of $Q$ can be interpreted as the ratio of uncertainties between two estimators: $\sigma^2(u_0)/n_0$ is proportional to the variance of the target-only estimator $\hat\btheta_{\LR}(u_0)$, whereas $M(\hat\btheta_{\DVCM}(u_0))$ represents the MSE of $\hat\btheta_{\DVCM}(u_0)$. With this construction, the second step interpolates between the target-only and pooled estimators: when the DVCM pilot is relatively precise, strong shrinkage occurs; when it is noisy, the estimator automatically reverts toward the target-only solution. Next, we describe a data-driven way to obtain $\hat Q$, which can be used in Equation \eqref{eqn:TL-linear} to compute $\hat \btheta_{\rm TL}(u_0)$. 
\\\\
\noindent
\fbox{\textbf{Estimation of $Q$:}} We now present a fully data-driven choice of $Q$, which also provably yields an optimal estimator $\hat\btheta_{\TL}(u_0)$, as will be established in Section \ref{sec:theory}. Our construction closely mimics the oracle choice in Equation~\eqref{eqn:Q-linear}, where the unknown MSE of $\hat\btheta_{\DVCM}(u_0)$ and $\hat\btheta_{\LR}(u_0)$ are replaced by consistent estimators. To this end, recall that the mean squared error matrix of $\hat\btheta_{\DVCM}(u_0)$ can be decomposed into bias and variance components:
\begin{equation}
\textstyle
\label{eq:bias_var_Q}
M(\hat\btheta_{\DVCM}(u_0)) = \underbrace{(\bbE[\hat\btheta_{\DVCM}(u_0)] - \btheta(u_0))(\bbE[\hat\btheta_{\DVCM}(u_0)] - \btheta(u_0))^\top}_{:=\Bias^{\otimes 2}(\hat\btheta_{\DVCM}(u_0))} + \Var\left(\hat\btheta_{\DVCM}(u_0)\right) \,.
\end{equation}
We estimate these two components separately using the procedures described in Section~\ref{subsec:est-Q}. Specifically, we employ a plug-in estimator for the bias term and a sandwich-type estimator for the variance term. Furthermore, we estimate $\sigma^2(u_0)$ using the sample mean of the squared residuals. Combining these estimators yields the following data-driven penalty matrix $\hat Q$:
\begin{equation}
\textstyle
\label{eqn:Q-linear-est}
    \hat Q = \delta\frac{\hat\sigma^2(u_0)}{n_0}\bigg\{ \widehat{\Bias}^{\otimes 2}\left(\hat\btheta_{\DVCM}(u_0)\right) + \widehat{\Var}\left(\hat\btheta_{\DVCM}(u_0)\right) \bigg\}^{-1}\,.
\end{equation}
Finally, the adaptive transfer-learning estimator $\hat\btheta_{\TL}(u_0)$ is obtained by substituting the above data-driven matrix $\hat Q$ into Equation~\eqref{eqn:TL-linear}.
Our entire procedure is summarized in Algorithm \ref{alg:TL-LVCM}. In the next subsection, we extend our algorithm to a generalized DVCM model. 
\begin{algorithm}[t]
\caption{Transfer Learning for Linear DVCM}
\label{alg:TL-LVCM}
\begin{algorithmic}[1]
\REQUIRE data $\{(U_k,\{X_{ki},Y_{ki}\}_{i\in\mathcal I_k})\}_{k=0}^K$, order $l$, bandwidth $h$
\STATE 1) Split the target-domain sample into two halves of equal size, indexed by $\cI_0$ and $\cI_0^*$
\STATE 2) Compute $\hat{\btheta}_{\DVCM}(u_0)$ as in \eqref{eqn:est-DVCM-linear}
\STATE 3) Estimate $\hat Q$ via Section~\ref{subsec:est-Q}
\STATE 4) Form $\hat{\btheta}_{\mathrm{TL}}(u_0)$ using \eqref{eqn:TL-linear} with $\hat Q$
\ENSURE $\hat{\btheta}_{\mathrm{TL}}(u_0)$
\end{algorithmic}
\end{algorithm}

\subsection{Transfer learning for generalized linear DVCM}
\label{subsec:method-TL-GLM}
In this section, we extend our methodology to the setting where the conditional distribution of $Y$ given $(X,U)$ follows a generalized linear model with a general link function $g$ (cf. Equation~\eqref{eq:dgp_dvcm}). More specifically, we assume that the conditional distribution belongs to a canonical exponential family:
\begin{equation}
    \textstyle
\label{eqn:dgp_glm}
    f(Y_{ki}=y\mid X_{ki}=\bx,U_k=u) = \exp\left\{\nu(u)^{-1}\left[y\bx^\top \btheta(u)-b\left(\bx^\top \btheta(u)\right)\right]+c\left(y,\nu(u)\right)\right\}\,,
\end{equation}
where $\nu(u)$ is a scale parameter. This formulation implies
$$
\bbE[Y_{ki} \mid X_{ki}, U_k] = b'(X_{ki}^\top \btheta(U_k)) \implies g\left(\bbE[Y_{ki} \mid X_{ki}, U_k]\right) = X_{ki}^\top \btheta(U_k) \,,
$$
where $g(\mu) = (b')^{-1}(\mu)$ is the canonical link function. The key ideas parallel those in Section~\ref{subsec:TL_linear}, with appropriate modifications to accommodate a general link. As before, we begin with a \emph{target-only} estimator computed using only target data, which serves as the no-transfer baseline. However, instead of minimizing a squared-error loss, we minimize the negative log-likelihood:
\begin{equation}
\textstyle
\label{eq:def-GLR}
    \hat\btheta_{\GLR}(u_0) =\argmin_{\balpha} \sum_{i \in {\cI_0}}\left\{ b(X^\top \balpha)-YX^\top \balpha\right\} :=  \argmin_{\balpha}\sum_{i \in {\cI_0}} \ell(X_{0i}^\top \alpha, Y_{0i})\,.
\end{equation}
By classical GLM asymptotics \citep{mccullagh2019generalized}, the target-only MLE $\hat\btheta_{\GLR}(u_0)$ is $\sqrt{n_0}$-consistent and asymptotically normal under standard regularity conditions. Nevertheless, as in the linear case, it ignores potentially informative source data. To exploit cross-domain similarity, we again proceed in two steps: 
\\\\
\noindent
\underline{\bf Step I: Nonparametric initialization.} We first construct a nonparametric estimator of $\btheta(u_0)$ using local polynomial regression, pooling all observations together:
\begin{equation}
\textstyle
    \label{eqn:VCM-GLM}
    \hat\btheta_{\GDVCM}(u_0)= \bA_{l} \cdot   \argmin_{\balpha\in \bbR^{(l+1)p}} \sum_{k = 0}^K \sum_{i \in \cI_k} \ell(Z_{ki}^\top \alpha, Y_{ki})W\left(\frac{U_k-u_0}{h}\right) \,,
\end{equation}
where $Z_{ki} = \bPhi_l\left(\tfrac{U_k-u_0}{h}\right)\otimes X_{ki}$. This estimator mirrors \eqref{eqn:est-DVCM-linear}, with the squared-error loss replaced by the negative log-likelihood.
\\\\
\noindent
\underline{\bf Step II: Fine-tuning.} As for the generalized linear case, in this step we refine the pilot $\hat\btheta_{\GDVCM}(u_0)$ using the target data via a ridge-regularized GLM:
\begin{equation}
\textstyle
\label{eqn:TL-GLM}
    \hat \btheta_{\TL}(u_0) = \argmin_{\balpha} \frac{1}{n_0} \sum_{i \in {\cI_0}}\ell(X_{0i}^\top \alpha, Y_{0i}) + \frac{1}{2} \|\balpha - \hat\btheta_{\GDVCM}(u_0)\|_Q^2 \,.
\end{equation}
In Section~\ref{sec:theory}, we show that the following oracle choice of $Q$ makes $\hat\btheta_{\TL}(u_0)$ adaptive and immune to negative transfer:
\begin{equation}
\textstyle
\label{eqn:Q-est-GLM}
        Q =\delta \frac{\nu(u_0)}{n_0}M\left(\hat\btheta_{\GDVCM}(u_0)\right)^{-1}
\end{equation}
for $\delta \in (1/2, 2)$.
This choice parallels \eqref{eqn:Q-linear}, with the Gaussian noise variance $\sigma^2(u_0)$ replaced by the GLM scale parameter $\nu(u_0)$. As in Section~\ref{subsec:TL_linear}, the matrix $Q$ captures the ratio of uncertainties between the target-only estimator $\hat\btheta_{\GLR}(u_0)$ and the pilot $\hat\btheta_{\GDVCM}(u_0)$. Consequently, shrinkage is strong when $\hat\btheta_{\GDVCM}(u_0)$ is relatively more precise, and it relaxes toward the target-only estimator $\hat\btheta_{\GLR}(u_0)$ otherwise. 
The complete procedure is summarized in Algorithm~\ref{alg:TL-GVCM}, and all theoretical guarantees are established in Section~\ref{sec:theory}.
In the next section, we develop a fully data-driven estimator $\hat Q$ and show that, when substituted into \eqref{eqn:TL-GLM}, the resulting estimator $\hat\btheta_{\TL}(u_0)$ achieves minimax-optimal rates while remaining adaptive to potential negative transfer. 
\noindent
\begin{remark}
\label{rem:data_splitting}
In the procedural description above, we used the entire dataset in Step~I to construct the nonparametric estimator $\hat\btheta_{\DVCM}(u_0)$ and then reused the target-domain data in Step~II to compute the fine-tuned estimator $\hat\btheta_{\TL}(u_0)$. Consequently, the target data are involved in both steps, which induces statistical dependence between $\hat\btheta_{\DVCM}(u_0)$ and the second-stage objective used to define $\hat\btheta_{\TL}(u_0)$.
While this dependence has a negligible impact on empirical performance, it complicates the theoretical analysis. To simplify the exposition and proofs in Section~\ref{sec:theory}, we therefore adopt a data-splitting scheme. In particular, we assume that $2n_0$ target-domain observations are available: the first $n_0$ samples are used in Step~I to construct $\hat\btheta_{\DVCM}(u_0)$, and the remaining $n_0$ samples are reserved for Step~II to compute $\hat\btheta_{\TL}(u_0)$. 
Under this scheme, $\hat\btheta_{\DVCM}(u_0)$ is independent of the second half of the target data used in the fine-tuning step, which substantially simplifies the theoretical arguments. Although it may be possible to establish the asymptotic properties of $\hat\btheta_{\TL}(u_0)$ without data-splitting, we do not pursue that direction in this paper.
\end{remark}

\begin{algorithm}[t]
\caption{Transfer Learning for Generalized DVCM}
\label{alg:TL-GVCM}
\begin{algorithmic}[1]
\REQUIRE data $\{(U_k, \{X_{ki}, Y_{ki}\}_{i \in \cI_k})\}_{k=0,\ldots,K}$, order $l$, bandwidth $h$, loss function $\ell$
\STATE 1) Randomly split the target-domain data into two equal parts, indexed by $\cI_0$ and $\cI_0^*$
\STATE 2) Compute $\hat\btheta_{\GDVCM}(u_0)$ as in Equation \eqref{eqn:VCM-GLM}
\STATE 3) Estimate $\hat Q$ using the procedure in Section~\ref{subsec:est-Q}
\STATE 4) Construct the transfer learning estimator $\hat\btheta_{\TL}(u_0)$ using Equation \eqref{eqn:TL-GLM} with $\hat Q$
\ENSURE $\hat\btheta_{\TL}(u_0)$
\end{algorithmic}
\end{algorithm}

\subsection{Estimating $Q$}
\label{subsec:est-Q}
In this subsection, we propose a consistent estimator $\hat Q$ of the oracle penalty matrix $Q$. We formulate the procedure in the generalized linear model setting, since the linear model is recovered as a special case by taking $\ell(\eta,y) = (\eta-y)^2/2$. Let $s_j(\eta,y) = \frac{\partial^j}{\partial \eta^j}\ell(\eta,y)$ denote the $j$-th derivative of the loss function with respect to the linear predictor $\eta$. Using the bias–variance decomposition in Equation~\eqref{eq:bias_var_Q}, we define
\begin{equation}
\textstyle
\label{eqn:Q-hat-in}
\hat Q = \delta \frac{\hat \nu(u_0)}{n_0}\,\hat M\left(\hat \btheta_{\GDVCM}(u_0)\right)^{-1}
= \delta \frac{\hat\nu(u_0)}{n_0}\Big\{ \widehat{\Bias}\left(\hat\btheta_{\GDVCM}(u_0)\right)^{\otimes 2} 
+ \widehat{\Var}\left(\hat\btheta_{\GDVCM}(u_0)\right) \Big\}^{-1}.
\end{equation}
We now describe how to estimate each component.

\medskip
\noindent
\underline{\bf Estimation of $\nu(u_0)$.}
The scale parameter $\nu(u_0)$ can be estimated using standard GLM methodology. A common approach is based on Pearson residuals:
\begin{equation}
\textstyle
\label{eqn:nu-est}
r_i^2 = 
\left(\frac{Y_{0i} - g^{-1}(X_{0i}^\top \hat\btheta_{\GLR}(u_0))}
{\sqrt{ b''(X_{0i}^\top \hat\btheta_{\GLR}(u_0))}}\right)^2,
\qquad 
\hat \nu(u_0) = \frac{1}{n_0}\sum_{i=1}^{n_0} r_i^2.
\end{equation}
Under standard regularity conditions, $\hat\nu(u_0)$ is a consistent estimator of $\nu(u_0)$ \citep{mccullagh2019generalized}.

\medskip
\noindent
\underline{\bf Estimation of the bias term.}
The leading bias of $\hat\btheta_{\GDVCM}(u_0)$ is of order $h^\beta$. We therefore estimate
$\widehat{\Bias}(\hat\btheta_{\GDVCM}(u_0)) = c_0 h^\beta$,
where $c_0>0$ is determined by the local polynomial approximation. 
For integer $\beta$, the plug-in bias estimator is
\begin{equation}\label{eqn:bias_glm}
\textstyle
\widehat{\Bias}\!\left(\hat\btheta_{\GDVCM}(u_0)\right)
=
\big[\hat\zeta_{0,1}^{-1}\hat\zeta_{\beta,1}\big]_{1,1}
\frac{1}{\beta!}\,
\hat\btheta^{(\beta)}(u_0)\,h^{\beta},
\end{equation}
where $\hat\zeta_{r,s}=(nh)^{-1}\sum_{k=0}^K n_k t_k^r W(t_k)^s$ and $t_k = (U_k-u_0)/h$.
A consistent estimator of $\btheta^{(\beta)}(u_0)$, namely, $\hat\btheta^{(\beta)}(u_0)$, can be obtained via another application of local polynomial regression \citep{ruppert1994multivariate}, i.e.
$$
\hat\btheta^{(\beta)}(u_0)=\tilde\bA_{\beta} \hat \balpha_{\rm DVCM},
\quad
\tilde\bA_{\beta}=h^{-\beta}[\mathbf{0}_{p\times \beta p},I_p],
$$
where $\hat \balpha_{\rm DVCM}$ minimizes the objective of \eqref{eqn:VCM-GLM}. The consistency of the resulting estimator can be obtained from classical kernel regression theory \citep{nadaraya1964estimating, gasser1984estimating}.

\medskip
\noindent
\underline{\bf Estimation of the variance term.}
Let $t_k = (U_k-u_0)/h$. Following \cite{fan1998local}, we estimate the variance component using a sandwich-type estimator:
\begin{equation}
\textstyle
\label{eqn:var-est-VCMS-in}
\widehat{\Var}\left(\hat\btheta_{\GDVCM}(u_0)\right) 
= \bA_{l}\, \hat\Lambda^{-1} \hat\Delta \hat\Lambda^{-1} \bA_{l}^\top,
\end{equation}
where
\begin{align*}
\hat\Delta &= \frac{1}{(nh)^2}\sum_{k=0}^K\sum_{i\in \cI_k} 
s_1^2\left(Z_{ki}^\top \hat\balpha_{\GDVCM},Y_{ki}\right) 
Z_{ki}Z_{ki}^\top W(t_k)^2,\\
\hat\Lambda &= \frac{1}{nh}\sum_{k=0}^K\sum_{i\in \cI_k}
s_2\left(Z_{ki}^\top \hat\balpha_{\GDVCM},Y_{ki}\right) 
Z_{ki}Z_{ki}^\top W(t_k).
\end{align*}
Here, $\hat\balpha_{\GDVCM}$ denotes the minimizer of the objective function in Equation~\eqref{eqn:VCM-GLM}. By a standard application of the weak law of large numbers, the estimator $\widehat{\Var}$ consistently estimates the variance component of $M(\hat\btheta_{\GDVCM}(u_0))$.

\medskip
Combining the above components, $\hat Q$ consistently mimics the oracle choices in Equations~\eqref{eqn:Q-linear} for linear DVCM and \eqref{eqn:Q-est-GLM} for generalized linear DVCM. Consequently, when substituted into the fine-tuning step, it yields a provably optimal adaptive estimator, as established in the next section.

\vspace{-3mm}
\section{Theoretical Analysis}
\label{sec:theory}
In this section, we establish the theoretical properties of the estimator $\hat{\btheta}_{\TL}(u_0)$ under both linear and generalized linear model settings. Our analysis characterizes the minimax-optimal rates of estimation and derives the limiting distributions necessary for statistical inference.
Section~\ref{sec:linear} develops a non-asymptotic theory for $\hat\btheta_{\TL}(u_0)$ under the linear model assumption (Equation~\eqref{eqn:TL-linear}). We show that the proposed estimator attains the minimax-optimal rate of convergence. Based on these results, Section~\ref{sec:linear-infer} establishes asymptotic normality and presents valid inference procedures.
Finally, Section~\ref{sec:G-linear} extends the analysis to the generalized linear model framework (Equation~\eqref{eqn:TL-GLM}), demonstrating that the minimax optimality and inferential guarantees continue to hold in this more general setting.

\subsection{Non-Asymptotic Results for Linear DVCMs}
\label{sec:linear}
Before presenting our main results, we briefly recall the notation and data-generating assumptions. For each domain $k \in \{0,1,\dots,K\}$, the observations follow the linear model $Y_{ki} = X_{ki}^\top \btheta(U_k)+\varepsilon_{ki}$, where $\varepsilon_{ki}$ is independent of $X_{ki}$ conditional on $U_k$, with $\bbE[\varepsilon_{ki}\mid U_k]=0$, and 
$\Var(\varepsilon_{ki}\mid U_k)=\sigma^2(U_k)$. Thus, the noise variance is allowed to vary across domains through its dependence on $U_k$. For any $u\in\cU$, we define $d_k(u) =\|u-U_k\|_2$, which measures the distance between the $k$-th domain identifier $U_k$ and $u$. In particular, at $u=u_0$, the collection $\{d_k(u_0)\}_{1\le k\le K}$ quantifies the similarity between each source domain and the target domain. Let $\{d_{(k)}(u_0)\}_{1\le k\le K}$ denote the corresponding order statistics. These ordered distances induce a ranking of the source domains according to their proximity to the target domain. We now state the assumptions required for our theoretical analysis.

\begin{assumption}[Functional coefficient]
\label{Assump:VCM-SN}
    The coefficient function $\btheta(u) = \left(\theta_1(u),\ldots,\theta_p(u)\right)^\top$ is a collection of $p$ functions that belong to a H\"older class $\mathcal{H}(\beta,L)$ with $\beta,L > 0$.
\end{assumption}

\begin{assumption}[Assumptions on the data-distribution]
    \label{Assump:VCM-UX}
    The data distribution $(U, X, Y)$ is assumed to satisfy the following: 
    \begin{enumerate}
    \item[(a)] The variables $U_k$ are i.i.d. from a location-scale family: 
    \[
    f_U(u) = \frac{1}{\gamma} f\left(\frac{u - u^*}{\gamma}\right),
    \]
    where $f$ is compactly supported on $\cB \subset \bbR$ and satisfies $a_0' \le f(u) \le a_0$ for all $u \in \cB$. 
    The constants $u^* \in \cB$ and $0 < \gamma < \infty$ are the location and scale parameters. Furthermore, the covariate $X_{ki}$ is assumed to be compactly supported; we assume $\Vert X_{ki} \Vert_2 \leq 1$ almost surely without loss of generality. 

    \item[(b)] The conditional fourth moment $\bbE[|Y|^4 \mid U = u]$ is assumed to be continuous, and consequently uniformly upper bounded on the support of $U$ (as it is compact). 

    \item[(c)] For $h \geq d_{(1)}(u_0)$, it holds almost surely that
    \[
    \big\|\bA_l(\bZ^\top \bW \bZ)^{-1}\bA_l^\top\big\|_2 \le \lambda_0^{-1},
    \]
    for some constant $\lambda_0 > 0$, where $(\bA_l, \bZ, \bW)$ is same as defined in Equation \eqref{eq:def_Z_W}.
    \end{enumerate}
\end{assumption}

\begin{assumption}[Balanced domain sizes]\label{Assump:balance-sample}
The source samples are assumed to be balanced, i.e., there exist positive constants $b_0', b_0$ such that $n_k/\bar n \in [b_0', b_0]$, for $k \in \{1, \dots, K\}$ and $\bar n = (\sum_{k=0}^K n_k)/(K+1)$ denotes the average sample size. For the target domain, we only assume an upper bound, i.e., $n_0/\bar n \le b_0$. 
\end{assumption}

\begin{assumption}[Uniform kernel]\label{Assump:Unif-Kernel}
The kernel function $W$ is a uniform pdf $W(u) = \frac{1}{2} \indc\{|u|\leq 1\}$.
\end{assumption}

\noindent{\bf Discussions on the assumptions: }
Assumption~\ref{Assump:VCM-SN} is standard in the nonparametric estimation literature and imposes smoothness conditions on the coefficient function $\btheta(\cdot)$. 

Assumption~\ref{Assump:VCM-UX}(a) models the domain indices $\{U_k\}$ as i.i.d.\ draws from a well-behaved location–scale family with compactly supported density $f$. The scale parameter $\gamma$ controls the dispersion of the domain indices: smaller values of $\gamma$ correspond to closely related domains (in which case transfer learning is beneficial), whereas larger values of $\gamma$ reflect more heterogeneous and widely dispersed domains. As established in Theorem~\ref{thm:rate-TL}, the minimax-optimal estimation rate is governed by the interplay between $(\gamma,\beta)$. Smaller $\gamma$ and/or larger $\beta$ correspond to more informative source domains and hence faster convergence rates. Conversely, larger $\gamma$ and smaller $\beta$ indicate weaker cross-domain similarity, in which case the performance of our estimator approaches that of the target-only estimator. The compactness assumption on the support of $X$ is made for technical convenience and can be relaxed using standard truncation arguments. Assumption~\ref{Assump:VCM-UX}(b) ensures that the error distribution is well behaved. In contrast to much of the existing literature, which assumes sub-Gaussian errors, we only require bounded fourth moments. Assumption~\ref{Assump:VCM-UX}(c) guarantees the well-posedness of the linear system in Equation~\eqref{eq:DVCM-closed-form}, thereby ensuring the existence of the estimator. A closely related assumption appears in Section~1.6.1 of \cite{tsybakov2008introduction}. 
In Appendix~\ref{sec:eigen-assum}, we present a general result showing that under mild conditions the assumption holds almost surely.  

Assumption~\ref{Assump:balance-sample} requires that the sample sizes across source domains grow at comparable rates, preventing any single source domain from becoming asymptotically negligible or overly dominant. If this were violated, the analysis could be restricted to domains with asymptotically non-negligible sample proportions. In contrast, \textit{we impose no lower bound} on the ratio $n_0/\bar n$ for the target domain and allow $n_0/\bar n \to 0$, thereby accommodating practically relevant scenarios in which the target sample size is of smaller order than the aggregate source sample sizes. 

Finally, although our theoretical development is presented under the uniform kernel assumption (Assumption~\ref{Assump:Unif-Kernel}), the analysis extends to any kernel $W$ that is compactly supported and bounded away from $0$ and $\infty$ on its support.
\\\\
\noindent
Our first result characterizes the rate of convergence of the target-only baseline $\hat\btheta_{\LR}(u_0)$ (defined in Equation~\eqref{eqn:est-GLR-linear}) and the nonparametric DVCM estimator $\hat\btheta_{\DVCM}(u_0)$ (defined in Equation~\eqref{eq:DVCM-closed-form}). By default, the order of the polynomial in constructing $\hat \btheta_{\rm DVCM}(u_0)$ is chosen as $l = \lfloor \beta \rfloor$ throughout this section.

\begin{proposition}
\label{prop:rate-LR-VCM}
Let $A$ be any matrix satisfying $0 \preceq A \preceq C I$ for some constant $C>0$, and define $\mse_A(\hat\btheta)= \bbE \|\hat\btheta - \btheta\|_A^2$, where $\|x\|_A^2 = x^\top A x$. Under Assumption~\ref{Assump:VCM-UX}, the target-only estimator $\hat\btheta_{\LR}(u_0)$ satisfies for some constant $C_1 > 0$: 
$$
\mse_A\!\left(\hat\btheta_{\LR}(u_0)\right) \le C_1  n_0^{-1}.
$$
Furthermore, under Assumptions~\ref{Assump:VCM-SN}--\ref{Assump:Unif-Kernel}, the DVCM estimator $\hat\btheta_{\DVCM}(u_0)$ computed with bandwidth
\begin{equation}
\label{eqn:optimal_bw}
    h^* = \med\left(e_0(n/\gamma)^{-\frac{1}{2\beta+1}},d_{(1)}(u_0),d_{(K)}(u_0)\right), \quad e_0>0 \text{ is a constant}.
\end{equation}
satisfies
$$
\mse_A\left(\hat \btheta_{\DVCM}(u_0) \right) \le  C_2\max\left((K/\gamma)^{-2\beta}, (n/\gamma)^{-\frac{2\beta}{2\beta+1}}, n^{-1}\right)\,,
$$
for some constant $C_2 > 0$, where $n = n_0 + \sum_{k=1}^K n_k$.
\end{proposition}
The proof is deferred to Appendix~\ref{proof:prop:rate-LR-VCM}. The rate for $\hat\btheta_{\LR}(u_0)$ follows directly from standard linear regression theory. We now provide intuition for the rate of $\hat\btheta_{\DVCM}(u_0)$ and the associated bandwidth choice. By the standard bias–variance trade-off in nonparametric regression, for any bandwidth $h$,
\begin{equation*}
\textstyle
\mse\!\left(\hat\btheta_{\DVCM}(u_0)\right)\asymp h^{2\beta}+ \gamma/nh.
\end{equation*}
Here, the effective sample size is scaled by $\gamma$ to reflect the dispersion of the domain indices $U_k$. Minimizing the right-hand side over $h$ without constraints yields $h_{\rm opt} \propto (n/\gamma)^{-1/(2\beta + 1)}$ which leads to the classical rate $(n/\gamma)^{-2\beta/(2\beta + 1)}$. However, in our setting, the bandwidth must lie in the interval $[d_{(1)}(u_0), d_{(K)}(u_0)]$. The lower bound arises because $h$ must exceed $d_{(1)}(u_0)$; otherwise, no source domain would fall inside the bandwidth window, leading to degeneracy. The upper bound reflects that if $h>d_{(K)}(u_0)$, then all domains are automatically included, and further enlargement has no additional effect. Therefore, the bandwidth selection amounts to minimizing MSE subject to the constraint $h \in [d_{(1)}(u_0), d_{(K)}(u_0)]$, which yields the truncated (median-based) choice of optimal bandwidth in Equation~\eqref{eqn:optimal_bw}, and consequently the resulting rate is precisely the one stated in Proposition~\ref{prop:rate-LR-VCM}.

We next present our main results regarding the rate of convergence of $\hat \btheta_{\TL}(u_0)$, which shows that for a range of choices of the adaptive penalty $Q$, our proposed estimator does not suffer from negative transfer: 
\begin{theorem}
\label{thm:rate-TL}
Consider the choice of $h$ as in \eqref{eqn:optimal_bw} and let the chosen $Q$ satisfy: 
\begin{equation}
\label{eqn:choice_Q}
\frac{1}{2} \ \frac{\sigma^2(u_0)}{n_0}M\left(\hat\btheta_{\DVCM}(u_0)\right)^{-1}\preceq Q \preceq 2 \ \frac{\sigma^2(u_0)}{n_0}M\left(\hat\btheta_{\DVCM}(u_0)\right)^{-1}
\end{equation}
Then, under Assumptions \ref{Assump:VCM-SN}-\ref{Assump:Unif-Kernel}, for any $u_0\in \mathcal U$ and $A \succeq 0$ the following holds
    \begin{align*}
    & \sup_{\forall j \in [p], \theta_j \in \cH(\beta, L)} \bbE\left[\Vert\hat\btheta_{\TL}(u_0)-\btheta(u_0)\Vert_A^2\right] & \le C\left( n_0^{-1} \wedge \max\left\{(K/\gamma)^{-2\beta}, (n/\gamma)^{-\frac{2\beta}{2\beta+1}}, n^{-1}\right\}\right)\,.
\end{align*}
\end{theorem}
\begin{remark}
Although we use a local polynomial regression–based estimator for $\btheta(\cdot)$ in constructing $\hat\btheta_{\DVCM}(u_0)$, the conclusion of the above theorem remains valid for a broad class of alternative nonparametric estimators of $\btheta(\cdot)$, including spline-based and neural-network–based methods.
\end{remark}
The proof of Theorem \ref{thm:rate-TL} is deferred to Appendix \ref{proof:thm:rate-TL}. As shown in Proposition \ref{prop:rate-LR-VCM}, the $n_0^{-1}$ part is the rate of parametric estimator $\hat\btheta_{\LR}(u_0)$  while the $ \max((K/\gamma)^{-2\beta}, (n/\gamma)^{-{2\beta}/{(2\beta+1)}}, n^{-1})$ part is the rate of nonparametric estimator $\hat\btheta_{\DVCM}(u_0)$. Therefore, the MSE of $\hat \btheta_{\TL}(u_0)$ is always smaller than or equal to the minimum of the MSEs of $\hat\btheta_{\LR}(u_0)$ and $\hat\btheta_{\DVCM}(u_0)$, i.e., the adaptive estimator consistently outperforms or matches the target-only estimator, making it robust to negative transfer. In the following corollary we show that the proposed data-driven estimator $\hat Q$ satisfies Equation~\eqref{eqn:choice_Q} with probability tending to one; consequently, the estimator $\hat\btheta_{\TL}(u_0)$ constructed using $\hat Q$ also satisfies the conclusion of the theorem.
\begin{corollary}
\label{thm:rate-TL-Qest}
Let $\hat M(\hat\btheta_{\DVCM}(u_0))$ and $\hat \sigma^2(u_0)$ be the consistent estimators obtained via the procedure in Section \ref{subsec:est-Q}. Define the estimator $\hat Q$ as: 
$$
\hat Q = \delta\frac{\hat\sigma^2(u_0)}{n_0}\hat M\left(\hat\btheta_{\DVCM}(u_0)\right)^{-1} \ \ \delta \in (1/2, 2) \,.
$$
The estimator $\hat\btheta_{\TL}(u_0)$, obtained by substituting $\hat Q$ into Equation \eqref{eqn:TL-linear}, achieves the same rate of convergence as in Theorem \ref{thm:rate-TL} with probability tending to 1.
\end{corollary}
The proof of Corollary \ref{thm:rate-TL-Qest} can be found in Appendix \ref{proof:thm:rate-TL-Qest}. We conclude this section with a theorem establishing the minimax lower bound for estimating $\btheta(u_0)$, thereby confirming that our proposed estimator achieves the minimax-optimal rate.
 
\begin{theorem}
\label{thm:minimax-TL}
    Under Assumptions \ref{Assump:VCM-SN}-\ref{Assump:balance-sample}, for any $u_0\in \mathcal U$, it holds that
\begin{align*}
    & \inf_{\hat{\btheta}(u_0)} \sup_{\forall j \in [p], \theta_j \in \cH(\beta, L)} \bbE\left[\|\hat\btheta(u_0)-\btheta(u_0)\|_A^2\right] \ge  C'\left\{n_0^{-1} \wedge \max\left((K/\gamma)^{-2\beta}, (n/\gamma)^{-\frac{2\beta}{2\beta+1}}, n^{-1}\right)\right\}\,.
\end{align*}
\end{theorem}
The proof of the above theorem is provided in Appendix~\ref{proof:thm:minimax-TL}. Together, Theorem~\ref{thm:rate-TL} and Theorem~\ref{thm:minimax-TL} establish the minimax optimality of our proposed estimator. 

\subsection{Inference with linear DVCM}
\label{sec:linear-infer}
In the previous subsection, we established the rate of convergence of $\hat{\btheta}(u_0)$. However, convergence rates alone are insufficient for inferential tasks, such as testing hypotheses of the form $H_0:\btheta(u_0)=\mathbf{0}$ versus $H_1:\btheta(u_0)\neq\mathbf{0}$. To this end, we now establish the asymptotic normality of the proposed estimator. We begin by introducing modifications to the earlier assumptions required to derive the asymptotic normality result.
\setcounter{assumptionprime}{1}
\begin{assumptionprime}[Modification of Assumption \ref{Assump:VCM-UX}]
\label{Assump:VCM-UX-new}
The distribution of $(U, X, Y)$ is assumed to satisfy the conditions of Assumption \ref{Assump:VCM-UX}. Furthermore, the conditional second moment matrix \(\Psi(u) = \bbE[XX^\top \mid U = u]\) is positive definite and continuous almost everywhere on \(\cU\). Moreover, there exists
constants \(0<c_0'<c_0<\infty\) such that, almost surely for all \(u\in\cU\),
\[
c_0' \le \lambda_{\min}\big(\Psi(u)\big)\le \lambda_{\max}\big(\Psi(u)\big)\le c_0.
\]
\end{assumptionprime}
\noindent 
\noindent
Compared to Assumption~\ref{Assump:VCM-UX}, Assumption~\ref{Assump:VCM-UX-new} imposes an additional uniform lower and upper bound on the conditional variance matrix of $X$ given $U$ to facilitate Lindeberg-type central limit theorem arguments. Towards presenting our main result, let us introduce $r_{\LR}$ and $r_{\DVCM}$, which denote the convergence rates of $\hat\btheta_{\LR}(u_0)$ and $\hat\btheta_{\DVCM}(u_0)$, respectively, i.e. 
$\hat\btheta_{\LR}(u_0)-\btheta(u_0)=O_p(r_{\LR})$ and
$\hat\btheta_{\DVCM}(u_0)-\btheta(u_0)=O_p(r_{\DVCM})$.
We define their relative rate by $\rho_n:= r_{\LR}/r_{\DVCM}$. Note that we allow $r_{\DVCM}$ to be bounded away from $0$, implying the non-informativeness of the sources. 
The following theorem characterizes the asymptotic distribution of the proposed
transfer learning estimator using $\rho_n$.

\begin{theorem}\label{thm:Dconv-LR-TF-0-infty}
Suppose Assumptions~\ref{Assump:VCM-SN}, \ref{Assump:VCM-UX-new},
\ref{Assump:balance-sample}, and \ref{Assump:Unif-Kernel} hold.
Let $\hat\btheta_{\TL}(u_0)$ be constructed as in
\eqref{eqn:est-DVCM-linear}, with the shrinkage matrix $\hat Q$
defined in Section~\ref{subsec:est-Q} and the bandwidth parameter $h$ is chosen to satisfy: 
\begin{equation}
\textstyle
\label{eq:bandwidth_cond_inf}
\gamma/K \ll h \ll \left(\gamma/n\right)^{\frac{1}{2\beta + 1}} \,.
\end{equation}
Then the adaptive transfer learning estimator $\hat\btheta_{\TL}(u_0)$
satisfies the following asymptotic results:
\begin{align*}
&\text{If } \rho_n \to 0,\quad
\sqrt{n_0}\bigl(\hat\btheta_{\TL}(u_0)-\btheta(u_0)\bigr)
\xrightarrow[]{d}
\cN\!\bigl(0,\Omega_{\LR}(u_0)\bigr),\\[0.5em]
&\text{If } \rho_n \to \infty, \quad
\sqrt{\frac{nh}{\gamma}}
\bigl(\hat\btheta_{\TL}(u_0)-\btheta(u_0)\bigr)
\xrightarrow[]{d}
\cN\!\bigl(0,\Omega_{\DVCM}(u_0)\bigr).
\end{align*}
Here $\Omega_{\LR}(u_0)$ and $\Omega_{\DVCM}(u_0)$ denote the asymptotic
covariance matrices of $\hat\btheta_{\LR}(u_0)$ and
$\hat\btheta_{\DVCM}(u_0)$, respectively. 
\end{theorem}

The proof of this theorem is deferred to Appendix \ref{proof:thm:Dconv-LR-TF-0-infty}. 
The theorem highlights the adaptivity of the transfer learning estimator. When $\rho_n \to 0$, the target-only estimator converges at a faster rate than the DVCM estimator, and consequently, the transfer learning estimator attains the same convergence rate as the target-only estimator. In contrast, when $\rho_n \to \infty$, the DVCM estimator converges faster than the target-only estimator, and the transfer learning estimator correspondingly achieves a convergence rate comparable to that of the DVCM estimator.

The asymptotic normality results, particularly in the regime where $\rho_n \to \infty$, require additional conditions on the choice of the bandwidth $h$ (see Equation~\eqref{eq:bandwidth_cond_inf}). To motivate the condition, let us briefly recall the classical bandwidth condition required for establishing the asymptotic normality of a pointwise nonparametric regression estimator. Suppose we observe $Y_i = f_*(X_i) + \varepsilon_i$, where $f_*$ belongs to a $\beta$-Hölder class and we wish to estimate $f_*(x_0)$ at a fixed point $x_0$. In this setting, the bandwidth that achieves the minimax-optimal rate is $h_* \propto n^{-1/(2\beta + 1)}$, which yields the optimal convergence rate $n^{-\beta/(2\beta + 1)}$. However, this bandwidth choice does not generally yield a centered asymptotic normal distribution, because the bias term is of the same order as the stochastic fluctuations. To establish asymptotic normality, then one should either correct for the bias or performs undersmoothing \cite{hall1992effect,calonico2018effect}, i.e., choose $h$ such that $nh^{2\beta+1} \to 0$, which would yield: s
\begin{equation*}
    \textstyle
\sqrt{nh}(\hat f(x_0) - f_*(x_0)) \implies \cN(0,\sigma^2),
\end{equation*}
where $\sigma^2$ is the asymptotic variance. The undersmoothing condition ensures that the bias term $h^{2\beta}$ is asymptotically negligible relative to the stochastic error $1/nh$, thereby yielding a centered normal limit. This centering is essential for constructing valid $(1-\alpha)$-level confidence intervals. The trade-off is a slightly slower rate of convergence, since $\sqrt{nh}$ is strictly smaller than the minimax-optimal rate $n^{\beta/(2\beta + 1)}$ whenever $n h^{2\beta+1} \to 0$. In this paper, we take this undersmoothing approach. 

The bandwidth condition in Equation~\eqref{eq:bandwidth_cond_inf} reflects precisely the undersmoothing phenomenon discussed above. In particular, the requirement $n h^{2\beta+1}/\gamma \to 0$ ensures that the squared bias of $\hat\btheta_{\DVCM}(u_0)$, which is of order $h^{2\beta}$, is asymptotically negligible compared to its variance term, which is of order $\gamma/(nh)$. Consequently, the stochastic fluctuations dominate the bias, leading to an asymptotically normal distribution centered at zero. The additional condition $Kh/\gamma \to \infty$ guarantees that the effective number of source domains satisfying $|U_k-u_0|\le h$ diverges (recall that $\gamma/K$ is the order of the minimum distance between the target and the source indicators). In other words, the local polynomial estimator underlying $\hat\btheta_{\DVCM}(u_0)$ is constructed from an increasing amount of source-domain information. Without this condition, the number of contributing domains would remain bounded, precluding the application of central limit theorem arguments and hence preventing asymptotic normality. Together, these conditions ensure that the estimator is both bias-negligible and supported by a sufficiently large effective sample size, thereby yielding a valid Gaussian limit suitable for inference.
It is apparent from Theorem~\ref{thm:Dconv-LR-TF-0-infty} that the convergence rate and limiting variance of $\hat{\btheta}_{\TL}(u_0)$ depend on whether $\rho_n \to 0$ or $\rho_n \to \infty$. However, in practice, the true regime is typically unknown. Therefore, a unified representation of the limiting variance is necessary for drawing valid inferences across all regimes. The following corollary serves this purpose: 
\begin{corollary}
    \label{coro:TL-linear-infer} 
    Under the conditions of Theorem~\ref{thm:Dconv-LR-TF-0-infty}, the estimator $\hat{\btheta}_{\TL}(u_0)$ satisfies: 
    \begin{equation*} 
    \textstyle
    \hat{\Sigma}_{\TL}^{-1/2} \left( \hat{\btheta}_{\TL}(u_0) - \btheta(u_0) \right) \xrightarrow{d} \cN(0, I), 
    \end{equation*} 
    where the unified covariance estimator is given by:
    \begin{equation*} 
    \textstyle
    \hat{\Sigma}_{\TL} = B_Q^{-1} \hat{Q} \, \hat{V}_{\DVCM}(u_0) \, \hat{Q} B_Q^{-1} + B_Q^{-1} \hat{\Psi}(u_0) \, \hat{V}_{\LR}(u_0) \, \hat{\Psi}(u_0) B_Q^{-1}, 
    \end{equation*} 
    with $B_Q = \hat{\Psi}(u_0) + \hat Q$. 
\end{corollary}
The proof of this corollary can be found in  Appendix~\ref{proof:coro:TL-linear-infer}. 
This corollary gives the practitioner a concrete form of the standard error of the $\hat \btheta_{\rm TL}(u_0)$, which, relies on $\hat\Psi(u_0)$, $\hat V_{\LR}(u_0)$ and $\hat V_{\DVCM}(u_0)$, consistent estimators for $\Psi(u_0)$, $\Var(\hat \btheta_{\LR}(u_0))$ and $\Var(\hat \btheta_{\DVCM}(u_0))$, respectively. One can easily construct such consistent estimators by taking their sample analogues, as prescribed below:  
\begin{enumerate}
    \item  $\hat\Psi(u_0)=\tfrac{1}{n_0}\sum_{i \in \cI_0}X_{0i}X_{0i}^\top$;
    \item $\hat V_{\LR}(u_0) = \hat\Psi(u_0)^{-1}[\tfrac{1}{n_0}\sum_{i\in\cI_0} X_{0i}X_{0i}^\top \, \hat\varepsilon_{0i}(u_0)^2]\hat\Psi(u_0)^{-1}$ with $\hat\varepsilon_{0i}(u_0)=Y_{0i}-X_{0i}^\top \hat\btheta_{\LR}(u_0)$;
    \item $\hat V_{\DVCM}(u_0)$ is defined in the same way as in Equation~\eqref{eqn:var-est-VCMS-in}.
\end{enumerate}
It is possible to readily use the result of the above corollary for various types of inference problems.
For instance, to test the null hypothesis $\btheta(u_0) = \bw$ for a given vector $\bw$, one may consider the test statistic $T_n = \|\hat{\Sigma}_{\TL}^{-1/2}(\hat{\btheta}_{\TL}(u_0) - \bw)\|_2^2$, which, under the null hypothesis, follows a $\chi^2$ distribution with $p$ degrees of freedom, where $p$ denotes the dimension of $X$.
Furthermore, to test a linear contrast of the form $\bv^\top \btheta(u_0) = \zeta$ for a given scalar $\zeta$, one may use the statistic $T_n = (\bv^\top \hat{\btheta}_{\TL}(u_0) - \zeta)/\sqrt{\bv^\top \hat{\Sigma}_{\TL} \bv}$, which converges in distribution to $\cN(0,1)$ under the null hypothesis.

\begin{remark}
A key technical ingredient in the proof of Theorem~\ref{thm:Dconv-LR-TF-0-infty} is the derivation of the limiting distribution of $\hat{\btheta}_{\DVCM}(u_0)$ under the regime $\rho_n \to \infty$. This result is formalized in Proposition~\ref{prop:asymp-LR-VCM}, stated in Appendix~\ref{sec:aux-lem}. Briefly, the proposition establishes that if the bandwidth $h$ is chosen to satisfy condition~\eqref{eq:bandwidth_cond_inf}, then under the stated assumptions,
\begin{equation*}
\textstyle
  \sqrt{nh/\gamma}\left(\hat\btheta_{\DVCM}(u_0) - \btheta(u_0) - \bb_{\DVCM}(u_0)\right) 
    \xrightarrow[]{d} \cN\left( 0, \Omega_{\DVCM}(u_0)\right) \,,
\end{equation*}
where
\begin{equation*}
\textstyle
\Omega_{\DVCM}(u_0) = \sigma^2(u_0)\left[\zeta_{0,1}^{-1}\,\zeta_{0,2}\,\zeta_{0,1}^{-1}\right]_{1,1} \Psi(u_0)^{-1}, \ \ \zeta_{r,s} = \int \bPhi_l(t)^{\otimes 2} t^r W^s(t) f\Big(\tfrac{u_0 - u^* + ht}{\gamma}\Big)\,dt \,,
\end{equation*}
and $\bb_{\DVCM}(u_0)\asymp h^\beta$ denotes the bias term. Here, $u^*$ is same as defined as in Assumption~\ref{Assump:VCM-UX}. This result naturally generalizes the classical asymptotic theory for varying-coefficient models (VCMs), as developed in \cite{fan1999statistical}, by allowing multiple observations per domain value $u$. In particular, when the per-domain sample size satisfies $n_k=1$ for all $k$, Proposition~\ref{prop:asymp-LR-VCM} reduces to the standard asymptotic normality result for the classical VCM estimator of \cite{fan1999statistical}.
\end{remark}

\medskip
\noindent
{\bf Does our choice of $h$ make $\hat\btheta_{\TL}(u_0)$ adaptive?} A natural question is whether the bandwidth choice in Equation~\eqref{eq:bandwidth_cond_inf} renders $\hat\btheta_{\TL}(u_0)$ adaptive, or whether it could lead to negative transfer. However, a closer inspection of our arguments (see Appendix \ref{proof:thm:rate-TL}), shows that the following conclusion holds regardless of the bandwidth choice:
\begin{equation*}
    \textstyle
    \mse_A\left(\hat\btheta_{\TL}(u_0) \right) \leq \min\left\{\mse_A\left(\hat\btheta_{\DVCM}(u_0) \right),\mse_A\left(\hat\btheta_{\LR}(u_0) \right)\right\}\,. 
\end{equation*}
as long as $Q$ satisfies Equation \eqref{eqn:choice_Q} (which $\hat Q$ satisfies with probability going to $1$, as established in Corollary \ref{thm:rate-TL-Qest}). In particular, any bandwidth $h$ satisfying Equation~\eqref{eq:bandwidth_cond_inf} still guarantees that $\hat\btheta_{\TL}(u_0)$ is free from negative transfer, since its risk never exceeds that of the target-only baseline. 

\subsection{Extension to generalized DVCM}
\label{sec:G-linear}
In this section, we establish theoretical properties of $\hat \btheta_{\rm TL}(u_0)$ under the assumption that the data are generated from a generalized linear model (see Equation~\eqref{eqn:dgp_glm}). As discussed in Section~\ref{subsec:method-TL-GLM}, the proposed estimation procedure is closely related to that of the linear model. The key difference lies in replacing the squared-error loss with a more general negative log-likelihood loss, which is appropriate for the GLM framework. The assumptions are similar to that for the linear DVCM model, except we modify Assumption \ref{Assump:VCM-UX-new} as follows: 
\setcounter{assumptiondoubleprime}{1}
\begin{assumptiondoubleprime}[Modification of Assumption \ref{Assump:VCM-UX-new}]
\label{Assump:VCM-UX-Generalized}
    The distribution of $(U, X, Y)$ is assumed to satisfy the conditions of Assumption \ref{Assump:VCM-UX-new}, but with the conditional second moment matrix defined as \(\bbE\left[b''\big(X^\top\btheta(U)\big)\,XX^\top \mid U=u\right]\). Furthermore, it is assumed that $\sup_{u\in\cU}\bbE\left[|b^{(3)}\big(X^\top\theta(u)\big)|^4\,\middle|\,U=u\right]$ is uniformly bounded. \hfill\relax
\end{assumptiondoubleprime}

\noindent{\bf Discussion on the augmented assumptions:}  
In Assumption~\ref{Assump:VCM-UX-Generalized}, we generalize the definition of $\Psi(\cdot)$ (defined in Assumption~\ref{Assump:VCM-UX-new}) by incorporating the second derivative of the mean function $b(\cdot)$. Note that, for the linear model, $b(x)=x^2/2$ and hence $b''(x)=1$, in which case the new definition of $\Psi(\cdot)$ reduces to the form given in Assumption~\ref{Assump:VCM-UX-new} for the linear DVCM model. Moreover,  a mild bounded-moment condition on $b^{(3)}\!\left(X^\top \theta(u)\right)$ imposed, which is required to establish a central limit theorem via higher-order Taylor expansions. Such a regularity condition is standard in the literature and is commonly assumed when establishing weak convergence of the proposed estimator.

We are now ready to present our main theoretical results. As in the linear model setting, we present two main theorems: one characterizing the rate of convergence and the other describing the asymptotic normality of the proposed estimator. Our first result concerns the convergence rate of $\hat{\btheta}_{\TL}(u_0)$ under the generalized DVCM framework and serves as the counterpart to Theorem~\ref{thm:rate-TL} in the linear case:

\begin{theorem}\label{thm:TL-GLM-rate-Op}
Suppose Assumptions~\ref{Assump:VCM-SN}, \ref{Assump:VCM-UX-Generalized},
\ref{Assump:balance-sample}, and \ref{Assump:Unif-Kernel} hold.
Let $\hat\btheta_{\TL}(u_0)$ be constructed as in \eqref{eqn:TL-GLM}, with the shrinkage
matrix $\hat Q$ defined in Section~\ref{subsec:est-Q}.
If the bandwidth is chosen as 
\[
h=\operatorname{med}\!\left(e_0 (n/\gamma)^{-\frac{1}{2\beta+1}},\ d_{(1)}(u_0),\ d_{(K)}(u_0)\right),
\quad\text{for some constant } e_0>0,
\]
then the adaptive transfer learning estimator $\hat\btheta_{\TL}(u_0)$ satisfies
\[
\|\hat\btheta_{\TL}(u_0)-\btheta(u_0)\|_2^2
= O_p\!\left( n_0^{-1}\ \wedge\ \max\Bigl\{(K/\gamma)^{-2\beta},\ (n/\gamma)^{-\frac{2\beta}{2\beta+1}},\ n^{-1}\Bigr\} \right).
\]
\end{theorem}
\noindent
The proof of Theorem~\ref{thm:TL-GLM-rate-Op} is deferred to Appendix~\ref{proof:thm:TL-GLM-rate-Op}. 
This result demonstrates that, under the proposed choices of the bandwidth $h$ and the weighting matrix $\hat Q$, the rate of convergence of $\hat \btheta_{\rm TL}(u_0)$ under GLM setting coincides with that obtained in the linear case (Theorem~\ref{thm:rate-TL}). 
The adaptivity of $\hat \btheta_{\rm TL}(u_0)$ is also evident: its convergence rate is never worse than $n_0^{-1/2}$, thereby precluding negative transfer. Moreover, when the number of source $K$, or the smoothness parameter $\beta$, is large, or the heterogeneity parameter $\gamma$ is small, the convergence rate is strictly faster than that of the target-only estimator, reflecting the ability of the method to efficiently leverage information from the relevant source domains.

Having established the rate, we next present a result on the asymptotic normality of the proposed estimator. 
Let $\hat{\Psi}(u_0)$, $\hat V_{\GLR}(u_0)$, and $\hat V_{\GDVCM}(u_0)$ denote consistent estimators of $\Psi(u_0)$, $\Var(\hat{\btheta}_{\GLR}(u_0))$, and $\Var(\hat{\btheta}_{\GDVCM}(u_0))$, respectively. Then, the following asymptotic normality result holds for $\hat \btheta_{\rm TL}(u_0)$:


\begin{theorem}\label{thm:TL-GLM-infer}
Let $\hat Q$ be a pre-specified positive semidefinite matrix.
Suppose Assumptions~\ref{Assump:VCM-SN}, \ref{Assump:VCM-UX-Generalized},
\ref{Assump:balance-sample}, and \ref{Assump:Unif-Kernel} hold and the bandwidth $h$ satisfy Equation \eqref{eq:bandwidth_cond_inf}.
The transfer learning estimator $\hat\btheta_{\TL}(u_0)$ is asymptotically normal:
\[
\hat\Sigma_{\TL}^{-1/2}\bigl(\hat\btheta_{\TL}(u_0)-\btheta(u_0)\bigr)
\xrightarrow[]{d} \cN(0,I).
\]
The asymptotic covariance estimator $\hat\Sigma_{\TL}$ is given by
\[
\hat\Sigma_{\TL}
=
 B_Q^{-1} \hat Q\, \hat V_{\GDVCM}(u_0)\, \hat Q\, B_Q^{-1}
+
B_Q^{-1} \hat\Psi(u_0)\, \hat V_{\GLR}(u_0)\, \hat\Psi(u_0)\, B_Q^{-1},
\]
with $B_Q=\hat\Psi(u_0)+\hat Q$.
\end{theorem}
The proof of the above theorem can be found in Appendix~\ref{proof:thm:TL-GLM-infer}. This result extends Corollary~\ref{coro:TL-linear-infer} to the GDVCM framework. Since the inference procedure relies on $\hat \Sigma_{\rm TL}$, which in turn depends on consistent estimation of $\Psi(u_0), \Var(\hat{\btheta}_{\GLR}(u_0)), \Var(\hat{\btheta}_{\GDVCM}(u_0))$, we next describe consistent estimators for these quantities, obtained as empirical analogues of their population definitions:
\begin{enumerate}
    \item $\hat\Psi(u_0)=\tfrac{1}{n_0}\sum_{i\in\cI_0} b''\!\big(X_{0i}^\top\hat\btheta_{\GLR}(u_0)\big)\,X_{0i}X_{0i}^\top$;
    \item $\hat V_{\GLR}(u_0)
=\hat\Psi(u_0)^{-1}\big[\tfrac{1}{n_0}\sum_{i\in\cI_0} X_{0i}X_{0i}^\top (Y_{0i}-\hat\mu_{0i})^2\big]\hat\Psi(u_0)^{-1}$,
with $\hat\mu_{0i}=b'\!\big(X_{0i}^\top\hat\btheta_{\GLR}(u_0)\big)$;
    \item $\hat V_{\GDVCM}(u_0)$ is defined in the same way as in Equation~\eqref{eqn:var-est-VCMS-in}.
\end{enumerate}
By a standard application of the law of large numbers, the proposed estimators are consistent, which in turn guarantees the validity of the resulting inferential procedures. As illustrated in the linear model setting in Section~\ref{sec:linear-infer}, the above asymptotic normality result can be directly employed to conduct inference in a variety of testing problems. These include, for example, testing pointwise hypotheses of the form $\btheta(u_0) = \bw$, as well as more general linear constraints on $\btheta(u_0)$, such as $R\,\btheta(u_0) = r,$ for a given matrix $R$ and vector $r$ (e.g., testing whether a particular coordinate or linear combination equals zero).

\section{Simulation experiments}
\label{sec:sim}
In this section, we present various numerical experiments to support and illustrate our theoretical results. We investigate three distinct models: linear regression, logistic regression, and Poisson regression. Across these settings, we examine several key properties of our estimator (e.g., its rate of convergence, sensitivity to bandwidth selection, robustness under varying levels of similarity between the source and target domains, asymptotic normality) by varying factors such as the sample size ($n$ and $\bar n$), the number of domains ($K$), and heterogeneity among domain identifiers ($\gamma$).
The (generalized) DVCM estimators considered in this section are local linear estimators, \,  i.e., \ we set $l=1$ in Equation \eqref{eqn:VCM-GLM}. 
\\\\
\noindent
{\bf Data generation: }We use the following data-generating setup for our simulation studies: 
\begin{enumerate}
\setlength\itemsep{0.5em}
    \item We generate $U_1, \dots, U_K \sim \text{Unif}(-\gamma/2,\gamma/2)$, \ie\ a centered uniform distribution of length $\gamma$. We vary the value of $\gamma$ to control the degree of heterogeneity among these domain identifiers. The target $u_0$ is fixed to be $0$. 

    \item We set $X_{ki} \in \reals^p$, where the first coordinate is the intercept and the other coordinates are generated from $\cN(0, \Sigma)$ for all $1 \le k \le K$, $1 \le i \le n_k$, with $\Sigma_{ij}=0.7^{|i-j|}$. (The choice of $p$ will be specified later).

    \item The true parameter vector is specified as $\btheta(u) = (\theta_0(u), \ldots,\theta_{p-1}(u))$, where $\theta_0(u) = -\tanh(16(u - 0.2)) + g(u)$, $\theta_1(u) = \exp(5u+2.5)/100-0.5+g(u)$, and $\theta_j(u) = (-0.5)^{j-1}\exp(2u)$ for $j\ge 2$. The additional term $g(u) = u^3\cdot \sign(u)$ is included to ensure that $\btheta(\cdot)$ possesses a continuous second derivative but a discontinuous third derivative. This construction also guarantees that the linear predictor $X^\top \btheta(u)$ remains in a reasonable range and, in the binary response setting (defined below), that the success probability $\Pr(Y=1 | X, U)$ is not too close to $0$ or $1$.

    \item The response variable $Y$ is generated as: 
    \begin{itemize}
        \item For linear regression: $Y_{ki} = X_{k, i}^\top \btheta(U_k) + 0.5 \times \eps_{ki}, \ \ \eps_{ki} \sim \cN(0, 1)$. 
        \item For logistic regression: $Y_{ki} \sim \Ber\left(\sigma(X_{k, i}^\top \btheta(U_k))\right)$ with $\sigma(x)= (1 + e^{-x})^{-1}$. 
        \item For Poisson regression: $Y_{ki} \sim \Poi\left(\exp{\left(X_{k, i}^\top \btheta(U_k)\right)}\right)$. 
\end{itemize}
\end{enumerate}
{\bf Estimation procedure: }For each of the three response-generating mechanisms, we obtain the maximum likelihood estimator by minimizing the negative log-likelihood, which serves as our loss function.
We compute three different estimators in our simulation studies: i) the target only $\hat \btheta_{\GLR}(u_0)$ defined in \eqref{eq:def-GLR}, ii) the DVCM estimator $\hat \btheta_{\DVCM}(u_0)$ defined in \eqref{eqn:VCM-GLM}, and iii) the transfer learning estimator $\hat \btheta_{\TL}$ in \eqref{eqn:TL-GLM}. We perform data-split on the target domain to make $\hat \btheta_{\DVCM}(u_0)$ and $\hat \btheta_{\GLR}(u_0)$ independent. For $\hat \btheta_{\TL}(u_0)$, we compute $\hat Q$ via method in Section \ref{subsec:est-Q}. 

\subsection{Bandwidth sensitivity analysis}
\label{sim-dev}
\noindent
\underline{\textbf{Performance across $\gamma$.}}
First, we vary $\gamma$ while keeping all other parameters fixed. Recall that $\gamma$ controls the dispersion of the domain identifiers $U_k$ around the target $u_0$. When $\gamma$ is small, the domains are concentrated near $u_0$, so the source domains are informative for the target. When $\gamma$ is large, the domains are more dispersed and become increasingly irrelevant. Consequently, the DVCM estimator is expected to perform better when $\gamma$ is small, whereas the GLR estimator should dominate when $\gamma$ is large. The transfer learning estimator $\hat\btheta_{\TL}(u_0)$ is designed to adapt between these regimes. We consider the setting $p=4$ (dimension of $X$), $n_S = 600$ (total number of source samples), $n_0 = 50$ (target samples), and $K = 5$ (number of source domains). Figure~\ref{fig:MSE_h_VarU} reports the MSE $\mathbb{E}\| \hat\btheta(u_0) - \btheta(u_0) \|_2^2$ of the three estimators $\hat\btheta_{\DVCM}$, $\hat\btheta_{\GLR}$, and $\hat\btheta_{\TL}$ based on 200 simulations for $\gamma \in \{0.5, 1, 1.5\}$ across a range of bandwidths $h$. The left, middle, and right columns correspond to linear, logistic, and Poisson regression models, respectively, while the upper, middle, and lower rows correspond to the cases $\gamma = 0.5, 1$, and $1.5$, respectively. 
Since $\hat{\btheta}_{\GLR}(u_0)$ is a target-only estimator independent of the bandwidth $h$, its MSE remains constant as $h$ varies and therefore appears as a flat line.
For the linear model (left column), when $\gamma = 0.5$, the domains are highly related and $\hat\btheta_{\DVCM}(u_0)$ achieves the smallest MSE. In contrast, the target-only estimator $\hat\btheta_{\GLR}(u_0)$ underperforms in this regime because it ignores the informative source data. The transfer learning estimator $\hat\btheta_{\TL}(u_0)$ closely tracks $\hat\btheta_{\DVCM}(u_0)$ and inherits its advantage. When $\gamma = 1$, the source domains are moderately close to the target. The estimator $\hat\btheta_{\DVCM}(u_0)$ performs well for smaller bandwidths but deteriorates as $h$ increases, while $\hat\btheta_{\GLR}(u_0)$ is stable as it does not depend on the choice of bandwidth. The adaptive estimator $\hat\btheta_{\TL}(u_0)$ adapts between the best of the two and remains near-optimal across bandwidth choices. Finally, when $\gamma = 1.5$, the source domains are less relevant, and $\hat\btheta_{\DVCM}(u_0)$ suffers from substantial bias. In this regime, $\hat\btheta_{\GLR}(u_0)$ outperforms the pooled estimator. The transfer learning estimator aligns with $\hat\btheta_{\GLR}(u_0)$ and again achieves the comparable performance across bandwidths. The logistic and Poisson models (middle and right columns) exhibit the same qualitative pattern: $\hat\btheta_{\DVCM}(u_0)$ dominates when $\gamma$ is small, $\hat\btheta_{\GLR}(u_0)$ dominates when $\gamma$ is large, and $\hat\btheta_{\TL}(u_0)$ adaptively tracks the better estimator in each regime. Overall, the figure demonstrates that $\hat\btheta_{\TL}(u_0)$ consistently achieves the lowest MSE across bandwidths, values of $\gamma$, and model families by adaptively combining the strengths of $\hat\btheta_{\GLR}(u_0)$ and $\hat\btheta_{\DVCM}(u_0)$.

\medskip
\noindent
\underline{\textbf{Performance across $K$.}} 
We next vary $K \in \{5,10,15\}$ and compare the performance of the three estimators as before. The results are summarized in Figure~\ref{fig:bw-K}, where we plot the MSE as a function of the bandwidth $h$ under the Linear, Logistic, and Poisson models. Throughout these simulations, we fix $p=4$, $\bar n = 120$ (i.e., 120 observations per source domain), $n_0 = 50$, and $\gamma = 1$. The qualitative conclusions are similar to those in the previous setup. As before, $\hat{\btheta}_{\GLR}(u_0)$ is independent of the bandwidth and therefore appears as a flat line. In contrast, $\hat{\btheta}_{\DVCM}(u_0)$ relies heavily on the bandwidth choice; selecting $h$ either too small or too large leads to larger/suboptimal MSE. The proposed estimator $\hat{\btheta}_{\TL}(u_0)$ adaptively combines the strengths of these two approaches. When $\hat{\btheta}_{\DVCM}(u_0)$ achieves a smaller MSE than $\hat{\btheta}_{\GLR}(u_0)$, the estimator $\hat{\btheta}_{\TL}(u_0)$ attains an MSE that is very close to, and occasionally even smaller than, that of $\hat{\btheta}_{\DVCM}(u_0)$. Conversely, when the MSE of $\hat{\btheta}_{\DVCM}(u_0)$ exceeds that of $\hat{\btheta}_{\GLR}(u_0)$ due to a suboptimal bandwidth choice, the performance of $\hat{\btheta}_{\TL}(u_0)$ automatically aligns with that of $\hat{\btheta}_{\GLR}(u_0)$. These experiments clearly demonstrate the adaptive nature of the proposed method.
\begin{figure}
\begin{center}
\includegraphics[width=\textwidth]{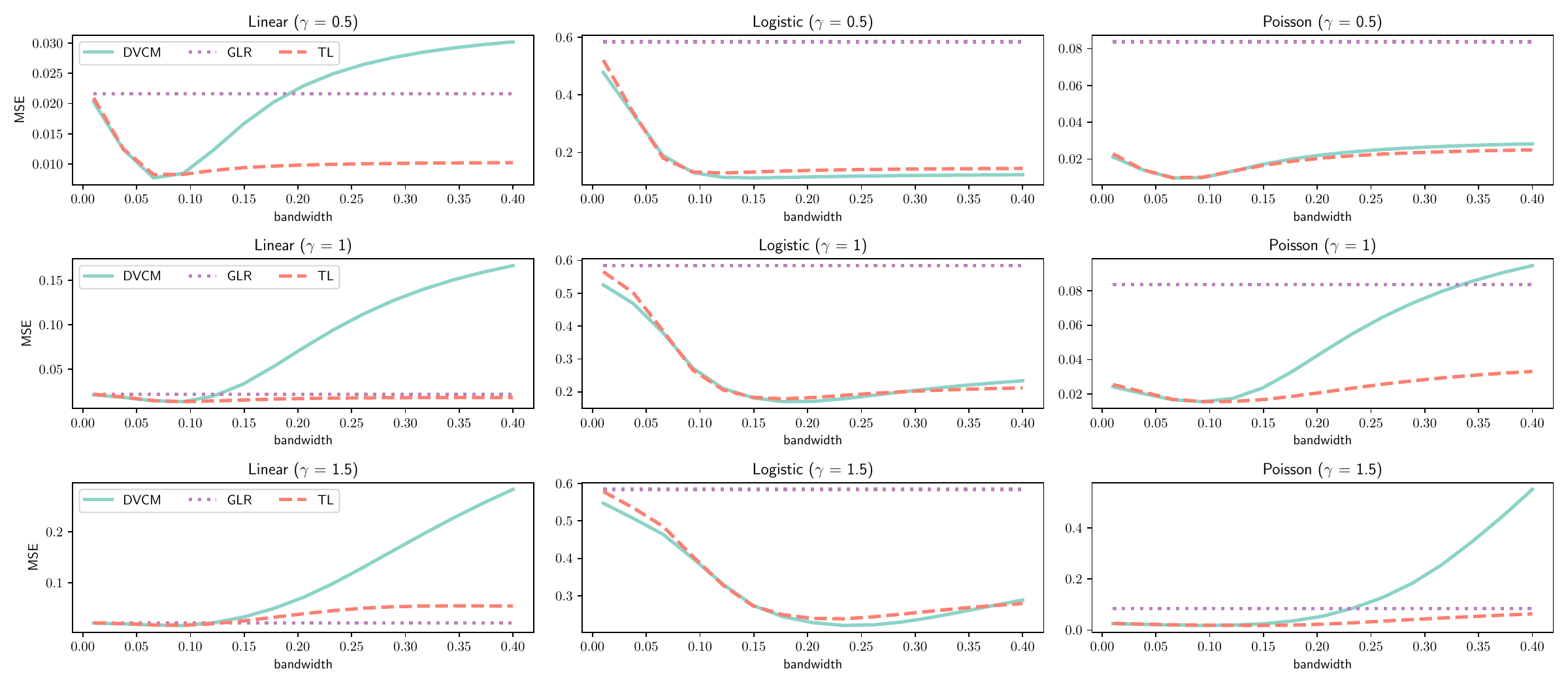}
\end{center}
\caption{MSE of the estimators across different $h$ and $\gamma$, with $(n, n_0,K)$ fixed. The left, middle, and right panels show MSE of linear, logistic, and Poisson-based estimators while upper, middle, and lower panels are cases where $\gamma = 0.5, 1$, and $1.5$.
}
\label{fig:MSE_h_VarU}
\end{figure}
\begin{figure}
    \centering
    \includegraphics[width=\linewidth]{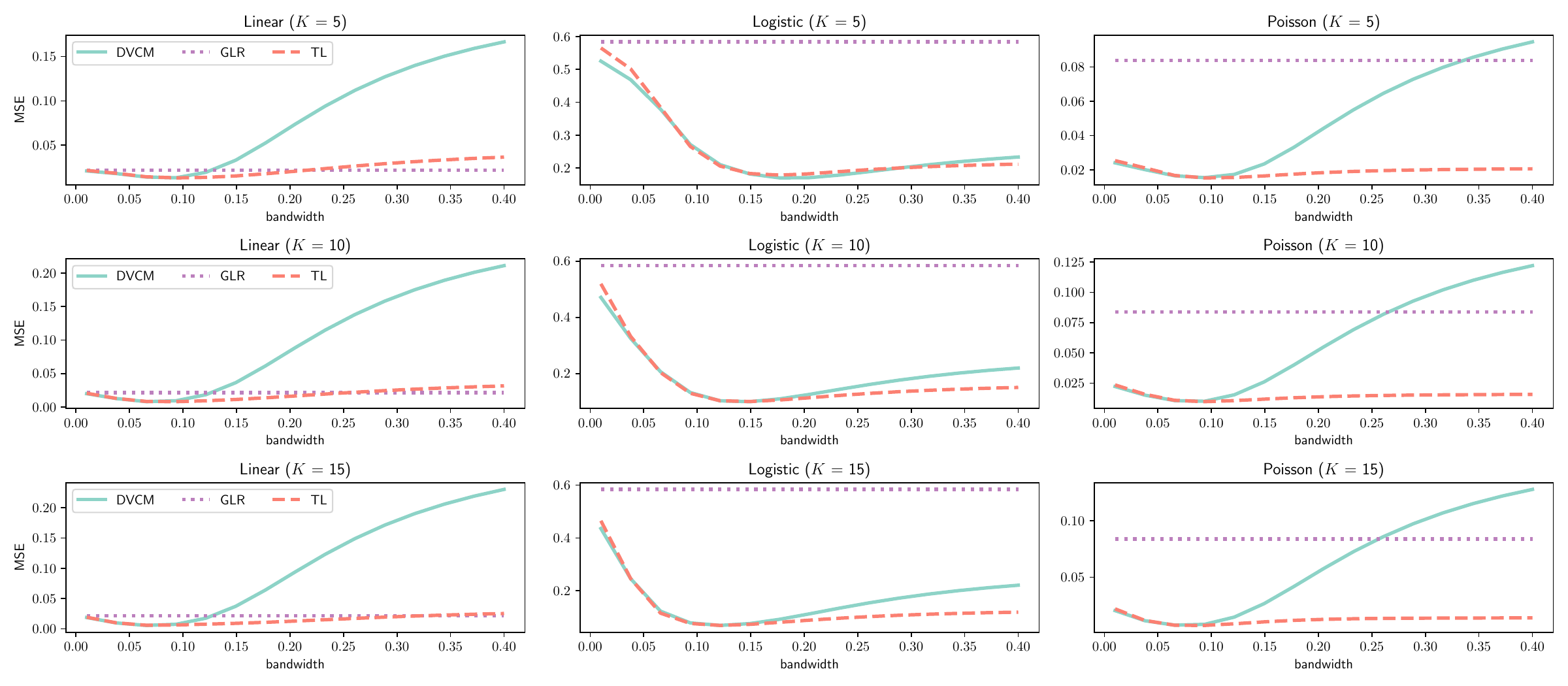}
    \caption{MSE of the estimators across different $h$ and $K$, with $(\bar n, n_0,\gamma)$ fixed. The left, middle, and right panels show MSE of linear, logistic, and Poisson-based estimators while upper, middle, and lower panels are cases where $K = 5, 10$, and $15$.}
    \label{fig:bw-K}
\end{figure}
\subsection{Asymptotic normality}\label{sim:norm}

In this section, we present simulation results illustrating the asymptotic normality of $\hat \btheta_{\TL}(u_0)$, as established in Theorem~\ref{thm:Dconv-LR-TF-0-infty}. For simplicity, we focus exclusively on the linear data-generating model. Recall that $\rho_n$ denotes the relative efficiency ratio between $\hat{\btheta}_{\LR}(u_0)$ and $\hat{\btheta}_{\DVCM}(u_0)$, and that the proposed estimator $\hat{\btheta}_{\TL}(u_0)$ adapts to the better of the two procedures, achieving asymptotic normality in both regimes ($\rho_n \downarrow 0$ or $\rho_n \uparrow \infty$). Here, we numerically demonstrate the asymptotic normality of $\hat{\btheta}_{\TL}(u_0)$ in both of these regimes. To simulate the case $\rho_n \to 0$, we set $(K,\gamma)=(5,5)$ with $\bar n=100$ and $n_0=50$, and for $\rho_n \to \infty$, we set $(K,\gamma)=(30,0.1)$ while keeping $\bar n$ and $n_0$ unchanged. 

As predicted by Theorem~\ref{thm:Dconv-LR-TF-0-infty}, the normalized estimator $\hat\btheta_{\TL}(u_0)$ should converge in distribution to the standard normal in both of these regimes. The standard error $\hat{\SE}(\hat\theta_{\TL,j}(u_0))$ is computed as the square root of the variance estimator proposed in Corollary~\ref{coro:TL-linear-infer}. Each coordinate is standardized as
$$
\check\theta_j(u_0)
=
\frac{\hat\theta_{\TL,j}(u_0)-\theta_j(u_0)}
{\hat{\SE}\!\left(\hat\theta_{\TL,j}(u_0)\right)},
\qquad j=0,1,2,3.
$$
Figure~\ref{fig:TL_distn} displays the histograms of the standardized estimators based on $200$ Monte Carlo replications under the linear model. The four columns correspond to $\check\theta_0$–$\check\theta_3$, while the upper and lower rows represent the regimes $\rho_n\to0$ and $\rho_n\to\infty$, respectively. We further statistically test the normality using the Kolmogorov–Smirnov test, and the corresponding $p$-values are reported in the histogram legends. Across all panels, the empirical distributions closely resemble the standard normal law, providing strong visual support for the theoretical results. Moreover, all reported 
$p$-values exceed $0.05$, offering additional empirical evidence for the asymptotic 
normality of $\check\theta_j(u_0)$.
\begin{figure}
\begin{center}
\includegraphics[width=\textwidth]{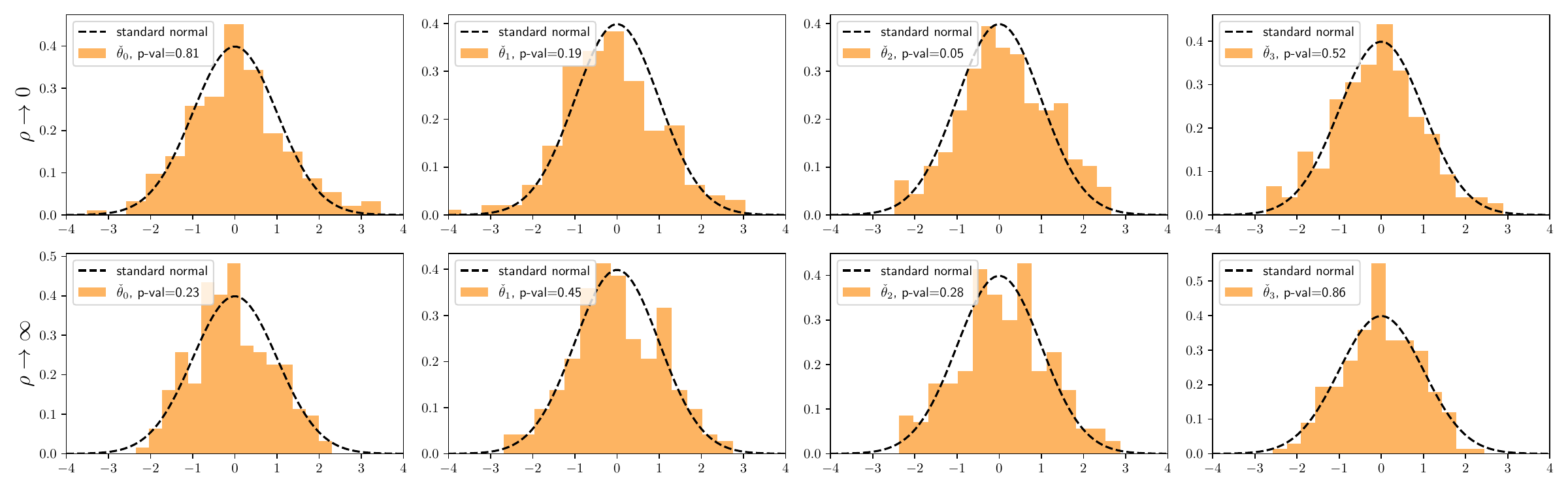}
\end{center}
\caption{The histogram of normalized estimators $\check\theta_j(u_0) = \tfrac{\hat\theta_{\TL,j}(u_0)-\theta_{j}(u_0)}{\hat{\SE}\left(\hat\theta_{\TL,j}(u_0)\right)}$ when  $\rho_n \to 0$ and $\infty$. 
 }\label{fig:TL_distn}
\end{figure}

\subsection{Phase transition in the rate of estimation}
\label{subsec:log-log}

In this subsection, we demonstrate the phase transition in the rate of convergence of $\hat \btheta_{\TL}(u_0)$. Recall that we established in Theorem~\ref{thm:rate-TL} that
\begin{equation}
    \label{eq:mse-sim}
    \mse(\hat\btheta_{\TL}(u_0)) 
    \lesssim 
    n_0^{-1} 
    \wedge 
    \max\left\{(K/\gamma)^{-2\beta}, \, (n/\gamma)^{-\frac{2\beta}{2\beta+1}}, \, n^{-1}\right\}.
\end{equation}
It follows immediately that, depending on the choice of $(K, n, n_0, \gamma, \beta)$, the rate of $\hat \btheta_{\TL}(u_0)$ transitions between different regimes. The goal of this subsection is to numerically illustrate this phase transition behavior by varying these parameters. We divide our presentation into three parts, depending on whether we vary $K$, $\gamma$, or $n$. To visualize the convergence rates, we present log–log plots with the logarithm of the varying parameter (i.e., $K$, $\gamma$, or $n$) on the $X$-axis and the logarithm of the MSE on the $Y$-axis. The slope of each segment in these plots can be interpreted as the convergence rate. Throughout the simulations, the shrinkage matrix $Q$ is set to its oracle value, and we fix the number of covariates to be $p=2$.

\medskip
\noindent
\textbf{Phase transition by varying $K$.} 
In this part, we vary $K$ to highlight its effect on the phase transition in the convergence rate of $\hat \btheta_{\TL}(u_0)$. We assume $\gamma \gg n^{-1/(2\beta)}$ in this setting. Under this regime, the MSE of $\hat \btheta_{\TL}(u_0)$ satisfies
$$
\mse(\hat\btheta_{\TL}(u_0)) \;\lesssim\;
\begin{cases}
n_0^{-1}, & K \lesssim \gamma\,n_0^{1/(2\beta)}, \\
\big(\tfrac{\gamma}{K}\big)^{2\beta}, & \gamma\,n_0^{1/(2\beta)} \lesssim K \lesssim \gamma^{\frac{2\beta}{2\beta+1}}\,n^{1/(2\beta+1)}, \\
\big(\tfrac{\gamma}{n}\big)^{\frac{2\beta}{2\beta+1}}, & K \gtrsim \gamma^{\frac{2\beta}{2\beta+1}}\,n^{1/(2\beta+1)}.
\end{cases}
$$
Since $\gamma \gg n^{-1/(2\beta)}$ implies $(\gamma/n)^{2\beta/(2\beta+1)} \gg n^{-1}$, the rate exhibits three distinct phases on a log–log scale with respect to $K$: 
(i) a flat region at level $n_0^{-1}$ for small $K$ (as the rate does not depend on $K$), 
(ii) a linear region with slope $-2\beta$, and 
(iii) another linear region with slope $-2\beta/(2\beta + 1)$ for large $K$.
As a numerical validation, we conduct simulations under the same data-generating process across linear, logistic, and Poisson models, this time varying $K$ while fixing $\bar n_S = 1500$, $\gamma=0.1$, and $n_0=30$. For computing $\hat \btheta_{\DVCM}(u_0)$, the bandwidth is chosen as described in Equation~\eqref{eqn:optimal_bw}. Figure~\ref{fig:MSE_K} presents the resulting log–log plot of $\log \mse$ against $\log K$. As expected, we observe three distinct linear phases in the plot, along with their corresponding empirical slopes, which align closely with the theoretical predictions of $-2\beta = -4$ and $-2\beta/(2\beta + 1) = -0.8$.
\begin{figure}
\begin{center}
\includegraphics[width=\textwidth]{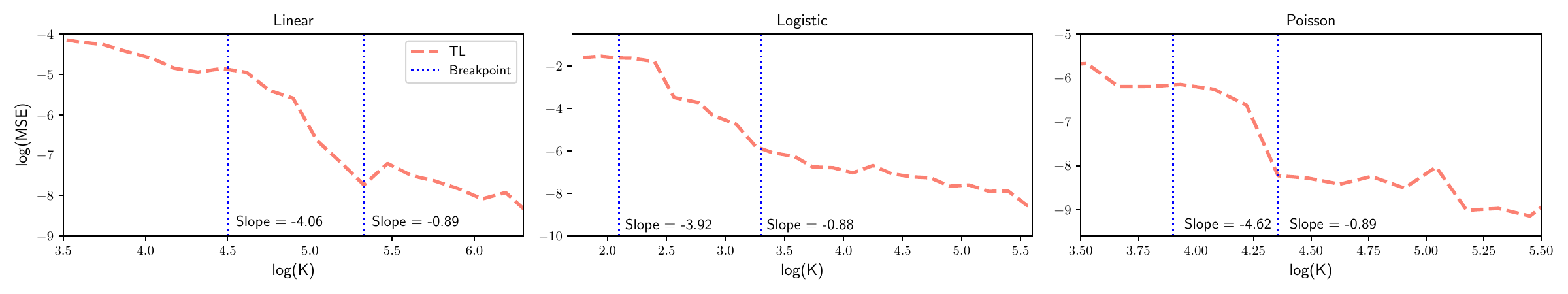}
\end{center}
\caption{
Log–log plot of MSE of $\hat \btheta_{\rm TL}$ as a function of $K$, while keeping $(\bar n_S, \gamma)$ fixed. Vertical dotted lines indicate empirical breakpoints that mark phase transitions in the convergence behavior. The left, middle, and right panels correspond to the linear, logistic, and Poisson models, respectively.
}
\label{fig:MSE_K}
\end{figure}

\medskip
\noindent
\textbf{Phase transition by varying $\gamma$.} We now illustrate the phase transition in the convergence rate of $\hat \btheta_{\TL}(u_0)$ by varying $\gamma$. Based on Theorem~\ref{thm:rate-TL}, the rate can be decomposed as follows:
\begin{equation}
\textstyle
\label{eq:mse_decomp_gamma}
\mse(\hat\btheta_{\TL}(u_0)) \;\lesssim\;
\begin{cases}
n^{-1}, & \gamma \lesssim n^{-1/(2\beta)}, \\[4pt]
n^{-\tfrac{2\beta}{2\beta+1}}\gamma^{\tfrac{2\beta}{2\beta+1}}, 
& n^{-1/(2\beta)} \lesssim \gamma \lesssim 
\min\!\Big\{K^{\frac{2\beta+1}{2\beta}}\,n^{-1/(2\beta)},\;
n\,n_0^{-\frac{2\beta+1}{2\beta}}\Big\}, \\[6pt]
n_0^{-1}, & 
\min\!\Big\{K^{\frac{2\beta+1}{2\beta}}\,n^{-1/(2\beta)},\;
n\,n_0^{-\frac{2\beta+1}{2\beta}}\Big\}\;\lesssim\; \gamma \lesssim K^{\frac{2\beta+1}{2\beta}}\,n^{-1/(2\beta)}, \\[6pt]
K^{-2\beta}\gamma^{2\beta}, & K^{\frac{2\beta+1}{2\beta}}\,n^{-1/(2\beta)} \lesssim \gamma \lesssim K\,n_0^{-1/(2\beta)}, \\[6pt]
n_0^{-1}, & \gamma \gtrsim K\,n_0^{-1/(2\beta)}.
\end{cases}
\end{equation}
In this simulation study, we assume $\bar n \gg n_0$, i.e., the target sample size is much smaller than the average source sample size. Under this regime, we have $K^{(2\beta+1)/2\beta}n^{-1/2\beta} \ll nn_0^{-(2\beta + 1)/2\beta}$, which eliminates the third phase in Equation~\eqref{eq:mse_decomp_gamma}. Consequently, the log–log plot (with $\log \gamma$ on the $X$-axis and $\log \mse(\hat \btheta_{\TL}(u_0))$ on the $Y$-axis) exhibits four distinct phases:(i) a flat region at level $n^{-1}$ for small $\gamma$;  
(ii) a linear growth with slope $2\beta/(2\beta + 1)$;  
(iii) a second linear growth with slope $2\beta$; and  
(iv) a flat region at level $n_0^{-1}$ for large $\gamma$.  
Figure~\ref{fig:MSE_gamma} illustrates this behavior. We set the average source sample size to $\bar n_S=600$, the target sample size to $n_0=30$, and use the oracle choice of $Q$. The number of domains is set to $K=12$ for the linear model and $K=10$ for the other settings. The functional coefficients are specified as $\theta_0(u) = \theta_1(u) = \tanh\big(8(u-0.2)\big)$. As predicted by the theory, the figure displays clear transitions across the four regimes, with the empirical slopes in the two linear phases closely matching the theoretical values $2\beta/(2\beta + 1)=0.8$ and $2\beta=4$ (for $\beta=2$).
\begin{figure}
\begin{center}
\includegraphics[width=\textwidth]{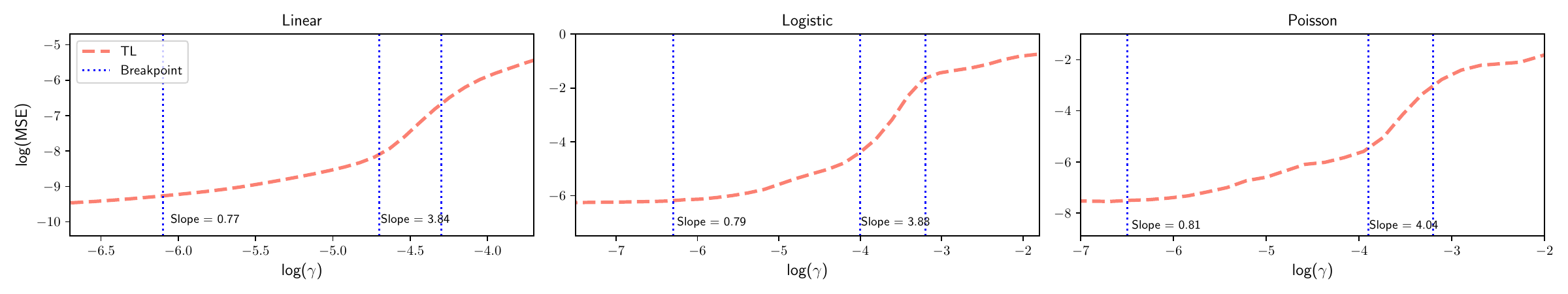}
\end{center}
\caption{
Log–log plot of MSE of $\hat \btheta_{\rm TL}$ as a function of $\gamma$, while keeping $(n, K)$ fixed. Vertical dotted lines indicate empirical breakpoints that mark phase transitions in the convergence behavior. The left, middle, and right panels correspond to the linear, logistic, and Poisson models, respectively.
}
\label{fig:MSE_gamma}
\end{figure}

\medskip
\noindent
\textbf{Phase transition by varying $\bar n_S$.} We now fix $(n_0, K, \gamma)$ and vary the average source sample size $\bar n_S$, so that the total source sample size is $n_S = K\bar n_S$ and the overall sample size is $n = n_0 + n_S$. Under the same condition $\gamma \gg n^{-1/(2\beta)}$, the MSE of $\hat \btheta_{\TL}(u_0)$ exhibits the following phase transition in its convergence rate:
\begin{equation*}
\textstyle
\mse(\hat\btheta_{\TL}(u_0))
\;\lesssim\;
\begin{cases}
n_0^{-1}, 
& n \,\lesssim\,  \gamma n_0^{\frac{2\beta+1}{2\beta}},\\[8pt]
\left(\dfrac{\gamma}{n}\right)^{\frac{2\beta}{2\beta+1}}, 
& \gamma n_0^{\frac{2\beta+1}{2\beta}}\,\lesssim\, n \,\lesssim\, \gamma^{-2\beta}K^{2\beta+1},\\[12pt]
n_0^{-1}\wedge \left(\dfrac{\gamma}{K}\right)^{2\beta}, 
& n \,\gtrsim\, \gamma^{-2\beta}K^{2\beta+1}.
\end{cases}
\end{equation*}
Consequently, on a log--log scale (with $\log n$ on the $X$-axis and $\log \mse(\hat\btheta_{\TL}(u_0))$ on the $Y$-axis), the curve exhibits three distinct phases:
(i) a flat region at level $n_0^{-1}$ for small $n$;
(ii) a linear regime with slope $-2\beta/(2\beta + 1)$; and
(iii) a second plateau at level $n_0^{-1}\wedge(\gamma/K)^{2\beta}$ for sufficiently large $n$.
Figure~\ref{fig:MSE_n} illustrates this behavior with $\gamma=0.1$ fixed while varying $\bar n_S$. We observe a short flat region followed by a linear regime whose empirical slope is close to $-2\beta/(2\beta + 1)=-0.8$ (for $\beta=2$). The third plateau emerges only when $n$ becomes extremely large, which is consistent with the transition scale $\gamma^{-2\beta}K^{2\beta+1}$. Across the three models, we use $K=10$ (linear), $K=3$ (logistic), and $K=2$ (Poisson), reflecting differing saturation behaviors in the third phase.
\begin{figure}
\centering
\includegraphics[width=\textwidth]{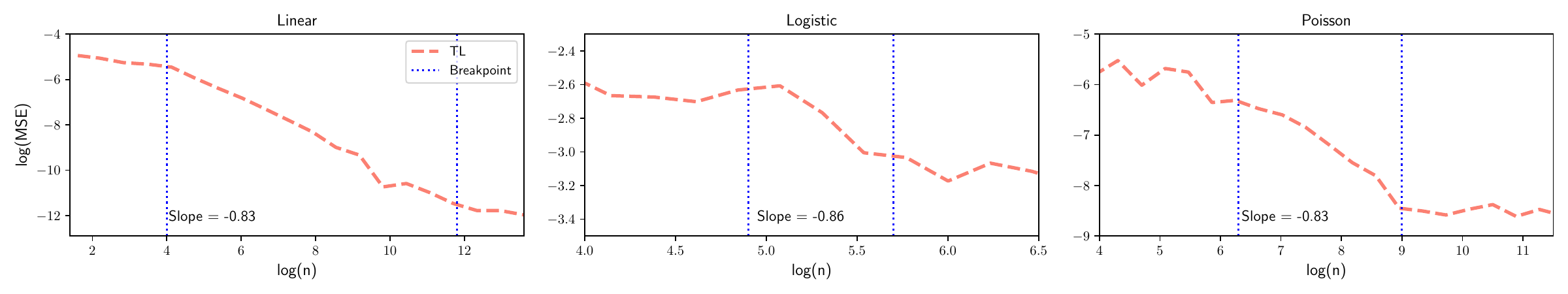}
\caption{
Log–log plot of MSE as a function of the sample size $n$. The vertical dotted lines represent the breakpoints for phase transition.
The left, middle, and right panels correspond to the linear, logistic, and Poisson models, respectively. 
}
\label{fig:MSE_n}
\end{figure}

\section{Real data analysis}
\label{sec:real-data}
In this section, we apply our methods to two real datasets to illustrate the performance of our proposed estimator. The first dataset  is \textbf{SLID-Ontario dataset} (Subsection \ref{subsec:SLID}); it is an economic dataset, where we want to predict a person's composite hourly wage using their demographic attributes. The second dataset is \textbf{US Adult Income} dataset (Subsection \ref{sec:adult}), where the goal is to predict whether a person's yearly wage is greater than 50,000 based on their various attributes. 
In both of these studies, we take $U$ (domain identifier) to be the \textit{years of employment}, as this typically determines an individual's base salary. 
For the simplicity of the implementation, we include two covariates in studies: i) $X_1$ is gender, which is a binary variable (1 if female, 0 if male), and ii) $X_2$ is years of education, as these two variables are known to affect the income quite significantly. Mathematically speaking, we fit the following (generalized) linear model: 
\begin{equation}
    g(\bbE (Y|\bX)) = \theta_0(U) + \theta_1(U)X_1 + \theta_2(U)X_2,
\end{equation}
where $g$ is a link function, $Y$ is the response variable (composite hourly wage in SLID-Ontario dataset, and indicator whether annual income is $> 50,000$ in adult income dataset).  Noticing that some values of $U$ are unrealistic outliers such as negative \textit{years of employment}, we remove all data points where $U$ falls outside the ``$3\sigma$'' region. The retention rates are 99.97\% (3996/3997) for SLID-Ontario and 99.53\% (48615/48842) for US Adult Income.


\subsection{Application 1: Survey of Labour and Income Dynamics in Ontario}
\label{subsec:SLID}
The first dataset we analyze is a public-use sample from the 1994 Survey of Labour and Income Dynamics in Ontario (SLID-Ontario) \cite{fox2015applied}. The dataset contains information on four attributes for 3997 individuals: \textit{age}, \textit{gender}, \textit{years of education}, and \textit{composite hourly wage}. In this application, we study the predictive relationship between the logarithm of the \textit{composite hourly wage} (the response variable $Y$) and other covariates. Following \cite{fox2015applied}, we use the log-transformed wage to mitigate non-normality. We employ the following domain-varying-coefficient model
\begin{equation}
\textstyle
\label{eqn:SLID-model}
    \bbE[Y \mid X, U] = \theta_0(U) + \theta_1(U)X_1 + \theta_2(U)X_2,
\end{equation}
where $X_1$ is a gender indicator and $X_2$ denotes years of education. We approximate the \textit{years of employment} $U$ by $U = \textit{age} - \textit{years of education} - 6$, assuming individuals begin schooling at age six and enter the workforce immediately after graduation. To ensure scale invariance and comparability across domains, we normalize $U$ via the min–max transformation
\begin{equation}
\textstyle
\label{eq:scale_transform}
U \leftarrow \frac{U - \min_{1 \le i \le n} U_i}
{\max_{1 \le i \le n} U_i - \min_{1 \le i \le n} U_i},
\end{equation}
so that the domain identifiers lie in $[0,1]$.

To construct source and target domains, we discretize $U$ into ten bins $[0,0.1], (0.1,0.2], \ldots, (0.9,1]$, and map each $U_i$ to the midpoint of its bin. Let $\cU^* = \{0.05, 0.15, \ldots, 0.95\}$ denote the set of bin midpoints. Each midpoint defines a domain identifier. Specifically, domain $j$ consists of all observations satisfying $U_i \in ((j-1)/10,\, j/10]$ and its associated identifier is $U_j^* = (j-0.5)/10$. For each $u_0 \in \cU^*$, we designate the corresponding domain as the target domain $\cD_0$ and treat the remaining observations as the source domain $\cD_S$. Our objective is to evaluate predictive performance on the target domain and assess whether borrowing information from nearby domains improves accuracy without inducing negative transfer. Towards that goal, we randomly split the target domain (of size $m$) into three equal parts $\{\cD_0^j\}_{j=1}^3$, where $\cD_0^1 \cup \cD_0^2$ is used for training and $\cD_0^3$ is reserved for testing. In our experiment, we compare performance of three estimators:
\begin{enumerate}
\item \emph{Target-only baseline} $\hat\btheta_{\LR}(u_0)$, fitted using $\cD_0^1 \cup \cD_0^2$.
\item \emph{Nonparametric DVCM} $\hat\btheta_{\DVCM}(u_0)$, computed using pooled data $\cD_S \cup \cD_0^1 \cup \cD_0^2$ via local polynomial regression.
\item \emph{Adaptive transfer-learning estimator} $\hat\btheta_{\TL}(u_0)$.
\end{enumerate}
To construct $\hat\btheta_{\TL}(u_0)$, we first compute a pilot estimator $\tilde\btheta_{\DVCM}(u_0)$ using $\cD_S \cup \cD_0^1$. We then fine-tune this estimate using $\cD_0^2$ by solving
\begin{equation*}
\textstyle
\hat\btheta_{\TL}(u_0)
=
\argmin_{\balpha}
\frac{1}{m/3}
\sum_{(X_{ki},Y_{ki}) \in \cD_0^2}
\left(Y_{ki} - X_{ki}^\top \balpha \right)^2
+
\|\balpha - \tilde\btheta_{\DVCM}(u_0)\|_{\hat Q}^2.
\end{equation*}
The data-splitting ensures independence between the pilot estimator and the refinement step, as discussed in Section~\ref{sec:theory}. The penalty matrix $\hat Q$ is estimated as in \eqref{eqn:Q-hat-in}. We then evaluate the predictive performance on the test set $\cD_0^3$ using
\begin{equation*}
\textstyle
\mse(u_0,\hat\btheta)
=
\frac{1}{m/3}
\sum_{(X_i,Y_i)\in \cD_0^3}
(X_i^\top \hat\btheta(u_0) - Y_i)^2.
\end{equation*}
To reduce variability due to random splitting, we repeat the procedure ten times and report the average MSE.

Each value in $\cU^*$ is treated in turn as the target domain, yielding a trajectory $\mse(u_0,\hat\btheta)$ across $u_0$. The three estimators $\{\hat\btheta_{\LR}, \hat\btheta_{\DVCM}, \hat\btheta_{\TL}\}$ produce the trajectories shown in Figure~\ref{fig:linear_real_data} (left panel), while the right panel displays the distribution of the unbinned $U_i$ values. Exact numerical MSE can be found in Appendix~\ref{sec:add-sim}.

Several patterns emerge from Figure~\ref{fig:linear_real_data}. First, the target-only estimator $\hat\btheta_{\LR}(u_0)$ exhibits large MSE near the right boundary ($u_0 \approx 1$), reflecting data scarcity in that region, as seen in the histogram. Second, $\hat\btheta_{\DVCM}(u_0)$ performs worse near the left boundary, where the target domain itself contains abundant data and therefore the target-only baseline itself is a strong predictor. Last but not least, the adaptive estimator $\hat\btheta_{\TL}(u_0)$ automatically tracks the better of the two estimators; when $u_0$ is small, it behaves similarly to $\hat\btheta_{\LR}(u_0)$; when $u_0$ is large, it aligns more closely with $\hat\btheta_{\DVCM}(u_0)$. Across all target domains, $\hat\btheta_{\TL}(u_0)$ achieves the most stable and favorable performance, corroborating both our theoretical results and simulation findings.

\subsection{Application 2: US Adult Income}
\label{sec:adult}
The US Adult Income dataset (also known as the ``Census Income’’ or ``Adult’’ dataset) contains demographic attributes and income levels for $48,842$ individuals from the 1994 US Census. It is widely used in the machine learning literature, particularly in studies of classification performance and algorithmic fairness (e.g., see \cite{yurochkin2020training,mukherjee2020two,yurochkin2021sensei}), where the goal is to predict whether an individual earns more than \$50,000 per year. To maintain consistency with the previous subsection, we select three covariates, age, gender, and years of education—to predict the binary response variable $Y$. We model the response using the generalized linear DVCM
\begin{equation}
\textstyle
\label{eqn:Adult-model}
    \Pr(Y=1 \mid X_1, X_2, U)
    =
    \frac{1}{1+\exp\left(-\{\theta_0(U) + \theta_1(U)X_1 + \theta_2(U)X_2\}\right)}.
\end{equation}
Here, $Y=1$ indicates annual income $\geq 50{,}000$, and $Y=0$ otherwise. As in Section~\ref{subsec:SLID}, we approximate years of employment $U$ and construct domain identifiers using the same binning and scaling procedure, resulting in $10$ domains indexed by $\cU^* = \{0.05, 0.15, \ldots, 0.95\}$. For a given $u_0 \in \cU^*$, the observations with $U_k^* = u_0$ form the target domain $\cD_0$ (of size $m$), while the remaining observations constitute the source domain $\cD_S$. The target domain is randomly split into training subsets $\cD_0^1$ and $\cD_0^2$ and a test subset $\cD_0^3$, each of size $m/3$. Here also, we compare three estimators: i) \textit{the target-only baseline} $\hat\btheta_{\GLR}(u_0)$ (constructed using $\cD_0^1\cup \cD_0^2$), ii) \textit{the non-parametric generalized DVCM} $\hat\btheta_{\DVCM}(u_0)$ (constructed using $\cD_S \cup \cD_0^1 \cup \cD_0^2$), and iii) \textit{our proposed transfer-learning estimator} $\hat\btheta_{\TL}(u_0)$, which, as in the previous subsection, constructed in two steps: first, we compute a pilot non-parametric DVCM estimator $\tilde\btheta_{\DVCM}(u_0)$ using $\cD_S \cup \cD_0^1$. We then refine this estimate on $\cD_0^2$ by solving
\begin{equation*}
\textstyle
\hat\btheta_{\TL}(u_0)
=
\argmin_{\balpha}
\frac{1}{m/3}
\sum_{(X_{ki},Y_{ki}) \in \cD_0^2}
\ell(X_{ki}^\top \balpha,Y_{ki})
+
\frac{1}{2}
\|\balpha - \tilde\btheta_{\DVCM}(u_0)\|_{\hat Q}^2,
\end{equation*}
where $\ell(\cdot,\cdot)$ denotes the cross-entropy loss and the penalty matrix $\hat Q$ is chosen as in Equation~\eqref{eqn:Q-hat-in}. To reduce variability due to random splitting, we repeat the procedure ten times and report the average cross-entropy loss on the test set $\cD_0^3$. Each value in $\cU^*$ is treated in turn as the target domain, yielding trajectories of predictive error across $u_0$. The three estimators $\{\hat\btheta_{\GLR}, \hat\btheta_{\DVCM}, \hat\btheta_{\TL}\}$ produce the curves shown in Figure~\ref{fig:logistic_real_data} (left panel), while the right panel displays the distribution of the unbinned $U_i$ values. Exact numerical results are provided in Appendix~\ref{sec:add-sim}.

The qualitative behavior mirrors that observed in Section~\ref{subsec:SLID}. The logistic regression estimator $\hat\btheta_{\GLR}(u_0)$ performs particularly well near the left endpoint ($u_0 \approx 0$), whereas in other regions the performance of the nonparametric estimator $\hat\btheta_{\DVCM}(u_0)$ is at par. The adaptive estimator $\hat\btheta_{\TL}(u_0)$ consistently aligns with the better-performing method across domains, highlighting its ability to automatically balance between pooling and target-only learning.

\begin{figure}
\begin{center}
\includegraphics[width=\textwidth]{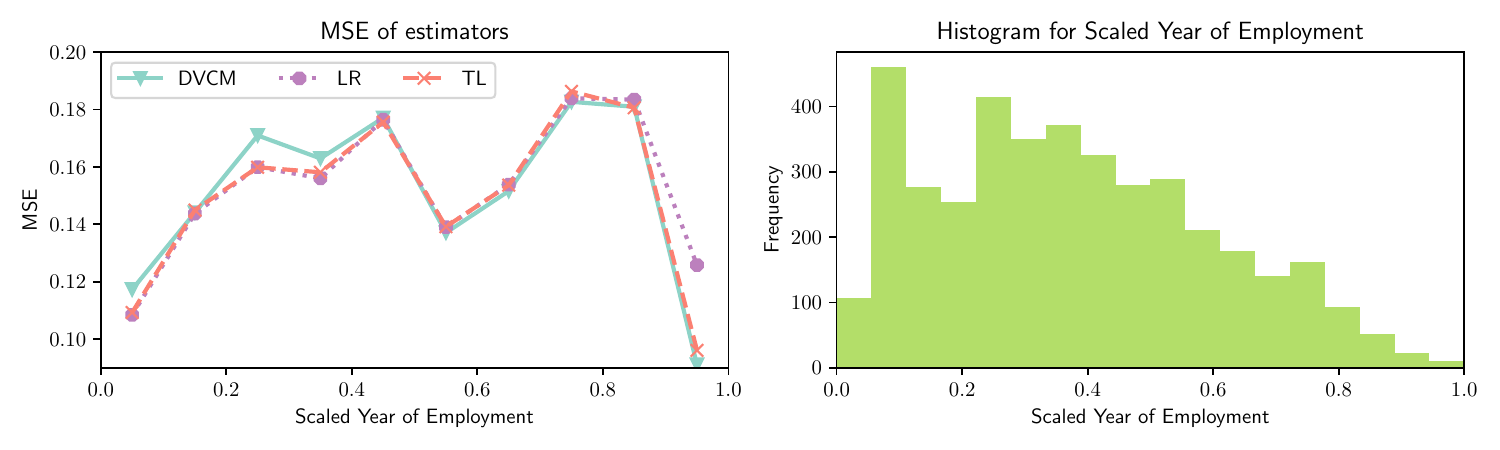}
\end{center}
\caption{Model evaluation on SLID-Ontario dataset. The MSEs of the linear regression, varying-coefficient model, and transfer learning estimator at different values of $u_0$ based on bandwidth of 0.2 are shown on the left. The histogram of the scaled year of employment, $U$, is shown on the right.
 }\label{fig:linear_real_data}
\end{figure}

\begin{figure}
\begin{center}
\includegraphics[width=\textwidth]{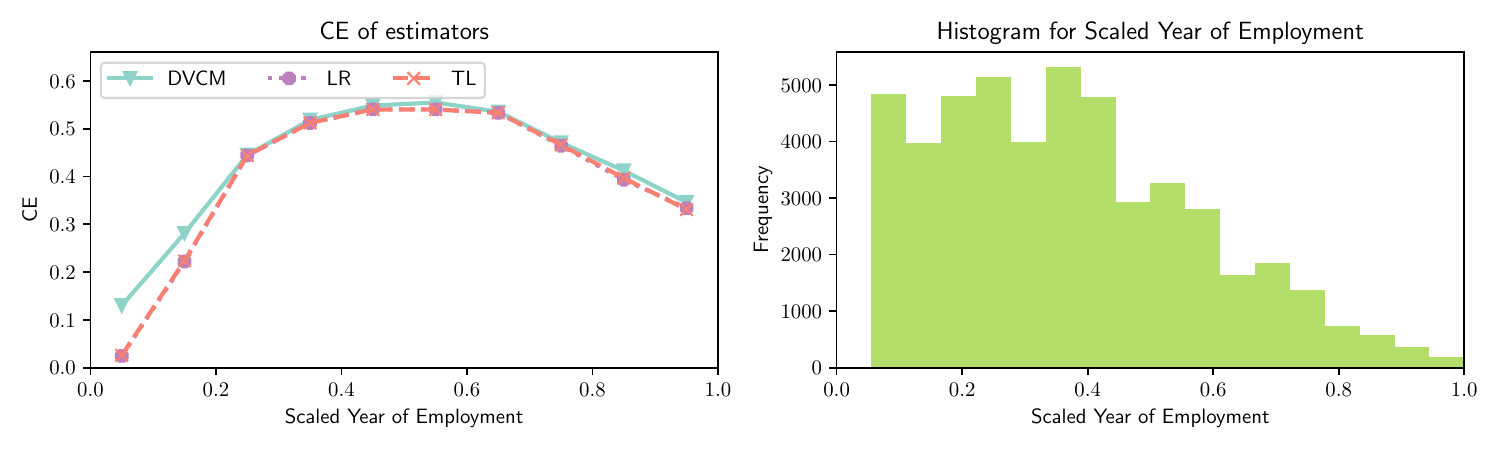}
\end{center}
\caption{Model evaluation on the US Adult Income dataset. The cross-entropy of the logistic regression, logistic-based varying-coefficient, and transfer-learning estimators at different values of $u_0$ based on bandwidth of 0.2 is shown on the left. The histogram of the scaled year of employment, $U$, is shown on the right.
 }\label{fig:logistic_real_data}
\end{figure}

\section{Conclusion and future work}\label{sec:conclusion}
We study multi-source transfer learning under \emph{posterior drift}, where the conditional relationship between response and covariates varies across environments indexed by a domain identifier \(U\). To capture this structured heterogeneity, we introduce a \emph{domain–varying coefficient model} (DVCM) and propose a two-step estimator that combines nonparametric pooling across source domains with a ridge-type fine-tuning step on the target domain. Our main contribution is a data-adaptive choice of the shrinkage matrix \(Q\) that provably prevents negative transfer: the resulting estimator automatically interpolates between target-only and pooled estimators and never incurs higher risk than the target-only baseline. We establish matching minimax upper and lower bounds for estimating \(\theta(u_0)\), revealing a phase transition governed by smoothness, domain dispersion, and the number of source environments. We further derive asymptotic normality with feasible variance estimation, enabling valid confidence intervals and hypothesis tests. Simulations and real-data experiments confirm that the procedure adaptively tracks the better of target-only and pooled learning across regimes. However, several interesting directions remain open for future investigation: 
\begin{enumerate}
    \item \textbf{Beyond linear models in $X$.} In this paper, we focused on models that are linear in $X$. A natural extension is to consider more flexible non-linear structures, such as single-index models $\bbE[Y \mid X, U] = g(X^\top \theta(U))$ for unknown link $g$, or additive models of the form $\bbE[Y \mid X, U] = \sum_{j = 1}^d \theta_j(U)f_j(X)$, where both $\theta_j(\cdot)$ and $f_j(\cdot)$ are unknown. An important theoretical question is whether the negative-transfer robustness and adaptive minimax optimality established here continue to hold under suitable reformulations of the estimation procedure for such broader nonparametric classes.

    \item \textbf{Modern machine learning estimators for the $U$-varying component.} We estimated the nonparametric components via local polynomial regression. A promising direction is to investigate neural network or transformer-based estimators for learning $\theta(U)$, especially when $U$ is multi-dimensional. As observed in recent work (e.g., see \cite{schmidt2020nonparametric,kohler2021rate}), neural network estimators can adapt to compositional structures and mitigate the curse of dimensionality. Understanding whether similar adaptivity and phase-transition phenomena persist under modern deep-learning architectures remains an important open problem.

    \item \textbf{High-dimensional covariates and structured sparsity.}
Our analysis assumes that the dimension of $X$ is fixed. In growing/high-dimensional settings, variable selection becomes essential, particularly when only a subset of covariates is informative, and the sparsity
pattern may vary across domains. One natural extension is to incorporate sparsity-inducing penalties
(e.g., $\ell_1$-regularization or structured group penalties) into the domain-adaptive framework. Alternatively, domain heterogeneity may be captured through a low-rank latent factor structure. Developing adaptive procedures that combine transfer learning with sparsity or low-rank structure on the covariates would substantially broaden the scope of the model.

\item \textbf{Bayesian formulations and adaptive borrowing.} A Bayesian perspective offers another appealing direction. One may model the domain identifiers $U_i$ as draws from a prior distribution (where the prior variance encodes cross-domain similarity), and the coefficient function $\theta(\cdot)$ is generated from a nonparametric prior (e.g., Gaussian process, spline-based prior, or Bayesian neural network prior) that models the smoothness. An important theoretical question is how to modify the likelihood equation appropriately so that the resulting posterior achieves adaptive contraction rates that match the minimax frequentist rates derived here. If so, one may use this approach for uncertainty quantification from a Bayesian perspective. 
\end{enumerate}

\begin{appendix}
\begin{center}
    {\large \textbf{Appendix}}
\end{center}
Throughout the theoretical analysis in the Appendix, we assume a target sample-splitting procedure. Specifically, we observe a total of $2n_0$ target samples, which are partitioned into two equal subsets satisfying $|\cI_0| = |\cI_0^*| = n_0$. The subset $\cI_0$ is used to construct the initial nonparametric estimator, while $\cI_0^*$ is reserved for the subsequent fine-tuning step. Under this setup, we first analyze the estimator in the linear response setting. The nonparametric estimator $\hat\btheta_{\DVCM}(u_0)$ is defined as
\begin{equation}
\textstyle
\hat\btheta_{\DVCM}(u_0)
=
\bA_l \cdot
\argmin_{\balpha\in \bbR^{(l+1)p}}
\sum_{k=0}^K
\sum_{i\in \cI_k}
\left[
Y_{ki}
-
\left(\bPhi_l\!\left(\frac{U_k-u_0}{h}\right)^\top
\otimes
X_{ki}^\top\right)\balpha
\right]^2
W\!\left(\frac{U_k-u_0}{h}\right),
\end{equation}
where $\bPhi_l(x) = \big(1, x, x^2/2!, \ldots, x^l/l!\big)^\top$ denotes the $l$-th order polynomial feature map, $\otimes$ is the Kronecker product, and $\bA_l = [I_{p}, \mathbf 0_{p\times lp}]$ extracts the first $p$ coordinates of the minimizer.

The fine-tuning estimator constructed from $\cI_0^*$ is then given by
\begin{align}
\hat\btheta_{\TL}(u_0)
&=
\argmin_{\balpha\in\bbR^p}
\frac{1}{2n_0}
\sum_{i\in \cI_0^*}
\big(
Y_{0i}-X_{0i}^\top \balpha
\big)^2
+
\frac{1}{2}
\|\balpha-\hat\btheta_{\DVCM}(u_0)\|_Q^2
\notag \\
&=
\left(
\frac{1}{n_0}\bX_0^\top \bX_0 + Q
\right)^{-1}
\left(
\frac{1}{n_0}\bX_0^\top \by_0
+
Q\,\hat\btheta_{\DVCM}(u_0)
\right).
\end{align}

Similarly, for the generalized linear response model, the Step~I estimator is defined as
\begin{equation}
\hat\btheta_{\GDVCM}(u_0)
=
\bA_{l}\cdot
\argmin_{\balpha\in \bbR^{(l+1)p}}
\sum_{k = 0}^K
\sum_{i \in \cI_k}
\ell\!\big(Z_{ki}^\top \alpha, Y_{ki}\big)
\,W\!\left(\frac{U_k-u_0}{h}\right),
\end{equation}
while the Step~II fine-tuning estimator based on $\cI_0^*$ is
\begin{equation}
\hat \btheta_{\TL}(u_0)
=
\argmin_{\balpha}
\frac{1}{n_0}
\sum_{i \in \cI_0^*}
\ell\!\big(X_{0i}^\top \alpha, Y_{0i}\big)
+
\frac{1}{2}
\|\balpha - \hat\btheta_{\GDVCM}(u_0)\|_Q^2 \, .
\end{equation}

This appendix, centered on the estimators defined above, is organized as follows.
Section~\ref{sec:aux-lem} collects auxiliary lemmas used throughout the paper. Section~\ref{sec:proof-main} contains proofs of the main theorems, propositions, and lemmas from the main body. Section~\ref{sec:proof-aux-lem} provides the proofs of the auxiliary lemmas. 

\section{Auxiliary lemmas}\label{sec:aux-lem}
In this section we collect auxiliary lemmas that support the main results. Section~\ref{sec:aux1} gathers tools for the nonasymptotic analysis: Lemmas~\ref{lem:gap}–\ref{lem:mseA-ub} serve as building blocks for Proposition~\ref{prop:rate-LR-VCM}. Section~\ref{sec:aux2} contains asymptotic tools: Lemma~\ref{lem:cond-normality-GDVCM} underpins Theorem~\ref{thm:TL-GLM-rate-Op} and Proposition~\ref{prop:asymp-GLM-VCM};  Lemma~\ref{lem:cond-normality-DVCM} is utilized to prove Proposition~\ref{prop:asymp-LR-VCM}; Lemma~\ref{lem:Dconv-GLM-TF-0-infty} is used in the proof of Theorem~\ref{thm:TL-GLM-rate-Op}; and Lemma~\ref{lem:TL-GLM-normal} aids Theorem~\ref{thm:TL-GLM-infer}.

\subsection{Auxiliary lemmas for nonasymptotic analysis}
\label{sec:aux1}
\noindent
Recall the linear DVCM estimator is
\begin{equation}
\textstyle
    \hat\btheta_{\DVCM}(u_0) =  \bA_{l}(\bZ^\top\bW\bZ)^{-1}\bZ^\top\bW\by \in \bbR^{p}\,.
\end{equation}
Let us first recall some basic notations: 
\begin{align*}
    & \bZ^\top = \left[Z_{01}, Z_{02}, \ldots, Z_{0n_0},\ldots, Z_{K1}, Z_{K2}, \ldots, Z_{Kn_K}\right],\\
    & \bW = S_h^{-1} \diag\bigg\{\underbrace{W\left(\frac{U_0-u_0}{h}\right),
    \ldots, W\left(\frac{U_0-u_0}{h}\right)}_{n_0 \text{ identical terms}},\ldots,\underbrace{W\left(\frac{U_K-u_0}{h}\right),
    \ldots, W\left(\frac{U_K-u_0}{h}\right)}_{n_K \text{ identical terms}}\bigg\},\\
    & \by^\top = \left[\by_0^\top,\ldots,\by_K^\top\right],
\end{align*}
where the normalizing constant $S_h=\sum_{k=0}^K\sum_{i\in\cI_k}W\left(\frac{U_k-u_0}{h}\right)$, and $W$ is a uniform kernel (Assumption \ref{Assump:Unif-Kernel}). 
The following lemma provides a finite sample concentration inequality on the distance of $u_0$ from its nearest and furthest neighbors, i.e., $d_{(1)}(u_0)$ and  $d_{(K)}(u_0)$. In particular, it shows that $d_{(1)}(u_0)$ and  $d_{(K)}(u_0)$ are of the order $\gamma/K$ and $\gamma$.
\begin{lemma} \label{lem:gap}
Under Assumption \ref{Assump:VCM-UX} (a), the following bounds hold:
\begin{itemize}
    \item[(1)]
    \[\left[ 1- \frac{2a_0t}{K} \right]^K\indc\left\{t \leq K/(2a_0)\right\} \leq \Pr\left(K d_{(1)}(u_0) > \gamma t\right)\leq\left[ 1- \frac{2a_0't}{K} \right]^K \indc\left\{t \leq K/(2a_0')\right\},\]
    \[C_1 \left(K/\gamma\right)^{-2\beta} \leq \bbE \left[d_{(1)}^{2\beta}(u_0)\right] \leq C_2 \left(K/\gamma\right)^{-2\beta}.\]
    \item[(2)]
    \[\bigg\{1 - \left[\frac{2a_0t}{K}\right]^K\bigg\}\indc\left\{t \leq K/(2a_0)\right\} \leq \Pr\left(Kd_{(K)}(u_0) > \gamma t\right) \leq \bigg\{1 - \left[\frac{2a_0't}{K}\right]^K\bigg\}\indc\left\{t \leq K/(2a_0')\right\},\] 
    \[C_3 \gamma^{2\beta} \leq \bbE \left[d_{(K)}^{2\beta}(u_0)\right] \leq C_4 \gamma^{2\beta}.\]
\end{itemize}
\end{lemma}
The proof on part (1) is in Appendix \ref{sec:proof:order-stat-d1} and the proof on part (2) is in Appendix \ref{sec:proof:order-stat-dK}. The following Lemma shows that under our assumptions, the random variables $Z_{ki}$ are upper bounded almost surely under certain constraints.
\begin{lemma}
\label{lem:bdd-Z}
Under Assumptions \ref{Assump:VCM-UX} and \ref{Assump:Unif-Kernel}, it holds that
$$
\indc\{|U_k-u_0|\le h\}\left\Vert Z_{ki}\right\Vert_2 \leq 2\,.
$$
\end{lemma}
The proof is in Appendix \ref{proof:lem:bdd-Z}. 

\begin{lemma} \label{lem:var-bd}
     Let Assumptions \ref{Assump:VCM-UX}--\ref{Assump:Unif-Kernel} hold and $h$ be such that $h \le |\cU|$. Then there exists a constant $C>0$ such that
     \[\bbE\left[S_h^{-1}\mid d_{(1)}(u_0) \le h\right] \leq C\frac{\gamma}{nh}\,. \]
\end{lemma}
The proof is in Appendix \ref{proof:lem:var-bd}. Recall that $\Gamma = \{U_k: 0 \leq k \leq K\}\cup \{X_{ki}: 0 \leq k \leq K, i \in \cI_k\}$ is the set of all covariates. The following lemma shows that the $\mse_A$ (conditioning on $\Gamma$) of the DVCM estimator is of order $O(h^{2\beta} + S_h^{-1})$.

\begin{lemma}\label{lem:mseA-ub} 
Under Assumptions \ref{Assump:VCM-SN}-- \ref{Assump:Unif-Kernel}, the following upper bound holds:
\begin{equation}
    \begin{split}
        & \bbE \left[\Vert\hat\btheta_{\DVCM}(u_0)-\btheta(u_0)\Vert_A^2\middle | \Gamma\right] \leq q_1^2 h^{2\beta} + q_2 S_h^{-1},
    \end{split}
\end{equation}
where $q_1,q_2>0$ are constants, and $S_h = (1/2)\sum_{k=0}^K n_k\indc\left\{|U_k-u_0|\leq  h\right\}$.
\end{lemma}
The proof of this lemma is in Appendix \ref{proof:lem:mseA-ub}.

\subsection{Auxiliary lemmas for asymptotic analysis}\label{sec:aux2}
The next lemma provides asymptotic expressions for the conditional mean and variance of the quantities \(\Delta_k\) and \(\Lambda_k\) in the GDVCM setting. These expansions serve as essential building blocks for the subsequent asymptotic analysis.
\begin{lemma}\label{lem:cond-normality-GDVCM}
Under Assumptions~\ref{Assump:VCM-SN} \ref{Assump:VCM-UX-Generalized}, \ref{Assump:balance-sample}, and \ref{Assump:Unif-Kernel}, define for $t_k = (u_k-u_0)/h$, 
\[
\bar\btheta(u_0)^\top = \big[\btheta(u_0)^\top,\, h\btheta'(u_0)^\top,\, \ldots,\, h^l\btheta^{(l)}(u_0)^\top\big],
\]
\begin{align*}
    \Delta_k &= n_k^{-1/2}\sum_{i\in\cI_k} s_1\big(Z_{ki}^\top \bar\btheta(u_0),Y_{ki}\big)\, Z_{ki} \, W(t_k), \\
    \Lambda_k &= n_k^{-1}\sum_{i\in\cI_k} s_2\big(Z_{ki}^\top \bar\btheta(u_0),Y_{ki}\big)\, Z_{ki}Z_{ki}^\top \, W(t_k),
\end{align*}
where $Z_{ki} = \bPhi_l(t_k) \otimes X_{ki}$ and $s_j(\eta,y) = \partial^j \ell(\eta,y) / \partial\eta^j$. Then, for any $k\in\{0\}\cup[K]$ and sufficiently small $h>0$, and some $\tilde u_k$ satisfying $|\tilde u_k - u_0| \le h$
\begin{align*}
    n_k^{-1/2}\,\bbE[\Delta_k \mid U_k = u_k] &=
   \bPhi_l(t_k)^{\otimes 2}\otimes \Psi(u_k)\,
  \bA_l^\top \frac{\btheta^{(l)}(u_0)-\btheta^{(l)}(\tilde u_k)}{l!}\,(t_k h)^{l}\,W(t_k)\{1+o(1)\} \lesssim h^\beta\\
    \Var[\Delta_k \mid U_k = u_k] &= \nu(u_k) \, \bPhi_l(t_k)^{\otimes 2} \otimes \Psi(u_k) \, W(t_k)^2 + O(h^{\beta}), \\
    \mathbb{E}[\Lambda_k \mid U_k = u_k] &= \bPhi_l(t_k)^{\otimes 2} \otimes \Psi(u_k) \, W(t_k) + O(h^{\beta}).
\end{align*}
\end{lemma}
The proof of this lemma is found in Appendix \ref{proof:lem:cond-normality-GDVCM}. The next lemma (on linear DVCM) is an immediate specialization of Lemma~\ref{lem:cond-normality-GDVCM} (on generalized DVCM) obtained by choosing the quadratic loss \(\ell(\eta,y)=\tfrac12(\eta-y)^2\). In this case \(s_2\equiv1\) and \(\Psi(u)=\bbE[XX^\top\mid U=u]\); moreover the GLR scale function \(\nu(u)\) is replaced by the noise variance \(\sigma^2(u)\). Hence all conclusions of Lemma~\ref{lem:cond-normality-GDVCM} remain the same, with the only substitution \(\sigma^2(\cdot)=\nu(\cdot)\).
\begin{lemma}\label{lem:cond-normality-DVCM}
Under Assumptions~\ref{Assump:VCM-SN}, \ref{Assump:VCM-UX-new}, \ref{Assump:balance-sample}, and \ref{Assump:Unif-Kernel}, define for $t_k = (u_k-u_0)/h$, 
\[
\bar\btheta(u_0)^\top = \big[\btheta(u_0)^\top,\, h\btheta'(u_0)^\top,\, \ldots,\, h^l\btheta^{(l)}(u_0)^\top\big],
\]
\begin{align*}
    \Delta_k &= n_k^{-1/2}\sum_{i\in\cI_k} s_1\big(Z_{ki}^\top \bar\btheta(u_0),Y_{ki}\big)\, Z_{ki} \, W(t_k), \\
    \Lambda_k &= n_k^{-1}\sum_{i\in\cI_k} s_2\big(Z_{ki}^\top \bar\btheta(u_0),Y_{ki}\big)\, Z_{ki}Z_{ki}^\top \, W(t_k),
\end{align*}
where $Z_{ki} = \bPhi_l(t_k) \otimes X_{ki}$ and $s_j(\eta,y) = \partial^j \ell(\eta,y) / \partial\eta^j$. Then, for any $k\in\{0\}\cup[K]$ and sufficiently small $h>0$, and some $\tilde u_k$ satisfying $|\tilde u_k - u_0| \le h$
\begin{align*}
    n_k^{-1/2}\,\bbE[\Delta_k \mid U_k = u_k] &=
   \bPhi_l(t_k)^{\otimes 2}\otimes \Psi(u_k)\,
  \bA_l^\top \frac{\btheta^{(l)}(u_0)-\btheta^{(l)}(\tilde u_k)}{l!}\,(t_k h)^{l}\,W(t_k)\{1+o(1)\} \lesssim h^\beta\\
    \Var[\Delta_k \mid U_k = u_k] &= \sigma^2(u_k) \, \bPhi_l(t_k)^{\otimes 2} \otimes \Psi(u_k) \, W(t_k)^2 + O(h^{\beta}), \\
    \mathbb{E}[\Lambda_k \mid U_k = u_k] &= \bPhi_l(t_k)^{\otimes 2} \otimes \Psi(u_k) \, W(t_k) + O(h^{\beta}).
\end{align*}
\end{lemma}
\begin{proof}
    This corollary on linear DVCM is a special case of Lemma \ref{lem:cond-normality-GDVCM} (on GDVCM). Therefore, Lemma \ref{lem:cond-normality-GDVCM} may be directly applied, with the only difference being variance function $\nu(\cdot)$ replaced by $\sigma^2(\cdot)$.
\end{proof}

The next proposition establishes the asymptotic distribution of the two base estimators  $\hat\btheta_{\LR}(u_0)$ and $\hat\btheta_{\DVCM}(u_0)$.

\begin{proposition}
\label{prop:asymp-LR-VCM}
Under Assumptions~\ref{Assump:VCM-SN}, \ref{Assump:VCM-UX-new}, \ref{Assump:balance-sample}, and \ref{Assump:Unif-Kernel}, the target-only estimator $\hat\btheta_{\LR}(u_0)$ satisfies: 
$$
\sqrt{n_0}\left(\hat\btheta_{\LR}(u_0) - \btheta(u_0) \right) 
    \xrightarrow[]{d} \cN\left( 0, \Omega_{\LR}(u_0)\right), 
    \quad \Omega_{\LR}(u_0) = \sigma^2(u_0)\,\Psi(u_0)^{-1} \,.
$$ 
Furthermore, if $h$ is such that $d_{(1)}(u_0) \le h \le d_{(K)}(u_0)$ and 
\(\frac{\gamma}{K} \ll h \lesssim \Big(\frac{\gamma}{n}\Big)^{\frac{1}{2\beta+1}}\), then $\hat\btheta_{\DVCM}(u_0)$ satisfies: 
$$
  \sqrt{\frac{nh}{\gamma}}\left(\hat\btheta_{\DVCM}(u_0) - \btheta(u_0) - \bb_{\DVCM}(u_0)\right) 
    \xrightarrow[]{d} \cN\left( 0, \Omega_{\DVCM}(u_0)\right) \,,
$$
 where
    \begin{align*}
    \bb_{\DVCM}(u_0) & \lesssim h^\beta, \\
    \Omega_{\DVCM}(u_0) &= \sigma^2(u_0)
    \left[\zeta_{0,1}^{-1}\,\zeta_{0,2}\,\zeta_{0,1}^{-1}\right]_{1,1} \Psi(u_0)^{-1}, \\
    \zeta_{r,s} & = \int \bPhi_l(t)^{\otimes 2} t^r W^s(t) f\Big(\tfrac{u_0 - u^* + ht}{\gamma}\Big)\,dt \,,
    \end{align*}
$u^*$ is same as defined in Assumption \ref{Assump:VCM-UX}. 
\end{proposition}

\begin{proposition}
\label{prop:asymp-GLM-VCM}
Under Assumptions~\ref{Assump:VCM-SN}, \ref{Assump:VCM-UX-Generalized}, \ref{Assump:balance-sample}, and \ref{Assump:Unif-Kernel}, the target-only estimator $\hat\btheta_{\GLR}(u_0)$ satisfies: 
$$
\sqrt{n_0}\left(\hat\btheta_{\GLR}(u_0) - \btheta(u_0) \right) 
    \xrightarrow[]{d} \cN\left( 0, \Omega_{\GLR}(u_0)\right), 
    \quad \Omega_{\GLR}(u_0) = \nu(u_0)\,\Psi(u_0)^{-1} \,.
$$
Furthermore, if $h$ is such that $d_{(1)}(u_0) \le h \le d_{(K)}(u_0)$ and 
\(\frac{\gamma}{K} \ll h \lesssim \Big(\frac{\gamma}{n}\Big)^{\frac{1}{2\beta+1}}\),
then $\hat\btheta_{\GDVCM}(u_0)$ satisfies: 
$$
  \sqrt{\frac{nh}{\gamma}}\left(\hat\btheta_{\GDVCM}(u_0) - \btheta(u_0) - \bb_{\GDVCM}(u_0)\right) 
    \xrightarrow[]{d} \cN\left( 0, \Omega_{\GDVCM}(u_0)\right) \,,
$$
 where
    \begin{align*}
    \bb_{\GDVCM}(u_0) &\lesssim h^\beta, \\
    \Omega_{\GDVCM}(u_0) &= \nu(u_0)
    \left[\zeta_{0,1}^{-1}\,\zeta_{0,2}\,\zeta_{0,1}^{-1}\right]_{1,1} \Psi(u_0)^{-1}, \\
    \zeta_{r,s} & = \int \bPhi_l(t)^{\otimes 2} t^r W^s(t)f\Big(\tfrac{u_0 - u^* + ht}{\gamma}\Big)\,dt \,,
    \end{align*}
$u^*$ is same as defined in Assumption \ref{Assump:VCM-UX}. 
\end{proposition}
Note that this result extends Proposition~\ref{prop:asymp-LR-VCM}, with
\(\sigma^2(\cdot)\) replaced by \(\nu(\cdot)\) and with \(\Psi(\cdot)\) redefined accordingly. The proof is given in Appendix~\ref{proof:prop:asymp-GLM-VCM}.

The next lemma establishes, within the GDVCM framework, that if the shrinkage matrix \(Q\) is chosen to be of the same order as \(\rho^2\) (i.e., \(Q\asymp_p \rho^2 I\)) where \(\rho:=r_{\GLR}/r_{\GDVCM}\) is the ratio of the convergence rates of the target-only GLR estimator and the GDVCM estimator, then the TL estimator \(\hat\btheta_{\TL}(u_0)\) adapts to the faster procedure: specifically,
\[
\|\hat\btheta_{\TL}(u_0)-\btheta(u_0)\|=O_p\big(r_{\TL}\big),\qquad
r_{\TL}:=r_{\GLR}\wedge r_{\GDVCM}.
\]
Moreover, if the bandwidth \(h\) satisfies the standard localization conditions, then \(\hat\btheta_{\TL}(u_0)\) attains the same asymptotic distribution as the faster estimator in Proposition~\ref{prop:asymp-GLM-VCM}: if \(\rho\to0\), 
\(\sqrt{n_0}\big(\hat\btheta_{\TL}(u_0)-\btheta(u_0)\big)\to_d\cN\big(0,\Omega_{\GLR}(u_0)\big)\);
if \(\rho\to\infty\), 
\(\sqrt{nh/\gamma}\big(\hat\btheta_{\TL}(u_0)-\btheta(u_0)-\bb_{\GDVCM}(u_0)\big)\to_d\cN\big(0,\Omega_{\GDVCM}(u_0)\big)\).
\begin{lemma}\label{lem:Dconv-GLM-TF-0-infty}
 Let Assumptions~\ref{Assump:VCM-SN} \ref{Assump:VCM-UX-Generalized}, \ref{Assump:balance-sample}, and \ref{Assump:Unif-Kernel} hold. Let the rates $r_{\GLR}$ and $r_{\GDVCM}$ be such that
\[\hat\btheta_{\GLR}(u_0) - \btheta(u_0) =O_p(r_{\GLR}), \quad \hat\btheta_{\GDVCM}(u_0) - \btheta(u_0)=O_p(r_{\GDVCM}).\]
Moreover, let \(Q\) be a positive-definite matrix satisfying for $\rho = r_{\GLR}/r_{\GDVCM}$, 
\[
c \rho^2 \le \lambda_{\min}(Q) \le \lambda_{\max}(Q) \le C \rho^2 \quad\text{for constants } C > c > 0,\quad \text{ w.p.} \to 1.
\]
Then
\[
\big\|\hat\btheta_{\TL}(u_0)-\btheta(u_0)\big\|_2
= O_p(r_{\GLR}\wedge r_{\GDVCM})\,.
\]
Furthermore, suppose \(h\) satisfies the same conditions in Proposition~\ref{prop:asymp-GLM-VCM}, then:
\begin{align*}
&\text{If } \rho \to 0,\quad  \sqrt{n_0}\left(\hat\btheta_{\TL}(u_0) - \btheta(u_0) \right) \xrightarrow[]{d} \mathcal{N}\left( 0, \Omega_{\GLR}(u_0)\right),\\
&\text{If } \rho \to \infty,\quad   \sqrt{\frac{nh}{\gamma}}\left(\hat\btheta_{\TL}(u_0) - \btheta(u_0) - \bb_{\GDVCM}(u_0)\right) \xrightarrow[]{d} \mathcal{N}\left( 0, \Omega_{\GDVCM}(u_0)\right),
\end{align*}
where \(\Omega_{\GLR}(u_0)\), \(\bb_{\GDVCM}(u_0)\), and \(\Omega_{\GDVCM}(u_0)\) are defined in Proposition~\ref{prop:asymp-GLM-VCM}.
\end{lemma}
The proof is given in Appendix~\ref{proof:lem:Dconv-GLM-TF-0-infty}. The next lemma, which is under the GDVCM framework, states that if we choose $h$ to be small enough, then the TL estimator achieves asymptotical normality with mean 0.
\begin{lemma}\label{lem:TL-GLM-normal}
Under the same conditions of Lemma~\ref{lem:Dconv-GLM-TF-0-infty}, suppose that the bandwidth additionally satisfies
\[
\frac{nh^{2\beta+1}}{\gamma} \to 0\,.
\]
Then
\begin{align*}
    &\text{If } \rho \to 0,\quad  \sqrt{n_0}\left(\hat\btheta_{\TL}(u_0) - \btheta(u_0) \right) \xrightarrow[]{d} \cN\left( 0, \Omega_{\GLR}(u_0)\right),\\
    &\text{If } \rho \to \infty,\quad  \sqrt{\frac{nh}{\gamma}}\left(\hat\btheta_{\TL}(u_0) - \btheta(u_0)\right) \xrightarrow[]{d} \cN\left( 0, \Omega_{\GDVCM}(u_0)\right),
\end{align*}
where $\Omega_{\GLR}(u_0)$, $\bb_{\GDVCM}(u_0)$, and $\Omega_{\GDVCM}(u_0)$ are defined in Proposition~\ref{prop:asymp-GLM-VCM}.
\end{lemma}
\begin{proof}
    This lemma follows directly from Lemma~\ref{lem:Dconv-GLM-TF-0-infty} by choosing $h$ such that $\sqrt{nh/\gamma}\,\bb_{\GDVCM} \to 0$. Note that by Proposition \ref{prop:asymp-GLM-VCM} the bias $\bb_{\GDVCM}$ is of order $O(h^{\beta})$, so this is equivalent to requiring $nh^{2\beta+1}/\gamma \to 0$. 
\end{proof}

\section{Proofs of main theorems}
\label{sec:proof-main}
\allowdisplaybreaks


\subsection{Proof of Proposition \ref{prop:rate-LR-VCM}} \label{proof:prop:rate-LR-VCM}



\begin{proof}
Recall that, we choose $h$ according to Equation \eqref{eqn:optimal_bw}: 
$$
h = \med\left\{e_0 (n/\gamma)^{-\frac{1}{2\beta+1}}, d_{(1)}(u_0), d_{(K)}(u_0)\right\} \,.
$$
Depending on the distribution of $\{u_0,\dots, u_K\}$, any one of the three elements can be chosen as $h$. Based on this, we divide our analysis into three disjoint events: 
\begin{align*}
      \cE_1 & = \left\{e_0 (n/\gamma)^{-\frac{1}{2\beta+1}} < d_{(1)}(u_0)\right\} \,, \\
      \cE_2 & = \left\{e_0 (n/\gamma)^{-\frac{1}{2\beta+1}} > d_{(K)}(u_0)\right\} \,,\\
      \cE_3 & = \left\{d_{(1)}(u_0)\leq e_0 (n/\gamma)^{-\frac{1}{2\beta+1}} \leq d_{(K)}(u_0)\right\} \,.
\end{align*}
As $d_{(1)}(u_0) \le d_{(K)}(u_0)$, it is immediate that we choose $h = d_{(1)}(u_0)$ under $\cE_1$, $h = d_{(K)}(u_0)$ under $\cE_2$ and $h = e_0 (n/\gamma)^{-1/(2\beta + 1)}$ under $\cE_3$. Now, from Lemma  \ref{lem:mseA-ub}, we know that for any choice of $h$ we have: 
\begin{equation}
    \begin{split}
        &\bbE \left[\Vert\hat\btheta_{\DVCM}(u_0)-\btheta(u_0)\Vert_A^2\middle | \Gamma\right] \leq q_1^2 h^{2\beta} + q_2S_h^{-1}.
    \end{split}
\end{equation}
Therefore, by a simple law of total expectation, we have: 
\begin{align*}
     &\bbE \left[\Vert\hat\btheta_{\DVCM}(u_0)-\btheta(u_0)\Vert_A^2\right] \leq \sum_{j=1}^3 \bbE\left[q_1^2 {h}^{2\beta} + q_2S_{h}^{-1} \mid \cE_j\right]\Pr(\cE_j) \coloneqq  \sum_{j = 1}^3 r_jp_j \,,
\end{align*}
where, for notational simplicity, define $r_j = \bbE[q_1^2 {h}^{2\beta} + q_2S_{h}^{-1} \mid \cE_j]$ and $p_j = \Pr(\cE_j)$. We next provide a bound on each $r_j$ and $p_j$ on a case-by-case basis. 
\\\\
\noindent 
\underline{First, consider $\cE_1$}: We choose $h = d_{(1)}(u_0)$, so only the nearest source (in terms of $u$) is selected. As a consequence, by Assumption \ref{Assump:balance-sample} a deterministic bound holds 
$$
S_h = \frac12\sum_{k=0}^K n_k\indc\left\{|U_k-u_0|\leq  h\right\} = \frac12(n_0 + n_{(1)}) \lesssim \bar n\,.
$$ 
where $n_{(1)}$ is the number of samples in the nearest source domain. 
Using this bandwidth, the bounds in Lemma \ref{lem:gap} yield 
\begin{align*}
    & \bbE\left[d_{(1)}(u_0)^{2\beta} \mid \cE_1 \right] \lesssim \left(\frac{K}{\gamma}\right)^{-2\beta}, & p_1 \asymp \left(1-(n\gamma^{2\beta})^{-\frac{1}{2\beta + 1}}\right)_+^K\,,
\end{align*}
which implies
\[
r_1 = \bbE\left[q_1^2 {h}^{2\beta} + q_2S_{h}^{-1} \mid \cE_1\right] \lesssim \left(\frac{K}{\gamma}\right)^{-2\beta} + \frac{1}{\bar n}\,.
\]  
\underline{Next, consider $\cE_2$}: We choose $h = d_{(K)}(u_0)$ in this case. This means all the domains are selected, and consequently, we have the following deterministic bound:  
$$
S_h = \frac12 \sum_{k=0}^K n_k\indc\left\{|U_k-u_0|\leq  h\right\} = \frac{n}{2}.
$$ 
Another application of Lemma \ref{lem:gap} yields: 
\begin{align*}
    & \bbE\left[d_{(K)}(u_0)^{2\beta} \mid \cE_2 \right] \lesssim \gamma^{2\beta}, &p_2 \coloneqq \Pr(\cE_2) \asymp 1 - \left(1-(n\gamma^{2\beta})^{-\frac{K}{2\beta + 1}}\right)_+\,.
\end{align*}
Thus, it follows that
\[
r_2 = \bbE\left[q_1^2 {h}^{2\beta} + q_2S_{h}^{-1} \mid \cE_2\right]\lesssim  \gamma^{2\beta} + \frac{1}{n}\,.
\]
\underline{Finally, consider $\cE_3$}: In this case, $h = e_0 (n/\gamma)^{-1/(2\beta + 1)}$. Hence, the bias part is deterministically bounded by 
$$
\bbE\left[h^{2\beta} \mid \cE_3\right] \lesssim \left(\frac{n}{\gamma}\right)^{-\frac{2\beta}{2\beta + 1}}.$$
By Lemma \ref{lem:var-bd}, the variance part is bounded by 
$$
\bbE\left[S_h^{-1} \mid \cE_3\right] \lesssim \frac{\gamma}{nh} \asymp \left(\frac{n}{\gamma}\right)^{-\frac{2\beta}{2\beta + 1}}\,.
$$
As a consequence: 
$$
r_3 = \bbE\left[q_1^2 {h}^{2\beta} + q_2S_{h}^{-1} \mid \cE_j\right] \lesssim \left(\frac{n}{\gamma}\right)^{-\frac{2\beta}{2\beta + 1}}, \qquad p_3 = 1 - p_1 - p_2\,.
$$
Now that we have established bounds on $\{r_j\}$ and $\{p_j\}$, we will bound the MSE using them. However, it is apparent from the definition of the events that one of these three events will dominate the others depending on the growth of $(n, \gamma, K)$. 
We discuss the bounds for MSE under three circumstances:  
$$
i) \  n\gamma^{2\beta} \gg K^{2\beta + 1}, \qquad ii) \ n\gamma^{2\beta} \ll 1, \qquad iii) \ 1 \lesssim n\gamma^{2\beta} \lesssim K^{2\beta + 1} \,.
$$
\underline{\bf Case 1: }At first we consider the case $n\gamma^{2\beta} \gg K^{2\beta + 1}$. Note that this implies: 
$$
n\gamma^{2\beta} \gg K^{2\beta + 1} \begin{cases}
    \implies \bar n^{-1} \asymp (n/K)^{-1} \ll (\gamma/K)^{2\beta} \,, \\
    \implies (n/\gamma)^{-2\beta/(2\beta+1)} \ll (\gamma/K)^{2\beta} \,.
\end{cases}
$$
Therefore, we have the upper bound on $r_1$: 
$$
r_1 \lesssim \left(\frac{K}{\gamma}\right)^{-2\beta} + \frac{1}{\bar n} \lesssim \left(\frac{K}{\gamma}\right)^{-2\beta} \,,
$$
and on $r_3$:
$$
r_3  \lesssim \left(\frac{n}{\gamma}\right)^{-\frac{2\beta}{2\beta + 1}} \ll \left(\frac{K}{\gamma}\right)^{-2\beta} \,.
$$
Furthermore, as $K \ge 1$, the assumption $n\gamma^{2\beta} \gg K^{2\beta + 1}$ immediately implies $n\gamma^{2\beta} \gg 1$, which, in turn, implies $n^{-1} \ll \gamma^{2\beta}$. Therefore, we have:  
\[
r_2 \lesssim \gamma^{2\beta} +\frac1n \lesssim \gamma^{2\beta}.\]
Moreover, the same condition $n\gamma^{2\beta} \gg 1$ also implies
\begin{align*}
    p_2 &\asymp 1- \big\{1-(n\gamma^{2\beta})^{-K/(2\beta+1)}\big\}_+ \asymp (n\gamma^{2\beta})^{-K/(2\beta+1)}\,.
\end{align*}
Combining all the bounds yields the following upper bound on the MSE of $\hat \theta_{\rm DVCM}$: 
\begin{align*}
   \bbE \left[\Vert\hat\btheta_{\DVCM}(u_0)-\btheta(u_0)\Vert_A^2\right] & \le r_1p_1 + r_2p_2 + r_3 (1-p_1-p_2)\\
    & \le r_1 \vee r_3 + r_2 p_2\\
    & \lesssim \left(\frac{K}{\gamma}\right)^{-2\beta} + \gamma^{2\beta}(n\gamma^{2\beta})^{-\frac{K}{2\beta + 1}} \lesssim \left(\frac{K}{\gamma}\right)^{-2\beta} \,.
\end{align*}
Here, the last inequality follows from the fact that $(K/\gamma)^{-2\beta} \gg \gamma^{2\beta}(n\gamma^{2\beta})^{-K/(2\beta+1)}$ as $n\gamma^{2\beta} \gg K^{2\beta + 1}$. This completes the bound under Case 1. 
\\\\
\noindent 
\underline{\bf Case 2: } In this case, we assume that $n\gamma^{2\beta} \ll 1 \lesssim K^{2\beta+1}$, i.e. $\gamma^{2\beta} \ll n^{-1}$. This immediately implies: 
$$
r_2 \lesssim \gamma^{2\beta}+ \frac1n \lesssim \frac{1}{n} \,,
$$
and 
$$
p_2 \asymp 1- \big\{1-(n\gamma^{2\beta})^{-K/(2\beta+1)}\big\}_+  \to 1 \qquad \text{since }n\gamma^{2\beta} \downarrow 0 \,, 
$$  
and consequently $p_1 = p_3 = 0$. Therefore,  
\begin{align*}
    & \bbE \left[\Vert\hat\btheta_{\DVCM}(u_0)-\btheta(u_0)\Vert_A^2\right] \lesssim r_2p_2 \lesssim \frac1n.
\end{align*}
\underline{\bf Case 3:} Finally, in this case, we consider the last case, $1 \lesssim n\gamma^{2\beta} \lesssim K^{2\beta + 1}$.  
We argue that in this case, MSE is 
upper bounded by the rate $(n/\gamma)^{-2\beta/(2\beta+1)}$. To establish this, it is enough to show $r_1p_1 \vee r_2p_2 \lesssim (n/\gamma)^{-2\beta/(2\beta +1)}$, as $r_3p_3 \lesssim (n/\gamma)^{-2\beta/(2\beta+1)}$ by the bound on $r_3$. 
With the definitions of $p_1$ and $p_2$, it follows that
\begin{align*}
    & p_1 \asymp \big(1-(n\gamma^{2\beta})^{-1/(2\beta+1)}\big)_+^K \lesssim \exp\left\{-K(n\gamma^{2\beta})^{-1/(2\beta+1)}\right\}\,, \qquad [\because (1-x)^K \lesssim e^{-Kx} \text{ for } x\in (0,1)]\\
    & p_2 \asymp 1- \big\{1-(n\gamma^{2\beta})^{-K/(2\beta+1)}\big\}_+ \lesssim (n\gamma^{2\beta})^{-K/(2\beta+1)}\,. \qquad [\because 1 \lesssim n\gamma^{2\beta} ]
\end{align*}
As we are consider the scenario $1 \lesssim n\gamma^{2\beta} \lesssim K^{2\beta + 1}$, we have: 
\begin{align*}
    n\gamma^{2\beta} \lesssim K^{2\beta + 1} & \implies \left(\frac{K}{\gamma}\right)^{-2\beta} \lesssim \frac{1}{\bar n} \,, \\
    1 \lesssim n\gamma^{2\beta} & \implies \frac1n \lesssim \gamma^{2\beta} \,.
\end{align*}
Thus we have the upper bounds for $r_1$ and $r_2$
\begin{align*}
    & r_1 \lesssim  \left(\frac{K}{\gamma}\right)^{-2\beta} + \frac{1}{\bar n} \lesssim \frac{1}{\bar n},\\
    & r_2 \lesssim \gamma^{2\beta} + \frac{1}{n} \lesssim \gamma^{2\beta}\,.
\end{align*}
These bounds, along with the upper bound on $(p_1, p_2)$, yield: 
\begin{align*}
    \frac{r_1p_1}{(n/\gamma)^{-\frac{2\beta}{2\beta+1}}} & = \frac{\bar n^{-1} \exp\left\{-K(n\gamma^{2\beta})^{-1/(2\beta+1)}\right\}}{(n/\gamma)^{-\frac{2\beta}{2\beta+1}}} \\
    & \asymp \frac{(K/n) \exp\left\{-K(n\gamma^{2\beta})^{-1/(2\beta+1)}\right\}}{(n/\gamma)^{-\frac{2\beta}{2\beta+1}}}\\
    & = K (n\gamma^{2\beta})^{-1/(2\beta+1)} \exp\left\{-K(n\gamma^{2\beta})^{-1/(2\beta+1)}\right\}\\
    & \lesssim 1\,, \qquad [\because xe^{-x} \text{ is always upper bounded by a constant for } x>0]
\end{align*}
and,
\begin{align*}
    \frac{r_2p_2}{(n/\gamma)^{-\frac{2\beta}{2\beta+1}}} = \frac{\gamma^{2\beta}(n\gamma^{2\beta})^{-K/(2\beta+1)}}{(n/\gamma)^{-\frac{2\beta}{2\beta+1}}} & = (n\gamma^{2\beta})^{\frac{2\beta-K}{2\beta+1}}\\
    & \lesssim 1 \qquad [\because n\gamma^{2\beta} \gtrsim 1, \text{ and assume }K \geq 2\beta \text{ w.l.o.g.}]
\end{align*}
Hence, we have shown in this case that
\begin{equation}\label{eqn:MSEA-DVCM-bound-c}
   \bbE \left[\Vert\hat\btheta_{\DVCM}(u_0)-\btheta(u_0)\Vert_A^2\right] \lesssim \left(\frac{n}{\gamma}\right)^{-\frac{2\beta}{2\beta+1}}\,.\
\end{equation}
\underline{\bf Justification of maximal upper bounds}

As a short summary, we have established the following regime-specific upper bound on the mean squared error:
\begin{equation*}
    \begin{split}
        & \bbE \left[\Vert\hat\btheta_{\DVCM}(u_0)-\btheta(u_0)\Vert_A^2\right] \lesssim \begin{cases} 
            (K/\gamma)^{-2\beta} & \text{if}\quad n\gamma^{2\beta} \gg K^{2\beta + 1} \gtrsim 1,\\
            (n/\gamma)^{-\frac{2\beta}{2\beta+1}} & \text{if}\quad 1 \lesssim n\gamma^{2\beta} \lesssim K^{2\beta + 1} ,\\
            n^{-1} & \text{if}\quad n\gamma^{2\beta} \ll 1 \lesssim K^{2\beta+1} \,.
        \end{cases}
    \end{split}
\end{equation*}
We next argue that, within each region, the overall upper bound is simply the maximum among the three individual upper bounds corresponding to that region, which follows from simple algebra. First observe that when $n\gamma^{2\beta} \gg K^{2\beta + 1}$, then: 
$$
n\gamma^{2\beta} \gg K^{2\beta + 1} \implies \frac{\gamma}{n} \ll \left(\frac{\gamma}{K}\right)^{2\beta + 1} \implies \left(\frac{n}{\gamma}\right)^{-\frac{2\beta}{2\beta + 1}} \ll \left(\frac{K}{\gamma}\right)^{-2\beta} \,,
$$
and multiplying $K/\gamma$ to middle inequality yields: 
$$
 \frac{\gamma}{n} \ll \left(\frac{\gamma}{K}\right)^{2\beta + 1} \implies \frac{1}{\bar n} \ll \left(\frac{K}{\gamma}\right)^{-2\beta} \implies \frac{1}{n} \ll \left(\frac{K}{\gamma}\right)^{-2\beta} \,.
$$
Hence, we conclude: 
$$
n\gamma^{2\beta} \gg K^{2\beta + 1} \implies \left(\frac{K}{\gamma}\right)^{-2\beta} \asymp \max\left\{\left(\frac{K}{\gamma}\right)^{-2\beta}, \left(\frac{n}{\gamma}\right)^{-\frac{2\beta}{2\beta + 1}}, \frac{1}{n}\right\}.
$$
Secondly, consider the case when $n\gamma^{2\beta} \ll 1$. In this case, we have: 
\begin{align*}
    \gamma^{2\beta} \ll n^{-1} \implies (\gamma/n)^{2\beta} \ll n^{-(2\beta+1)}  \implies (n/\gamma)^{-\frac{2\beta}{2\beta+1}} \ll n^{-1}
\end{align*}
and 
$$
\gamma^{2\beta} \ll n^{-1} \implies  (\gamma/K)^{2\beta} \ll n^{-1} \,.
$$
Hence, 
$$
n\gamma^{2\beta} \ll 1 \implies \frac1n \asymp \max\left\{\left(\frac{K}{\gamma}\right)^{-2\beta}, \left(\frac{n}{\gamma}\right)^{-\frac{2\beta}{2\beta + 1}}, \frac{1}{n}\right\} \,.
$$
Finally, let us consider the third case $1 \lesssim n\gamma^{2\beta} \lesssim K^{2\beta + 1}$. In this case, we have: 
$$
n\gamma^{2\beta} \lesssim K^{2\beta + 1} \implies (\gamma/K)^{2\beta+1} \lesssim \gamma/n  \implies (K/\gamma)^{-2\beta} \lesssim (n/\gamma)^{-\frac{2\beta}{2\beta+1}} \,,
$$
and
$$
 1 \lesssim n\gamma^{2\beta}\implies n^{-1} \lesssim \gamma^{2\beta} \implies  n^{-(2\beta+1)} \lesssim  (\gamma/n)^{2\beta} \implies n^{-1} \lesssim (n/\gamma)^{-\frac{2\beta}{2\beta+1}} \,.
$$
Therefore, in this case, we have: 
$$
1 \lesssim n\gamma^{2\beta} \lesssim K^{2\beta + 1} \implies \left(\frac{n}{\gamma}\right)^{-\frac{2\beta}{2\beta+1}} \asymp \max\left\{\left(\frac{K}{\gamma}\right)^{-2\beta}, \left(\frac{n}{\gamma}\right)^{-\frac{2\beta}{2\beta + 1}}, \frac{1}{n}\right\} \,.
$$
This completes the proof. 
\end{proof}

\subsection{Proof of Theorem \ref{thm:rate-TL}} \label{proof:thm:rate-TL}

\begin{proof}
We apply Theorem 1 of \cite{theobald1974generalizations}, which roughly says that for any two estimators $\hat\btheta_1$ and $\hat\btheta_2$, $M(\hat\btheta_1)\succeq M(\hat\btheta_2) \implies \mse_A(\hat\btheta_1)\geq\mse_A(\hat\btheta_2)$ for any $A\succeq 0$. 
Our goal is to show that there exists a matrix $Q$ that makes the difference between squared error matrices positive semidefinite, i.e. 
$$
M\left(\hat \btheta_{\DVCM}(u_0)\right) - M\left(\hat\btheta_{\TL}(u_0)\right) \succeq 0 \,,
$$
and 
$$
M\left(\hat \btheta_{\LR}(u_0)\right) - M\left(\hat\btheta_{\TL}(u_0)\right) \succeq 0 \,.
$$
Define $S_Q = \hat\Sigma_0+Q$ with $\hat\Sigma_0 = (\bX_0^\top \bX_0)/n_0$ and $\bX_0$ is a matrix concatenating all the $\{X_{0i}\}_{i \in \cI_0^*}$. From the first order condition, we have: 
\begin{align*}
    \hat\btheta_{\TL}(u_0) & = \left(\frac{1}{n_0}(\bX_0^\top \bX_0) +Q\right)^{-1}\left(\frac{1}{n_0}(\bX_0^\top \by_0)+Q\hat\btheta_{\DVCM}(u_0)\right) \\
    & = (\hat \Sigma_0 + Q)^{-1}\left(\frac{1}{n_0}(\bX_0^\top \by_0)+Q\hat\btheta_{\DVCM}(u_0)\right) \\
    & = S_Q^{-1}\left(\frac{1}{n_0}(\bX_0^\top \by_0)+Q\hat\btheta_{\DVCM}(u_0)\right) \\
    & = S_Q^{-1}\left(\hat \Sigma_0 \btheta(u_0) + \frac{1}{n_0} \bX_0^\top\beps_0 + Q\hat\btheta_{\DVCM}(u_0)\right) \qquad [\because \by_0 = \bX_0 \btheta(u_0) + \beps_0]\\
    & = S_Q^{-1}\left((S_Q - Q) \btheta(u_0) + \frac{1}{n_0} \bX_0^\top\beps_0 + Q\hat\btheta_{\DVCM}(u_0)\right) \\
    & = \btheta(u_0) + S_Q^{-1}\left(\frac{1}{n_0} \bX_0^\top\beps_0\right) + S_Q^{-1}Q(\hat\btheta_{\DVCM}(u_0) - \btheta(u_0)) \,.
\end{align*}
Therefore, the conditional bias of $\hat\btheta_{\TL}(u_0)$ is
\begin{equation*}
\begin{split}
\Bias\left(\hat\btheta_{\TL}(u_0) \middle | \bX_0\right) = S_Q^{-1} Q\Bias\left(\hat\btheta_{\DVCM}(u_0)\middle | \bX_0\right) = S_Q^{-1} Q\Bias\left(\hat\btheta_{\DVCM}(u_0)\right)
\end{split}
\end{equation*}
because $\hat\btheta_{\DVCM}(u_0)$ is independent of $\bX_0$ due to sample splitting and $\bbE[\eps_0 \mid \bX_0] = 0$ on the target domain. Now for the conditional variance, we utilize this independence again and have
\begin{equation*}
\begin{split}
&\Var\left(\hat\btheta_{\TL}(u_0) \middle | \bX_0\right)\\
& = \Var\left[S_Q^{-1}\left(\frac{1}{n_0}\bX_0^\top\beps_0+Q(\hat\btheta_{\DVCM}(u_0) - \btheta(u_0))\right) \middle |\bX_0\right]\\
& = S_Q^{-1} \left(  \frac{\sigma^2(u_0)}{n_0}\hat\Sigma_0 + Q \Var\left(\hat\btheta_{\DVCM}(u_0)\right)Q  \right) S_Q^{-1}.
\end{split}
\end{equation*}
Hence,
\begin{equation*}
    \begin{split}
        M\left(\hat\btheta_{\TL}(u_0) \middle | \bX_0\right)& = \Bias\left(\hat\btheta_{\TL}(u_0) \middle | \bX_0\right)^{\otimes 2} + \Var\left(\hat\btheta_{\TL}(u_0) \middle |\bX_0\right) \\
        & = S_Q^{-1}\left(  \frac{\sigma^2(u_0)}{n_0}\hat\Sigma_0 + QM\left(\hat\btheta_{\DVCM}(u_0)\right)Q  \right) S_Q^{-1}\,.
    \end{split}
\end{equation*}
Using the fact that
\begin{align*}
    M\left(\hat \btheta_{\LR}(u_0) \middle |\bX_0\right)= \frac{\sigma^2(u_0)}{n_0} \hat\Sigma^{-1}_0 & = \frac{\sigma^2(u_0)}{n_0} S_Q^{-1}S_Q\hat\Sigma^{-1}_0S_QS_Q^{-1} \\
    & = \frac{\sigma^2(u_0)}{n_0} S_Q^{-1}\left[\hat\Sigma_0+2Q+Q\hat\Sigma^{-1}_0Q\right]S_Q^{-1} \,,\qquad [\because S_Q = \hat\Sigma_0+Q]
\end{align*}
we obtain the difference between the second-order moments

\begin{equation*}
\begin{split}
& M\left(\hat \btheta_{\LR}(u_0)\right) - M\left(\hat\btheta_{\TL}(u_0) \right) \\
= & \bbE\left[ M\left(\hat \btheta_{\LR}(u_0) \middle | \bX_0\right) - M\left(\hat\btheta_{\TL}(u_0) \middle | \bX_0\right)\right]\\
= & \bbE\left[ S_Q^{-1}Q \left[  \frac{\sigma^2(u_0)}{n_0}\hat\Sigma^{-1}_0 - M\left(\hat \btheta_{\DVCM}(u_0)\right) + \frac{2\sigma^2(u_0)}{n_0}Q^{-1} \right] QS_Q^{-1}\right]
\end{split}
\end{equation*}
and
\begin{equation*}
\begin{split}
& M\left(\hat \btheta_{\DVCM}(u_0)\right) - M\left(\hat \btheta_{\TL}(u_0)\right) \\
= & \bbE\left[ M\left(\hat \btheta_{\DVCM}(u_0) \middle | \bX_0\right) - M\left(\hat \btheta_{\TL}(u_0) \middle | \bX_0\right)\right]\\
= & \bbE\left[ M\left(\hat \btheta_{\DVCM}(u_0)\right) - S_Q^{-1}\left(  \frac{\sigma^2(u_0)}{n_0}\hat\Sigma_0 + QM\left(\hat\btheta_{\DVCM}(u_0)\right)Q  \right) S_Q^{-1}\right]\\
= &  \bbE\left\{  S_Q^{-1}\left( S_QM\left(\hat \btheta_{\DVCM}(u_0)\right)S_Q -  \frac{\sigma^2(u_0)}{n_0}\hat\Sigma_0 - QM\left(\hat\btheta_{\DVCM}(u_0)\right)Q  \right) S_Q^{-1}\right\}.
\end{split}
\end{equation*}
Replacing $S_Q$ with $\hat\Sigma_0 + Q$, we have
\begin{equation*}
\begin{split}
    & M\left(\hat \btheta_{\DVCM}(u_0)\right) - M\left(\hat \btheta_{\TL}(u_0)\right)\\
    = & \bbE \bigg\{ S_Q^{-1}\hat\Sigma_0 \left[ M\left(\hat \btheta_{\DVCM}(u_0)\right) - \frac{\sigma^2(u_0)}{n_0}\hat\Sigma^{-1}_0 + \hat\Sigma^{-1}_0QM\left(\hat \btheta_{\DVCM}(u_0)\right) \right. \\ 
&\left. \hspace{20em}+ M\left(\hat \btheta_{\DVCM}(u_0)\right) Q \hat\Sigma^{-1}_0\right] \hat\Sigma_0 S_Q^{-1}\bigg\}\,.
\end{split}
\end{equation*}
Notice that if we set $Q$ such that $$\frac{1}{2}\frac{\sigma^2(u_0)}{n_0}M\left(\hat\btheta_{\DVCM}(u_0)\right)^{-1}\preceq Q \preceq 2\frac{\sigma^2(u_0)}{n_0}M\left(\hat\btheta_{\DVCM}(u_0)\right)^{-1}$$ then due to the positive semidefiniteness of $M\left(\hat \btheta_{\DVCM}(u_0)\right)$, $\hat\Sigma_0$, and $S_Q^{-1}$, the following holds
\begin{equation*}
    \begin{split}
        &M\left(\hat \btheta_{\LR}(u_0)\right) - M\left(\hat\btheta_{\TL}(u_0)\right) \succeq 0,\\
        & M\left(\hat \btheta_{\DVCM}(u_0)\right) - M\left(\hat\btheta_{\TL}(u_0)\right) \succeq 0.
    \end{split}
\end{equation*}
\noindent
By Theorem 1 of \cite{theobald1974generalizations}, this implies that
\begin{equation*}
    \begin{split}
        \mse_A\left(\hat\btheta_{\TL}(u_0) \right) & \leq \min\left\{\mse_A\left(\hat\btheta_{\DVCM}(u_0) \right),\mse_A\left(\hat\btheta_{\LR}(u_0) \right)\right\}\,,
    \end{split}
\end{equation*}
for any positive semidefinite $A$.

Then proof of the theorem follows from Proposition \ref{prop:rate-LR-VCM}, and
we have already argued that $\mse_A\left(\hat\btheta_{\LR}(u_0)\right) \le C n_0^{-1}$, which follows from basic properties of linear regression. Furthermore, in Proposition \ref{prop:rate-LR-VCM} we have established that: 
$$
\mse_A\left(\hat\btheta_{\DVCM}(u_0) \right)  \le C \max\left\{(K/\gamma)^{-2\beta}, (n/\gamma)^{-\frac{2\beta}{2\beta+1}}, n^{-1}\right\} \,.
$$
Hence, combining these, we obtain: 
\begin{align*}
    \mse_A\left(\hat\btheta_{\TL}(u_0) \right)  & \leq \min\left\{\mse_A\left(\hat\btheta_{\LR}(u_0) \right), \ \mse_A\left(\hat\btheta_{\DVCM}(u_0) \right) \right\} \\
    &  \le C\left[n_0^{-1} \wedge \max\left((K/\gamma)^{-2\beta}, (n/\gamma)^{-\frac{2\beta}{2\beta+1}}, n^{-1}\right)\right]\,.
\end{align*}
This completes the proof. 
\end{proof}

\subsection{Proof of Corollary \ref{thm:rate-TL-Qest}}\label{proof:thm:rate-TL-Qest}
\begin{proof}

Define the event
    \[\cE = \left\{\frac{1}{2}\frac{\sigma^2(u_0)}{n_0}M\left(\hat\btheta_{\DVCM}(u_0)\right)^{-1}\preceq \hat Q\preceq 2\frac{\sigma^2(u_0)}{n_0}M\left(\hat\btheta_{\DVCM}(u_0)\right)^{-1}\right\}\,,\]
and the desired rate
$$
r_{n,K,\gamma} = \sqrt{n_0 \vee \min\left\{(K/\gamma)^{2\beta}, \ (n/\gamma)^{2\beta/(2\beta  +1)}, n\right\}} \,.
$$
By Markov's inequality, it follows that
\begin{equation*}
\begin{split}
     \Pr\left(r_{n,K,\gamma}\left(\hat\btheta_{\TL,\hat Q}(u_0) - \btheta(u_0)\right) \ge t\right) 
     & = \Pr\left(r_{n,K,\gamma}\left(\hat\btheta_{\TL,\hat Q}(u_0) - \btheta(u_0)\right) \ge t, \cE\right)  + \Pr\left(r_{n,K,\gamma}\left(\hat\btheta_{\TL,\hat Q}(u_0) - \btheta(u_0)\right) \ge t, \cE^c\right)\\
     & \leq \Pr\left(r_{n,K,\gamma}\left(\hat\btheta_{\TL,\hat Q}(u_0) - \btheta(u_0)\right) \ge t\middle | \cE\right)  + \Pr\left(\cE^c\right)\\
     & \leq \frac{r_{n,K,\gamma}^2 }{t^2}\bbE\left[\left(\hat\btheta_{\TL,\hat Q}(u_0) - \btheta(u_0)\right)^2\middle | \cE\right] + \Pr\left(\cE^c\right)
\end{split}
\end{equation*}
Consider the expectation term. By Theorem \ref{thm:rate-TL}, it holds that
\[r_{n,K,\gamma}^2 \bbE\left[\left(\hat\btheta_{\TL,\hat Q}(u_0) - \btheta(u_0)\right)^2\middle | \cE\right]\leq C\]
for some constant $C>0$, and thus we have
\[\Pr\left(r_{n,K,\gamma}\left(\hat\btheta_{\TL,\hat Q}(u_0) - \btheta(u_0)\right) \ge t\right) \leq \frac{C}{t^2} + \Pr\left(\cE^c\right)\,.\]
Taking $\limsup_{n\to\infty}$ on both sides, it follows that
\[\limsup_{n\to\infty}\Pr\left(r_{n,K,\gamma}\left(\hat\btheta_{\TL,\hat Q}(u_0) - \btheta(u_0)\right) \ge t\right) \leq \frac{C}{t^2} + \limsup_{n\to\infty} \Pr\left(\cE^c\right)\,.\]

Now let's deal with term $\Pr\left(\cE^c\right)$. By definition of $\cE$, we obtain
\begin{equation*}
    \begin{split}
        & \Pr\left( \cE\right)\\
        & = \Pr\bigg( \frac{1}{2}\frac{\sigma^2(u_0)}{n_0}M\left(\hat\btheta_{\DVCM}(u_0)\right)^{-1}\preceq  \delta\frac{\hat\sigma^2(u_0)}{n_0}\hat M\left(\hat\btheta_{\DVCM}(u_0)\right)^{-1} \preceq 2\frac{\sigma^2(u_0)}{n_0}M\left(\hat\btheta_{\DVCM}(u_0)\right)^{-1}\bigg)\\
        & = \Pr\bigg( \frac{1}{2} I \preceq  \delta\frac{\hat\sigma^2(u_0)}{\sigma^2(u_0)}M\left(\hat\btheta_{\DVCM}(u_0)\right)\hat M\left(\hat\btheta_{\DVCM}(u_0)\right)^{-1} \preceq 2I \bigg)\,.
    \end{split}
\end{equation*}
By their definitions,
\begin{align*}
    & \frac{\hat\sigma^2(u_0)}{\sigma^2(u_0)} = 1 + o_p(1)\\
    & M\left(\hat\btheta_{\DVCM}(u_0)\right)\hat M\left(\hat\btheta_{\DVCM}(u_0)\right)^{-1} = I + o_p(1)
\end{align*}
and thus the middle term
\[\delta\frac{\hat\sigma^2(u_0)}{\sigma^2(u_0)}M\left(\hat\btheta_{\DVCM}(u_0)\right)\hat M\left(\hat\btheta_{\DVCM}(u_0)\right)^{-1} \xrightarrow[p]{n\to \infty} \delta I\,.\]
The coefficient $\delta\in(\frac{1}{2},2)$ by its definition, and this implies that  as $n \to \infty$,
$$\Pr(\cE^c)\rightarrow 0.$$
Hence, we get
\[\limsup_{n\to\infty}\Pr\left(r_{n,K,\gamma}\left(\hat\btheta_{\TL,\hat Q}(u_0) - \btheta(u_0)\right) \ge t\right) \leq \frac{C}{t^2} \,,\]
which means that
$$r_{n,K,\gamma}\left(\hat\btheta_{\TL,\hat Q}(u_0) - \btheta(u_0)\right) = O_p(1) \qquad \text{as } n \to \infty.$$

\end{proof}

\subsection{Proof of Theorem \ref{thm:minimax-TL}} \label{proof:thm:minimax-TL}
\begin{proof}

We want to establish a lower bound for
$$
\inf_{\hat{\btheta}(u_0)} \sup_{\substack{\forall j \in [p],\\ \theta_j \in \cH(\beta, L)}}\bbE \left(\Vert\hat\btheta(u_0)-\btheta(u_0)\Vert_A^2 \right).
$$ 

We will use Le Cam's approach \citep{tsybakov2008introduction} to find the minimax lower bound. We consider the data generating process 
$$Y_{ki}|U_k, X_{ki} \sim \cN(X_{ki}^\top \btheta_0(U_k),1), \quad \btheta_0 = (\theta_{01},\ldots,\theta_{0p})^\top;$$ 
$$Y_{ki}|U_k, X_{ki} \sim \cN(X_{ki}^\top \btheta_1(U_k),1), \quad \btheta_1 = (\theta_{11},\ldots,\theta_{1p})^\top.$$ 
Recall that $A$ is a positive semidefinite matrix whose $j$-th largest eigenvalue is $\lambda_{A,j}$ (with $\lambda_{A,1} > 0$) and corresponding eigenvector is $\bv_{A,j} = [v_{A,j1}, \ldots, v_{A,jp}]^\top$, with which each coordinate of functional coefficients is defined as
\begin{align*}
    \theta_{0j}(u) &\equiv 0&\text{ for }j = 1, \ldots,p\\
    \theta_{1j}(u) &= v_{A,1j}Lh^\beta W\left(\frac{u-u_0}{h}\right) &\text{ for }j = 1, \ldots,p,
\end{align*}
where $W(u) = c_0 \exp\left( -\frac{1}{1 - u^2} \right) \indc(|u| \leq 1)$, and $c_0 > 0$ is a small constant. There exists a threshold $C_0 > 0$ such that if $c_0 \leq C_0$, then $W$ belongs to $\mathcal{H}(\beta, 1/2)$.
 Using the definition $l = \lfloor \beta \rfloor$ and the fact that $W \in \mathcal H(\beta,1/2)$, it is shown that
\[\left|\theta_{1j}^{(l)}(u) - \theta_{1j}^{(l)}(u')\right| = |v_{A,1j}|Lh^{\beta-l}\left|W^{(l)}\left(\frac{u-u_0}{h} \right) - W^{(l)}\left(\frac{u'-u_0}{h} \right)\right| \leq \frac{|v_{A,1j}|L}{2}|u-u'|^{\beta-l}, \]
thus the constructed hypotheses lie in the parameter space. That is, $\forall j \in [p], \theta_{1j} \in \mathcal H(\beta,L)$. Moreover, we select domain sizes to be equal:
\begin{equation}\label{eq:minimax-lb-cond-n}
    n_k \equiv \bar n \qquad \text{for }k = 0,1,\ldots,K,
\end{equation}
let $U_k$'s be iid from pdf 
\begin{equation}\label{eq:minimax-lb-cond-U}
    \frac{1}{\gamma}f\left(\frac{u-u^*}{\gamma}\right), \qquad \text{with } f(\cdot) \leq a_0,
\end{equation}
and make $X_{ki}$ bounded:
\begin{equation}\label{eq:minimax-lb-cond-X}
    \|X_{ki}\|_2 \leq 1\,.
\end{equation}
Such choices of $n_k, U_k$ and $X_{ki}$ also lie in the presumed DGP, see Assumptions \ref{Assump:VCM-UX} and \ref{Assump:balance-sample}. Our metric of interest is lower bounded by
\begin{align*}
    \norm{ \btheta_1(u_0)- \btheta_0(u_0)}_A = \sqrt{\btheta_1(u_0)^\top A\btheta_1(u_0)} & = \sqrt{\sum_{j=1}^p\lambda_{A,j} \left(\bv_{A,j}^\top\btheta_1(u_0)\right)^2 }\\
    & \geq \sqrt{\lambda_{A,1}  \left(\bv_{A,1}^\top\btheta_1(u_0)\right)^2} = \sqrt{\lambda_{A,1}} Lh^{\beta} W(0)\,.
\end{align*}
We set $$r_{n,K,\gamma} = \sqrt{\lambda_{A,1}} Lh^{\beta} W(0)/3\,,$$ with $h$ to be chosen later.
Let
$\Pr_0$ and $\Pr_1$ are the joint distributions of $\{(U_k,X_{ki},Y_{ki}): k \in \{0\}\cup [K], i\in [n_k]\}$ under hypotheses $\btheta_0$ and $\btheta_1$. Then the following holds: 
\begin{equation}\label{eqn:minimax-reduc-1}
    \begin{split}
        \inf_{\hat{\btheta}(u_0)} \sup_{\substack{\forall j \in [p],\\ \theta_j \in \cH(\beta, L)}} \bbE \left(\Vert\hat\btheta(u_0)-\btheta(u_0)\Vert_A^2 \right)
&\ge \inf_{\hat{\btheta}(u_0)} \max_{\btheta \in \{\btheta_{0}, \btheta_{1}\}}  \bbE \left(\Vert\hat\btheta(u_0)-\btheta(u_0)\Vert_A^2 \right) \\
&\ge r_{n,K,\gamma}^2 \inf_{\hat{\btheta}(u_0)} \max_{\btheta \in \{\btheta_{0}, \btheta_{1}\}} \Pr \left(\Vert\hat\btheta(u_0)-\btheta(u_0)\Vert_A \geq r_{n,K,\gamma}\right) \\
&\ge r_{n,K,\gamma}^2 \inf_{\psi} \max_{j \in \{0,1\}} \Pr_j\left(\psi \ne j\right),
    \end{split}
\end{equation}
Now we justify the last inequality in \eqref{eqn:minimax-reduc-1}. By definition of $r_{n,K,\gamma}$, it holds that
$$
\|\btheta_1(u_0) - \btheta_0(u_0)\|_A > 2r_{n,K,\gamma}\,.
$$
Then, for any estimator $\hat\btheta(u_0)$, it is impossible for both
$$
\|\hat\btheta(u_0) - \btheta_0(u_0)\|_A < r_{n,K,\gamma} \quad \text{and} \quad \|\hat\btheta(u_0) - \btheta_1(u_0)\|_A < r_{n,K,\gamma}
$$
to simultaneously hold. Indeed, by the triangle inequality, this would imply
$$
\|\btheta_1(u_0) - \btheta_0(u_0)\|_A \le \|\btheta_1(u_0) - \hat\btheta(u_0)\|_A + \|\hat\btheta(u_0) - \btheta_0(u_0)\|_A < 2r_{n,K,\gamma},
$$
which contradicts the assumption. Therefore, for any estimator $ \hat\btheta(u_0)$, at least one of the following events must occur:
$$
\|\hat\btheta(u_0) - \btheta_0(u_0)\|_A \ge r_{n,K,\gamma} \quad \text{or} \quad \|\hat\btheta(u_0) - \btheta_1(u_0)\|_A \ge r_{n,K,\gamma}.
$$
Now define a test function $\psi \in \{0,1\}$ based on the estimator by
$$
\psi = \argmin_{j \in \{0,1\}} \|\hat\btheta(u_0) - \btheta_j(u_0)\|_A\,,
$$
which selects the hypothesis closest to the estimator. Then,
$$
\psi \ne j \quad \implies \quad \|\hat\btheta(u_0) - \btheta_j(u_0)\|_A \ge r_{n,K,\gamma}\,,
$$
so that
$$
\max_{j \in \{0,1\}} \Pr_j\left( \|\hat\btheta(u_0) - \btheta_j(u_0)\|_A \ge r_{n,K,\gamma} \right) \ge \max_{j \in \{0,1\}} \Pr_j(\psi \ne j),
$$
which proves last line of \eqref{eqn:minimax-reduc-1}. Moreover, by the definition of total variation distance and Pinsker’s inequality, we have 

\begin{equation}
    \begin{split}
        \Pr_0(\psi \ne 0) + \Pr_1(\psi \ne 1) 
        &= 1 - \left( \Pr_0(\psi \ne 1) - \Pr_1(\psi \ne 1) \right) \\
        &\ge 1 - \TV(\Pr_0, \Pr_1) \\
        &\ge 1 - \sqrt{ \frac{1}{2} \KL\left(\Pr_0\ \|\ \Pr_1\right) }, 
    \end{split}
\end{equation}
and thus
\begin{equation}\label{eqn:minimax-reduc-2}
    \inf_{\hat{\btheta}(u_0)} \sup_{\substack{\forall j \in [p],\\ \theta_j \in \cH(\beta, L)}} \bbE \left(\Vert\hat\btheta(u_0)-\btheta(u_0)\Vert_A^2 \right) \geq \frac{r_{n,K,\gamma}^2}{2}\left(1 - \sqrt{ \frac{1}{2} \KL\left(\Pr_0\ \|\ \Pr_1\right) }\right)\,.
\end{equation}
This completes the reduction part in Le Cam's approach. Now let $p_0,p_1$ be the density functions associated with $\Pr_0,\Pr_1$, then 
\begin{equation*}
    \begin{split}
        &\KL(\Pr_0\ \|\ \Pr_1)\\
        & = \bbE_{\Pr_0}\left[\log\frac{p_0\left(U_k,X_{ki}: k \in \{0\}\cup [K], i\in \cI_k\right)\prod_{k=0}^K \prod_{i\in\cI_k}p_0\left(Y_{ki}\mid U_k,X_{ki}\right)}{p_1\left(U_k,X_{ki}: k \in \{0\}\cup [K], i\in \cI_k \right)\prod_{k=0}^K \prod_{i\in\cI_k}p_1\left(Y_{ki}\mid U_k,X_{ki}\right)}\right],
    \end{split}
\end{equation*}
where $p_j\left(U_k,X_{ki}: k \in \{0\}\cup [K], i\in \cI_k\right)$ is the joint density function of all the $U_k$ and $X_{ki}$, and $p_j\left(Y_{ki}\mid U_k,X_{ki}\right)$ is the conditional density of $Y_{ki}$ given $(U_k,X_{ki})$. The marginal distribution of $U_k, X_{ki}$ are the same under the two hypotheses, which implies 
$$
p_0\left(U_k,X_{ki}: k \in \{0\}\cup [K], i\in \cI_k\right) = p_1\left(U_k,X_{ki}: k \in \{0\}\cup [K], i\in \cI_k\right),
$$ 
so it suffices to consider the conditional distribution $Y_{ki} | U_k,X_i$ when calculating $\KL(\Pr_0\ \|\ \Pr_1)$. Utilizing the normality of $Y_{ki} | U_k,X_{ki}$ and letting $\varphi$ be the pdf of standard normal distribution, we obtain:
\begin{equation}\label{eqn:KL-ub}
    \begin{split}
        \KL(\Pr_0\ \|\ \Pr_1) & = \bbE_{\Pr_0}\left[\log\frac{\prod_{k=0}^K \prod_{i\in\cI_k}p_0\left(Y_{ki}\mid U_k,X_{ki}\right)}{\prod_{k=0}^K \prod_{i\in\cI_k}p_1\left(Y_{ki}\mid U_k,X_{ki}\right)}\right]\\
        & = \sum_{k=0}^K\sum_{i\in\cI_k}\bbE_{U_k,X_{ki}} \int \log \frac{\varphi(t)}{\varphi\left(t- \btheta_1(U_k)\cdot X_{ki}\right)}\varphi(t)dt\\
        &=  \sum_{k=0}^K\sum_{i\in\cI_k} \bbE_{U_k,X_{ki}} \left(\btheta_1(U_k)\cdot X_{ki}\right)^2\\
        & \leq \sum_{k=0}^K\sum_{i\in\cI_k} \bbE_{U_k}  \Vert \btheta_1(U_k) \Vert^2_2. \quad [\because\text{Equation \eqref{eq:minimax-lb-cond-X}}]
    \end{split}
\end{equation}
It holds by definition of $\btheta_1$ that $\Vert \btheta_1(U_k) \Vert^2_2 = L^2 h^{2\beta} W^2\left(\frac{u_0-U_k}{h}\right)\leq L^2 h^{2\beta} W^2\left(0\right)\indc\left( \left|u_0 - U_k\right| \leq h \right)$, so we get an upper bound:
\begin{equation}\label{eqn:KL-ub-1}
    \begin{split}
        \KL(\Pr_0\ \|\ \Pr_1) & \leq \sum_{k=0}^K\sum_{i\in\cI_k} \bbE_{U_k} \left[ \Vert \btheta_1(U_k) \Vert^2_2 \right]\\
        &\leq L^2 W^2\left(0\right) h^{2\beta} \sum_{k=0}^K\sum_{i\in\cI_k} \bbE_{U_k} \left[ \indc\left( \left|u_0 - U_k\right| \leq h \right)\right]\\
        & \leq L^2c_0^2 h^{2\beta}\sum_{k=0}^Kn_k \bbE_{U_k} \left[ \indc\left( \left|u_0 - U_k\right| \leq h \right)\right]. \quad [\because W(0) \leq c_0 \text{ by its definition}]
    \end{split}
\end{equation}
Due to Equation \eqref{eq:minimax-lb-cond-U} on the upper bound of density of $U_k$, for any $k \in [K]$, the expectation part is upper bounded by
\begin{align*}
    & \bbE_{U_k} \left[ \indc\left( \left|u_0 - U_k\right| \leq h \right)\right] \leq \frac{2a_0h}{\gamma}\,,
\end{align*}
while for $k=0$, 
$$\bbE_{U_k} \left[ \indc\left( \left|u_0 - U_k\right| \leq h \right)\right]=1\,,$$
and thus
\begin{align*}
     &\KL(\Pr_0\ \|\ \Pr_1) \leq L^2c_0^2 h^{2\beta} \left(n_0 + \frac{2a_0 h}{\gamma}\sum_{k=1}^Kn_k\right) \leq (1\vee 2a_0)L^2c_0^2 h^{2\beta} \left(n_0 + \frac{h}{\gamma}\sum_{k=1}^Kn_k\right) \,.
\end{align*}
Because of the inequality $x + y \leq 2(x\vee y)$ for $x,y > 0$, we have
\begin{equation}\label{eqn:KL-ub-2}
    \KL(\Pr_0\ \|\ \Pr_1) \leq (2\vee 4a_0)L^2c_0^2 h^{2\beta}\left(n_0 \vee \frac{h}{\gamma}\sum_{k=1}^Kn_k\right)\,.
\end{equation}
We are going to discuss 3 possible choices of $h\in \{h_1, h_2, h_3\}$. All the choices of $h$ lead to the same upper bound of KL-divergence, that is $\KL(\Pr_0\ \|\ \Pr_1) \leq 2\alpha$ for some constant $\alpha \in (0,1/2)$. Then we apply Equation \eqref{eqn:minimax-reduc-2} and use the definition $r_{n,K,\gamma} = \sqrt{\lambda_{A,1}} Lh^{\beta} W(0)/3$ to obtain
\begin{equation}\label{eq:minimax-lb-max-3h}
\begin{split}
    \inf_{\hat{\btheta}(u_0)} \sup_{\substack{\forall j \in [p],\\ \theta_j \in \cH(\beta, L)}} \bbE \left(\Vert\hat\btheta(u_0)-\btheta(u_0)\Vert_A^2 \right) & \geq \frac{r_{n,K,\gamma}^2}{2}\left(1 - \sqrt{ \frac{1}{2} \KL\left(\Pr_0\ \|\ \Pr_1\right) }\right)\\ 
    &\geq \frac{\Big\{\sqrt{\lambda_{A,1}} Lh_j^{\beta} W(0)/3\Big\}^2}{2}\left(1 - \sqrt{\alpha}\right) \quad \text{for } j = 1,2,3\\
    & = \frac{\lambda_{A,1} L^2  W^2(0)h_j^{2\beta}}{18}\left(1 - \sqrt{\alpha}\right) \quad \text{for } j = 1,2,3.
\end{split}
\end{equation}
Hence, we are able to take maximum across the 3 $h_j$'s to find a tighter bound, and this yields the lower bound
  \[\inf_{\hat{\btheta}(u_0)} \sup_{\substack{\forall j \in [p],\\ \theta_j \in \cH(\beta, L)}} \bbE \left(\Vert\hat\btheta(u_0)-\btheta(u_0)\Vert_A^2 \right) \geq \max_{j \in \{1,2,3\}}\frac{\lambda_{A,1} L^2  W^2(0)h_j^{2\beta}}{18}\left(1 - \sqrt{\alpha}\right)\,. \]
Now, we discuss the 3 choices of $h \in \{h_1, h_2, h_3\}$, and their corresponding minimax lower bound. 
\\\\
\noindent 
\underline{\bf Case 1: } In this case, we find the first choice of $h$: 
$$h = h_1 = n_0^{-\frac{1}{2\beta}}\wedge \left(\gamma/K\right),$$
and use it to find corresponding $\alpha$, and upper bound of KL-divergence. Equation \eqref{eq:minimax-lb-cond-n} says $n_k \equiv \bar n$ for $k=0,1,\ldots,K$. Following the intermediate result in Equation \eqref{eqn:KL-ub-2} we then obtain an upper bound
\[\KL(\Pr_0\ \|\ \Pr_1) \leq  (2\vee 4a_0)L^2c_0^2 h^{2\beta} \bar n \left(1 \vee \frac{hK}{\gamma}\right)\,.\]
With the definition of $h = h_1$, it follows that
\begin{align*}
    1 \vee \frac{hK}{\gamma} \leq 1 \vee \frac{K}{\gamma}\left( n_0^{-\frac{1}{2\beta}}\wedge \left(\gamma/K\right)\right) \leq 1\,.
\end{align*}
Also, by Equation \eqref{eq:minimax-lb-cond-n}, $n_0 = \bar n$, and we obtain
\[  h^{2\beta} = \left( n_0^{-1}\wedge \left(K/\gamma\right)^{-2\beta}\right) \leq n_0^{-1} = \bar n^{-1}\,.\]
Hence, we can plug the upper bound for $1 \vee hK/\gamma$ and $h^{2\beta}$ into $\KL(\Pr_0\ \|\ \Pr_1)$'s upper bound. This yields
\[\KL(\Pr_0\ \|\ \Pr_1) \leq  (2\vee 4a_0)L^2c_0^2 \,.\]
 Note that the rate of $r_{n,K,\gamma} = \sqrt{\lambda_{A,1}} Lh^{\beta} W(0)/3$ is immediately obtained from the selected $h$ in each case. Therefore, by Equation \eqref{eqn:minimax-reduc-2} we have
\begin{equation}\label{eqn:minimax-lb-case1}
    \begin{split}
        \inf_{\hat{\btheta}(u_0)} \sup_{\substack{\forall j \in [p],\\ \theta_j \in \cH(\beta, L)}} \bbE \left(\Vert\hat\btheta(u_0)-\btheta(u_0)\Vert_A^2 \right) 
        &\geq \frac{r_{n,K,\gamma}^{2}}{2} \left(1 - \sqrt{ (1\vee 2a_0)L^2c_0^2 }\right)\\
        & =   \frac{\lambda_{A,1} L^2  W^2(0)\left( n_0^{-1}\wedge \left(K/\gamma\right)^{-2\beta}\right)}{18} \left(1 - \sqrt{ (1\vee 2a_0)L^2c_0^2 }\right)\\
        & =   \frac{\lambda_{A,1} L^2  c_0^2\left( n_0^{-1}\wedge \left(K/\gamma\right)^{-2\beta}\right)}{18e^2} \left(1 - \sqrt{ (1\vee 2a_0)L^2c_0^2 }\right)
    \end{split}
\end{equation}
where the last line is by definition
$$W(u) = c_0 \exp\left( -\frac{1}{1 - u^2} \right) \indc(|u| \leq 1) \implies W(0) = \frac{c_0}{e}.$$
We want to find $c_0$ such that
\[1 - \sqrt{ (1\vee 2a_0)L^2c_0^2 } \geq \frac{1}{2} \implies c_0 \leq \frac{1}{2\sqrt{ (1\vee 2a_0)L^2}}\,.\]
Recall that $W \in \mathcal{H}(\beta, 1/2)$ only if $c_0 \leq C_0$, so we can set $$c_0 = C_0 \wedge \frac{1}{2\sqrt{ (1\vee 2a_0)L^2}}\,.$$
This yields a lower bound
\[ \inf_{\hat{\btheta}(u_0)} \sup_{\substack{\forall j \in [p],\\ \theta_j \in \cH(\beta, L)}} \bbE \left(\Vert\hat\btheta(u_0)-\btheta(u_0)\Vert_A^2 \right)  \geq \frac{\lambda_{A,1} L^2  c_0^2}{36e^2}\left( n_0^{-1}\wedge \left(K/\gamma\right)^{-2\beta}\right).\]
\\\\
\noindent 
\underline{\bf Case 2: } In this case, we define
$$
h = h_2 = n_0^{-\frac{1}{2\beta}}\wedge n^{-\frac{1}{2\beta}} \implies h^{2\beta} = n_0^{-1} \wedge n^{-1}\,.
$$ 
Plugging this bandwidth into Equation \eqref{eqn:KL-ub-1} yields
\begin{equation*}
    \begin{split}
        \KL(\Pr_0\ \|\ \Pr_1)
        & \leq L^2c_0^2 h^{2\beta} \bbE_{U_k} \left[\sum_{k=0}^Kn_k \indc\left( \left|u_0 - U_k\right| \leq h \right)\right].
    \end{split}
\end{equation*}
Notice that the following upper bound always holds:
\[\sum_{k=0}^Kn_k \indc\left( \left|u_0 - U_k\right| \leq h \right) \leq n.\]
Hence,
\begin{equation*}
    \begin{split}
     \KL(\Pr_0\ \|\ \Pr_1)
        & \leq  L^2 c_0^2 h^{2\beta} n\\
        & = L^2 c_0^2\left(n_0^{-\frac{1}{2\beta}}\wedge n^{-\frac{1}{2\beta}}\right)^{2\beta}n\\
        & \leq  L^2 c_0^2\\
        & \leq  (2\vee 4a_0)L^2c_0^2\,.
    \end{split}
\end{equation*}
Again, by Equation \eqref{eqn:minimax-reduc-2} we have
\begin{equation}\label{eqn:minimax-lb-case2}
    \begin{split}
        \inf_{\hat{\btheta}(u_0)} \sup_{\substack{\forall j \in [p],\\ \theta_j \in \cH(\beta, L)}} \bbE \left(\Vert\hat\btheta(u_0)-\btheta(u_0)\Vert_A^2 \right) 
        &\geq \frac{r_{n,K,\gamma}^{2}}{2} \left(1 - \sqrt{ (1\vee 2a_0)L^2c_0^2 }\right)\\
        & =   \frac{\lambda_{A,1} L^2  c_0^2\left( n_0^{-1} \wedge n^{-1}\right)}{18e^2} \left(1 - \sqrt{ (1\vee 2a_0)L^2c_0^2 }\right)
    \end{split}
\end{equation}
Same as Case 1, we can set $c_0 = C_0 \wedge \frac{1}{2\sqrt{ (1\vee 2a_0)L^2}}$ to find lower bound
\[\geq \frac{\lambda_{A,1} L^2  c_0^2}{36e^2}\left( n_0^{-1} \wedge n^{-1}\right)\,.\]
\\\\
\noindent 
\underline{\bf Case 3: } In this case, we utilize the third choice of $h$: 
$$h = h_3 = (n/\gamma)^{-\frac{1}{2\beta+1}}\wedge n_0^{-\frac{1}{2\beta}}.$$
Firstly, it follows from Equation \eqref{eqn:KL-ub-2} that
\begin{equation*}
    \begin{split}
        \KL(\Pr_0\ \|\ \Pr_1) 
        & \leq (2\vee 4a_0)L^2c_0^2 h^{2\beta}\left(n_0 \vee \frac{h}{\gamma}\sum_{k=1}^Kn_k\right)\,.
    \end{split}
\end{equation*}
By the construction of $h=h_3$ we have
\begin{align*}
    & n_0 \vee \frac{h}{\gamma}\sum_{k=1}^Kn_k = n_0 \vee   \frac{ (n/\gamma)^{-\frac{1}{2\beta+1}}\wedge n_0^{-\frac{1}{2\beta}}}{\gamma}\sum_{k=1}^Kn_k \leq  n_0 \vee   \frac{ (n/\gamma)^{-\frac{1}{2\beta+1}}}{\gamma}\sum_{k=0}^Kn_k\,.
\end{align*}
Recall that $n= \sum_{k=0}^Kn_k$ and thus
\[ n_0 \vee \frac{h}{\gamma}\sum_{k=1}^Kn_k \leq n_0 \vee \left(n/\gamma\right)^{\frac{2\beta}{2\beta+1}}\]
Moreover, the definition of $h$ also implies
\begin{align*}
    h^{2\beta} = n_0^{-1} \wedge (n/\gamma)^{-\frac{2\beta}{2\beta+1}}   = \left\{n_0 \vee (n/\gamma)^{\frac{2\beta}{2\beta+1}} \right\}^{-1}\,.
\end{align*}
Now we can plug the values of $n_0 \vee \frac{h}{\gamma}\sum_{k=1}^Kn_k$ and $h^{2\beta}$ into $\KL(\Pr_0\ \|\ \Pr_1) $ and get
\begin{equation*}
    \begin{split}
        \KL(\Pr_0\ \|\ \Pr_1) 
        & \leq (2\vee 4a_0)L^2c_0^2 h^{2\beta}\left(n_0 \vee \frac{h}{\gamma}\sum_{k=1}^Kn_k\right) \\
        & \leq  (2\vee 4a_0)L^2c_0^2 \left\{n_0 \vee (n/\gamma)^{\frac{2\beta}{2\beta+1}} \right\}^{-1} \left\{n_0 \vee \left(n/\gamma\right)^{\frac{2\beta}{2\beta+1}}\right\}\\
        & = (2\vee 4a_0)L^2c_0^2\,.
    \end{split}
\end{equation*}
Again, using the lower bound in Equation \eqref{eqn:minimax-reduc-2} we have
\begin{equation}\label{eqn:minimax-lb-case3}
    \begin{split}
        \inf_{\hat{\btheta}(u_0)} \sup_{\substack{\forall j \in [p],\\ \theta_j \in \cH(\beta, L)}} \bbE \left(\Vert\hat\btheta(u_0)-\btheta(u_0)\Vert_A^2 \right) 
        &\geq \frac{r_{n,K,\gamma}^{2}}{2} \left(1 - \sqrt{ (1\vee 2a_0)L^2c_0^2 }\right)\\
        & =   \frac{\lambda_{A,1} L^2  c_0^2\left(n_0^{-1} \wedge (n/\gamma)^{-\frac{2\beta}{2\beta+1}}\right)}{18e^2} \left(1 - \sqrt{ (1\vee 2a_0)L^2c_0^2 }\right)
    \end{split}
\end{equation}
Same as Case 1, we can set $c_0 = C_0 \wedge \frac{1}{2\sqrt{ (1\vee 2a_0)L^2}}$ to find lower bound
\[\geq \frac{\lambda_{A,1} L^2  c_0^2}{36e^2}\left(n_0^{-1} \wedge (n/\gamma)^{-\frac{2\beta}{2\beta+1}}\right)\,.\]
\\\\
\noindent 
\underline{\bf Summarizing Case 1 - Case 3: }
By Equation \eqref{eq:minimax-lb-max-3h}, we find the maximum of the lower bounds on in Equations \eqref{eqn:minimax-lb-case1} - \eqref{eqn:minimax-lb-case3}, it follows that
\begin{equation*}
    \begin{split}
        \inf_{\hat{\btheta}(u_0)} \sup_{\substack{\forall j \in [p],\\ \theta_j \in \cH(\beta, L)}} \bbE \left(\Vert\hat\btheta(u_0)-\btheta(u_0)\Vert_A^2 \right) \geq C\cdot \frac{1}{n_0} \wedge \max\left\{\left(\frac{K}{\gamma}\right)^{-2\beta}, \left(\frac{n}{\gamma}\right)^{-\frac{2\beta}{2\beta + 1}}, \frac{1}{n}\right\}\,,
    \end{split}
\end{equation*}
where 
\[C =  \frac{\lambda_{A,1} L^2  c_0^2}{36e^2} \qquad \text{for } c_0 = C_0 \wedge \frac{1}{2\sqrt{ (1\vee 2a_0)L^2}}. \]
\end{proof}

\subsection{Proof of Theorem \ref{thm:Dconv-LR-TF-0-infty}}\label{proof:thm:Dconv-LR-TF-0-infty}

\begin{proof}
Linear regression is a special GLM with squared loss $\ell(\eta,y)=\frac12(y-\eta)^2$.
Accordingly we focus on the major changes relative to proof of Lemma \ref{lem:Dconv-GLM-TF-0-infty}. Readers may refer to part 2 of Appendix \ref{proof:lem:Dconv-GLM-TF-0-infty} for the proof under a more general setup. The norm $\|\cdot\|$ used in the proof denotes the $\ell_2$-norm when applied to a vector, and the spectral norm when applied to a matrix.

Write \(\Psi(u_0)=\bbE[XX^\top\mid U=u_0]\).
Let \(\btau_{\TL}\coloneqq \hat\btheta_{\TL}(u_0)-\btheta(u_0)\),
\(\btau_{\DVCM}\coloneqq \hat\btheta_{\DVCM}(u_0)-\btheta(u_0)\),
and \(\btau_{\LR}\coloneqq \hat\btheta_{\LR}(u_0)-\btheta(u_0)\). Recall the rates of $\hat\btheta_{\LR}(u_0)$, $\hat\btheta_{\DVCM}(u_0)$, and their ratio
\[
\btau_{\LR} =O_p(r_{\LR}), \quad \btau_{\DVCM}=O_p(r_{\DVCM}), \qquad \rho\coloneqq \frac{r_{\LR}}{r_{\DVCM}},
\]
and by construction $Q = \hat Q\asymp_p\rho^2I$, i.e.\ $c\rho^2\le\lambda_{\min}(Q)\le\lambda_{\max}(Q)\le C\rho^2$ w.p. $\to1$.

Consider the TL objective with linear DVCM:
\[
L_n(\alpha)
=\frac{1}{2n_0}\sum_{i\in\cI_0^*}\big(Y_{0i}-X_{0i}^\top\alpha\big)^2
+\frac12\|\alpha-\hat\btheta_{\DVCM}(u_0)\|_Q^2.
\]
Taking the gradient and setting it to zero at $\alpha=\hat\btheta_{\TL}(u_0)$ gives
\[
\frac{1}{n_0}\bX_0^\top\big(\bX_0\hat\btheta_{\TL}(u_0)-\bY_0\big)+Q\big(\hat\btheta_{\TL}(u_0)-\hat\btheta_{\DVCM}(u_0)\big)=0,
\]
where $\bX_0 \in \reals^{n_0 \times p}$ with $X_{0i}$ being its $i^{th}$ row, $\bY_0 \in \reals^{n_0}$ with entries being $Y_{0i}$. Let $\beps_0 \in \reals^{n_0}$ be a vector with $\varepsilon_{0i}$ being its $i^{th}$ entry. Substitute
$\bY_0=\bX_0\btheta(u_0)+\beps_0$ and $\hat\btheta_{\TL}(u_0)=\btheta(u_0)+\btau_{\TL}$:
\[
\frac{1}{n_0}\bX_0^\top\big(\bX_0\btau_{\TL}-\beps_0\big)+Q\big(\btau_{\TL}-\btau_{\DVCM}\big)=0
\ \ \Longleftrightarrow\ \
\btau_{\TL}
=\Big(\tfrac{1}{n_0}\bX_0^\top \bX_0+Q\Big)^{-1}\left(Q\,\btau_{\DVCM}+\tfrac{1}{n_0}\bX_0^\top\beps_0\right).
\]
Moreover, the involved components have the following limits:
\begin{align*}
    & \tfrac{1}{n_0}\bX_0^\top\bX_0=\Psi(u_0)+o_p(1),\quad \\
    & \tfrac{1}{n_0}\bX_0^\top\beps_0
\ =\ \Big(\tfrac{1}{n_0}\bX_0^\top\bX_0\Big)\btau_{\LR}
\ \implies\ 
\tfrac{1}{n_0}\bX_0^\top\beps_0=(\Psi(u_0)+o_p(1))\,\btau_{\LR}, \qquad (\because\text{OLS solution}).
\end{align*}
so the linear representation is written as
\begin{equation}\label{eq:TL-lin-rep-linear}
\begin{split}
    \btau_{\TL}
& =\big(\Psi(u_0)+o_p(1)+Q\big)^{-1}\Big\{\,Q\,\btau_{\DVCM}+(\Psi(u_0)+o_p(1))\btau_{\LR}\Big\}\\
& =\big(\Psi(u_0)+Q\big)^{-1}\Big\{\,Q\,\btau_{\DVCM}+\Psi(u_0)\btau_{\LR}\Big\}(1+o_p(1)).
\end{split}
\end{equation}
By sample splitting, \(\btau_{\LR}\) is independent of \(\btau_{\DVCM}\). Now we discuss the two cases  \(\rho\to0\) and  \(\rho\to \infty\).

\noindent\emph{Regime \(\rho\to0\) (LR–dominated).}
The properties \(\rho\coloneqq r_{\LR}/r_{\DVCM}\to0\) and \(Q=O_p(\rho^2)\) jointly imply the denominator in \eqref{eq:TL-lin-rep-linear} is
\(\big(\Psi(u_0)+Q\big)^{-1}=\Psi(u_0)^{-1}+o_p(1)\). Moreover, 
\(\|\Psi^{-1}(u_0)Q\btau_{\DVCM}\|=O_p(\rho^2r_{\DVCM})=o_p(r_{\LR})\).
Thus, from the representation in \eqref{eq:TL-lin-rep-linear},
\[
\btau_{\TL}=\btau_{\LR}+o_p(r_{\LR}).
\]
By Slutsky’s theorem and part 1 of Proposition \ref{prop:asymp-LR-VCM}, it follows that
\[r_{\LR}^{-1}\big(\hat\btheta_{\TL}(u_0)-\btheta(u_0)\big)
\xrightarrow[]{d}\cN\big(0,\ \Omega_{\LR}(u_0)\big).\]

\noindent\emph{Regime $\rho\to\infty$ (DVCM–dominated).}
Start from the linearization
\[
\btau_{\TL}=(\Psi(u_0)+Q)^{-1}\left\{\,Q\,\btau_{\DVCM}+\Psi(u_0)\,\btau_{\LR}\right\}+o_p(r_{\TL})(1+o_p(1)).
\]
Use the identity
\[
(\Psi(u_0)+Q)^{-1}Q=(\Psi(u_0)+Q)^{-1}(\Psi(u_0)+Q-\Psi(u_0))=I-(\Psi(u_0)+Q)^{-1}\Psi(u_0),
\]
to rewrite
\begin{equation}\label{eq:tau-TL-decomp}
    \btau_{\TL}
=\left\{\btau_{\DVCM}-(\Psi(u_0)+Q)^{-1}\Psi(u_0)\,\btau_{\DVCM}+(\Psi(u_0)+Q)^{-1}\Psi(u_0)\,\btau_{\LR}\right\}(1+o_p(1)).
\end{equation}
Since $\lambda_{\min}(Q)\asymp_p\rho^2$ and $\Psi(u_0)\succeq c_0'I$, we have
\[
\|(\Psi(u_0)+Q)^{-1}\|\le\|Q^{-1}\|=O_p(\rho^{-2}),\qquad \|\Psi\|=O(1).
\]
Hence it follows that
\[
\|(\Psi(u_0)+Q)^{-1}\Psi(u_0)\,\btau_{\DVCM}\|
\le \|(\Psi(u_0)+Q)^{-1}\|\,\|\Psi(u_0)\|\,\|\btau_{\DVCM}\|
=O_p(\rho^{-2} r_{\DVCM})=o_p(r_{\DVCM}),
\]
and, using $r_{\LR}=\rho\,r_{\DVCM}$,
\begin{align*}
    & \|(\Psi(u_0)+Q)^{-1}\Psi(u_0)\btau_{\LR}\|
\le \|(\Psi(u_0)+Q)^{-1}\|\,\|\Psi(u_0)\|\|\btau_{\LR}\|\\
    & =O_p(\rho^{-2} r_{\LR})
=O_p(r_{\DVCM}/\rho)=o_p(r_{\DVCM}).
\end{align*}
Hence, going back to Equation \eqref{eq:tau-TL-decomp}, this yields
\[
\btau_{\TL}=\btau_{\DVCM}+o_p(r_{\DVCM}).
\]
Therefore, by Slutsky’s theorem and part 2 of Proposition \ref{prop:asymp-LR-VCM}
\[
r_{\DVCM}^{-1}\Big(\hat\btheta_{\TL}(u_0)-\btheta(u_0)-\bb_{\DVCM}(u_0)\Big)
\ \xrightarrow[]{d}\ \cN\big(0,\ \Omega_{\DVCM}(u_0)\big).
\]
Proposition \ref{prop:asymp-LR-VCM} also implies that
\[r_{\DVCM}^2 \asymp h^{2\beta} + \gamma/nh, \qquad \bb_{\DVCM}(u_0) \lesssim h^{\beta}\]
and thus the condition $nh^{2\beta+1}/\gamma \to 0$ implies $r_{\DVCM}^{-1}\bb_{\DVCM}(u_0) = o_p(1)$, and this proves the theorem.
\end{proof}

\subsection{Proof of Theorem \ref{thm:TL-GLM-rate-Op}}\label{proof:thm:TL-GLM-rate-Op}
\begin{proof}

It is shown in Proposition \ref{prop:asymp-GLM-VCM} that
\[
\big\|\hat\btheta_{\GDVCM}(u_0)-\btheta(u_0)\big\|_2^2
=O_p\big(h^{2\beta}\big)+O_p\Big(\frac{\gamma}{nh}\Big).
\]
and $h$ is set to be
    $$
\med\left(e_0(n/\gamma)^{-\frac{1}{2\beta+1}},d_{(1)}(u_0),d_{(K)}(u_0)\right) = \med\left(O\left((n/\gamma)^{-\frac{1}{2\beta+1}}\right), O_p(\gamma/K), O_p(\gamma)\right),
$$ 
where the order of $d_{(1)}(u_0)$  and $d_{(K)}(u_0)$ are discussed in Lemma \ref{lem:gap}.
Then it follows that the bias-squared and variance of DVCM are of order
    \begin{align*}
        & O_p\left(h^{2\beta}\right) = \med\left(O_p\left((n/\gamma)^{-\frac{2\beta}{2\beta+1}}\right), O_p\left((\gamma/K)^{2\beta}\right), O_p\left(\gamma^{2\beta}\right)\right)\,,\\
        & O_p\left(\gamma/nh\right) =  \med\left(O_p\left((n/\gamma)^{-\frac{2\beta}{2\beta+1}}\right), O_p\left(K/n\right), O_p\left(n^{-1}\right)\right)\,.
    \end{align*}
We discuss the rates of $O_p\left(h^{2\beta}\right)$ and $O_p\left(\gamma/nh\right)$ under three circumstances:  
$$
i) \  n\gamma^{2\beta} \gg K^{2\beta + 1}, \qquad ii) \ n\gamma^{2\beta} \ll 1, \qquad iii) \ 1 \lesssim n\gamma^{2\beta} \lesssim K^{2\beta + 1} \,.
$$
\noindent\underline{\bf Case 1: }At first we consider the case $n\gamma^{2\beta} \gg K^{2\beta + 1}$. This implies: 
$$
n\gamma^{2\beta} \gg K^{2\beta + 1} \implies (n/\gamma)^{-2\beta/(2\beta+1)} \ll (\gamma/K)^{2\beta} \,.
$$
Also, $K \geq 1$ immediately imply $(\gamma/K)^{2\beta} \lesssim \gamma^{2\beta}$. Thus, it follows that
\[O_p\left(h^{2\beta}\right) =  O_p\left((\gamma/K)^{2\beta}\right)\,\]
by its definition. Moreover, the same condition $n\gamma^{2\beta} \gg K^{2\beta + 1}$ also implies that
$$
n\gamma^{2\beta} \gg K^{2\beta + 1} \implies  K/n \ll  (\gamma/K)^{2\beta}
$$
where the $K/n$ term has lower bound
$$
 n^{-1} \lesssim K/n\,.
$$
Thus, we have shown that
\[ (\gamma/K)^{2\beta} \gtrsim  K/n \vee  n^{-1} \vee (n/\gamma)^{-2\beta/(2\beta+1)},\]
 and this means 
 $$O_p\left(\gamma/nh\right) \lesssim O_p\left(h^{2\beta}\right) = O_p\left((\gamma/K)^{2\beta}\right),$$
 so we conclude in this case
\[\left\|\hat\btheta_{\GDVCM}(u_0)-\btheta(u_0)\right\|_2^2 = O_p\left((\gamma/K)^{2\beta}\right)\,.\]
\\\
\noindent 
\underline{\bf Case 2: } In this case, we assume that $n\gamma^{2\beta} \ll 1 \lesssim K^{2\beta+1}$. The first part of inequality implies that
$$
n\gamma^{2\beta} \ll 1 \implies (n/\gamma)^{-\frac{2\beta}{2\beta+1}} \ll n^{-1}
$$
while $K \geq 1$ immediately yields 
$$
K/n \gtrsim n^{-1}\,.
$$
Therefore, by its definition, 
\[O_p\left(\gamma/nh\right) = O_p\left(n^{-1}\right)\,.\]
The condition $n\gamma^{2\beta} \ll 1$ also implies that
\begin{align*}
    & n\gamma^{2\beta} \ll 1 \implies n\gamma^{2\beta} \ll K^{2\beta} \implies (\gamma/K)^{2\beta} \ll n^{-1}\,,\\
    & n\gamma^{2\beta} \ll 1 \implies \gamma^{2\beta} \ll n^{-1}\,.
\end{align*}
Hence, we have shown that
\[O_p\left(\gamma/nh\right) = O_p\left(n^{-1}\right) \gg O_p\left(h^{2\beta}\right) = \med\left(O_p\left((n/\gamma)^{-\frac{2\beta}{2\beta+1}}\right), O_p\left((\gamma/K)^{2\beta}\right), O_p\left(\gamma^{2\beta}\right)\right)\,.\]
and therefore
\[\left\|\hat\btheta_{\GDVCM}(u_0)-\btheta(u_0)\right\|_2^2 = O_p\left(n^{-1}\right)\,.\]
\noindent 
\underline{\bf Case 3:} In this case, we consider the last case, $1 \lesssim n\gamma^{2\beta} \lesssim K^{2\beta + 1}$. This condition implies
\begin{align*}
    & n\gamma^{2\beta} \lesssim K^{2\beta + 1} \implies (n/\gamma)^{-\frac{2\beta}{2\beta+1}} \gtrsim (\gamma/K)^{2\beta} \,, \\
    & 1 \lesssim n\gamma^{2\beta} \implies \gamma^{2\beta} \gtrsim (n/\gamma)^{-\frac{2\beta}{2\beta+1}}\,.
\end{align*}
Thus by definition
\[O_p\left(h^{2\beta}\right) = O_p\left((n/\gamma)^{-\frac{2\beta}{2\beta+1}}\right)\,.\]
Moreover, the same condition also implies
\begin{align*}
    & n\gamma^{2\beta} \lesssim K^{2\beta + 1} \implies (n/\gamma)^{-\frac{2\beta}{2\beta+1}} \lesssim K/n\,, \\
    & 1 \lesssim n\gamma^{2\beta} \implies n^{-1} \lesssim (n/\gamma)^{-\frac{2\beta}{2\beta+1}}\,.
\end{align*}
Thus by definition
\[O_p\left(\gamma/nh\right) = O_p\left((n/\gamma)^{-\frac{2\beta}{2\beta+1}}\right)\,.\]
Therefore, in this case
\[\left\|\hat\btheta_{\GDVCM}(u_0)-\btheta(u_0)\right\|_2^2 =  O_p\left((n/\gamma)^{-\frac{2\beta}{2\beta+1}}\right)\,.\]
\noindent 
\underline{\bf Collecting Case 1 to Case 3.}
Collecting the the results from \underline{\bf Case 1} to \underline{\bf Case 3}, we have established that
\begin{equation*}
    \begin{split}
        & \left\|\hat\btheta_{\GDVCM}(u_0)-\btheta(u_0)\right\|_2^2 = 
        \begin{cases} 
            O_p\left((K/\gamma)^{-2\beta}\right) & \text{if}\quad n\gamma^{2\beta} \gg K^{2\beta + 1} \gtrsim 1,\\
            O_p\left((n/\gamma)^{-\frac{2\beta}{2\beta+1}}\right) & \text{if}\quad 1 \lesssim n\gamma^{2\beta} \lesssim K^{2\beta + 1} ,\\
            O_p\left(n^{-1}\right) & \text{if}\quad n\gamma^{2\beta} \ll 1 \lesssim K^{2\beta+1} \,.
        \end{cases}
    \end{split}
\end{equation*}
Now we can apply the same argument in the last part of Appendix \ref{proof:prop:rate-LR-VCM} (``Justification of maximal upper bounds'') to conclude that
\[\left\|\hat\btheta_{\GDVCM}(u_0)-\btheta(u_0)\right\|_2^2 = O_p\left(\max\left\{(K/\gamma)^{-2\beta}, (n/\gamma)^{-\frac{2\beta}{2\beta+1}}, n^{-1}\right\}\right)\,.\]
Now we have established the rate of $\hat\btheta_{\GDVCM}(u_0)$ and we will move forward to $\hat\btheta_{\TL}(u_0)$. It has been shown in Lemma \ref{lem:Dconv-GLM-TF-0-infty} that
    \begin{align*}
        \left\|\hat\btheta_{\TL}(u_0)-\btheta(u_0)\right\|_2^2 = O_p\left(r_{\GLR}^2 \wedge r_{\GDVCM}^2\right)\,.
    \end{align*}
    The rate $r_{\GLR}^2 = n_0^{-1}$ so we conclude that
    \begin{align*}
        \left\|\hat\btheta_{\TL}(u_0)-\btheta(u_0)\right\|_2^2 = O_p\left(n_0^{-1} \wedge \max\left\{(K/\gamma)^{-2\beta}, (n/\gamma)^{-\frac{2\beta}{2\beta+1}}, n^{-1}\right\}\right)\,.
    \end{align*}
\end{proof}

\subsection{Proof of Theorem \ref{thm:TL-GLM-infer}}\label{proof:thm:TL-GLM-infer}
\begin{proof}
Let \(\btau_{\TL}\coloneqq \hat\btheta_{\TL}(u_0)-\btheta(u_0)\),
\(\btau_{\GDVCM}\coloneqq \hat\btheta_{\GDVCM}(u_0)-\btheta(u_0)\),
and \(\btau_{\GLR}\coloneqq \hat\btheta_{\GLR}(u_0)-\btheta(u_0)\).
Recall the expansion in Equation \eqref{eq:TL-lin-rep-linear}
\begin{equation}\label{eq:tau-TL-expand}
\btau_{\TL}
=\big(\Psi(u_0)+\hat Q\big)^{-1}\Big\{\hat Q\,\btau_{\GDVCM}+\Psi(u_0)\,\btau_{\GLR}\Big\}\,\{1+o_p(1)\},
\end{equation}
and set $B_Q\coloneqq \Psi(u_0)+\hat Q$. It is shown in Proposition~\ref{prop:asymp-GLM-VCM} that the asymptotic bias and variance for DVCM are of orders $O(h^\beta)$ and $O(\gamma/(nh))$. Therefore the condition $nh^{2\beta+1}/\gamma \to 0$ leads to the asymptotic unbiasedness of DVCM 
\begin{equation*}
      \sqrt{\frac{nh}{\gamma}}\left(\hat\btheta_{\GDVCM}(u_0) - \btheta(u_0)\right) 
    \xrightarrow[]{d} \cN\left( 0, \Omega_{\GDVCM}(u_0)\right) \,.
\end{equation*}
Therefore, by Proposition~\ref{prop:asymp-GLM-VCM}, and with definition on the variance for $\GLR$ and $\GDVCM$, 
\[V_{\GLR}(u_0) := \frac{1}{n_0} \Omega_{\GLR}(u_0), \quad V_{\GDVCM}(u_0) := \frac{\gamma}{nh} \Omega_{\GDVCM}(u_0)\]
it follows that
\[
V_{\GLR}^{-1/2}(u_0)\,\btau_{\GLR}\ \xrightarrow{d}\ \cN\big(0,I\big),
\qquad
V_{\GDVCM}^{-1/2}(u_0)\,\btau_{\GDVCM}\ \xrightarrow{d}\ \cN\big(0,I\big).
\]
Under sample splitting, the two terms $\btau_{\GLR}, \btau_{\GDVCM}$ are independent of each other. Since $\Psi(u_0)$ is positive definite
and $\hat Q$ is positive semidefinite, $B_Q$ is invertible w.p.~$\to1$.
By the continuous mapping theorem,
\[
\Sigma_1^{-1/2} \left[B_Q^{-1}\Psi(u_0)\,\btau_{\GLR} \right]
\ \xrightarrow{d}\ 
\cN\Big(0,I\Big),\quad \Sigma_1 = B_Q^{-1}\Psi(u_0)\,V_{\GLR}(u_0)\,\Psi(u_0)\,B_Q^{-1}
\]
and
\[
\Sigma_2^{-1/2} \left[B_Q^{-1}\hat Q\,\btau_{\GDVCM} \right]
\ \xrightarrow{d}\ 
\cN\Big(0,I\Big), \quad \Sigma_2 = B_Q^{-1}\hat Q\,V_{\GDVCM}(u_0)\,\hat Q\,B_Q^{-1}
\]
Independence implies their covariances are additive. Using
\eqref{eq:tau-TL-expand} and by Slutsky’s theorem,
\[
\Sigma_{\TL}^{-1/2}\,\btau_{\TL}\ \xrightarrow{d}\ \cN(0,I),
\]
where
\begin{align*}
    \Sigma_{\TL}
& = \Sigma_1 + \Sigma_2 \\
& =  B_Q^{-1}\Psi(u_0)\,V_{\GLR}(u_0)\,\Psi(u_0)\,B_Q^{-1} +  B_Q^{-1}\hat Q\, V_{\GDVCM}(u_0)\,\hat Q\,B_Q^{-1}\,.
\end{align*}
For feasible inference, let $\hat\Psi(u_0)$, $\hat V_{\GLR}(u_0)$, and $\hat V_{\GDVCM}(u_0)$ be consistent estimators. 
Define $\hat\Sigma_{\TL}$ by replacing the population quantities in $\Sigma_{\TL}$ with their estimators.
Then $\hat\Sigma_{\TL}=\Sigma_{\TL}\{1+o_p(1)\}$, and another application of Slutsky’s theorem yields
\[
\hat\Sigma_{\TL}^{-1/2}\,\btau_{\TL}\ \xrightarrow{d}\ \cN(0,I).
\]
\end{proof}

\subsection{Proof of Corollary \ref{coro:TL-linear-infer}}\label{proof:coro:TL-linear-infer}
\begin{proof}
This proof is the linear DVCM specialization of of the generalized DVCM derivation in Appendix \ref{proof:thm:TL-GLM-infer}. it follows the same steps, but with \(\Psi(u) = \mathbb{E}[XX^\top \mid U = u]\) and the variance component $\sigma(\cdot)$ specialized to the linear model.

Let \(\btau_{\TL}\coloneqq \hat\btheta_{\TL}(u_0)-\btheta(u_0)\),
\(\btau_{\DVCM}\coloneqq \hat\btheta_{\DVCM}(u_0)-\btheta(u_0)\),
and \(\btau_{\LR}\coloneqq \hat\btheta_{\LR}(u_0)-\btheta(u_0)\).
In the linear case (quadratic loss), the same linearization as in \eqref{eq:TL-lin-rep-linear} holds:
\begin{equation}\label{eq:tau-TL-expand-LR}
\btau_{\TL}
=\big(\Psi(u_0)+\hat Q\big)^{-1}\Big\{\hat Q\,\btau_{\DVCM}+\Psi(u_0)\,\btau_{\LR}\Big\}\,\{1+o_p(1)\},
\end{equation}
and we write \(B_Q\coloneqq \Psi(u_0)+\hat Q\).
By Proposition~\ref{prop:asymp-LR-VCM}, the DVCM bias and variance are of orders
\(O(h^\beta)\) and \(O(\gamma/(nh))\). Hence, if \(nh^{2\beta+1}/\gamma\to 0\),
the bias is negligible and
\[
\sqrt{\frac{nh}{\gamma}}\big(\hat\btheta_{\DVCM}(u_0)-\btheta(u_0)\big)
\ \xrightarrow[]{d}\ \cN\big(0,\ \Omega_{\DVCM}(u_0)\big).
\]
Moreover, by the same proposition, and with definition on the variance for $\LR$ and $\DVCM$, 
\[V_{\LR}(u_0) := \frac{1}{n_0} \Omega_{\LR}(u_0), \quad V_{\DVCM}(u_0) := \frac{\gamma}{nh} \Omega_{\DVCM}(u_0)\]
it follows that
\[
V_{\LR}^{-1/2}(u_0)\,\btau_{\LR}\ \xrightarrow{d}\ \cN\big(0,I\big),
\qquad
V_{\DVCM}^{-1/2}(u_0)\,\btau_{\DVCM}\ \xrightarrow{d}\ \cN\big(0,I\big).
\]
and (by sample splitting) these limits are independent.
Since \(\Psi(u_0)\succ 0\) and \(\hat Q\succeq 0\), \(B_Q\) is invertible w.p.\(\to1\).
By the continuous mapping theorem,
\[
\Sigma_1^{-1/2}\left[B_Q^{-1}\Psi(u_0)\,\btau_{\LR}\right]\ \xrightarrow{d}\ \cN(0,I),
\quad
\Sigma_1:=B_Q^{-1}\Psi(u_0)\,V_{\LR}(u_0)\,\Psi(u_0)\,B_Q^{-1},
\]
and
\[
\Sigma_2^{-1/2}\left[B_Q^{-1}\hat Q\,\btau_{\DVCM}\right]\ \xrightarrow{d}\ \cN(0,I),
\quad
\Sigma_2:=B_Q^{-1}\hat Q\,V_{\DVCM}(u_0)\,\hat Q\,B_Q^{-1}.
\]
Independence implies additivity of covariances. From \eqref{eq:tau-TL-expand-LR},
\[
\btau_{\TL}=B_Q^{-1}\Psi(u_0)\,\btau_{\LR}+B_Q^{-1}\hat Q\,\btau_{\DVCM}+o_p(\btau_{\TL}),
\]
so by Slutsky’s theorem,
\[
\Sigma_{\TL}^{-1/2}\,\btau_{\TL}\ \xrightarrow{d}\ \cN(0,I),
\qquad
\Sigma_{\TL}
=\Sigma_1+\Sigma_2
=B_Q^{-1}\Psi\,V_{\LR}\,\Psi B_Q^{-1}
+B_Q^{-1}\hat Q\,V_{\DVCM}\,\hat Q B_Q^{-1},
\]
For feasible inference, take consistent estimators
\(\hat\Psi(u_0),\ \hat V_{\LR}(u_0),\ \hat V_{\DVCM}(u_0)\),
form \(\hat\Sigma_{\TL}\) by plug-in, and conclude
\(\hat\Sigma_{\TL}^{-1/2}\,\btau_{\TL}\rightarrow_d\cN(0,I)\).
\end{proof}

\section{Proof of auxiliary lemmas}\label{sec:proof-aux-lem}
\subsection{Proof of part (1) of Lemma \ref{lem:gap}}\label{sec:proof:order-stat-d1}
Before going into the details, let's lay down some notations. Assume that $f$ is supported on $[a, b]$ with $f \in [1/a_0,a_0]$ for some large $a_0 > 0$. We further assume without loss of generality that $(b - a) \le a_0/2$. As the density of $U$ is $f((u - u_*)/\gamma)/\gamma$, it is immediate that: 
$$
F_U(u) := \Pr(U \le u) = F\left(\frac{u - u_*}{\gamma}\right) \,.
$$
Hence $U$ is supported on $[a\gamma + u_*, b\gamma + u_*]$. From the definition of nearest neighbour, we have: 
\begin{align*}
    \Pr(d_{(1)}(u_0) \ge \gamma t/K) = \Pr\left(|U - u_0| \ge \gamma t/K\right)^K & = \left(1 - \Pr(|U - u_0| \le \gamma t/K)\right)^K \\
    & = \left[1 - \left\{F\left(\frac{u_0 - u_*}{\gamma} + \frac{t}{K}\right) - F\left(\frac{u_0 - u_* }{\gamma} - \frac{t}{K}  \right) \right\}\right]^K
\end{align*}
Observe that the probability is not non-zero for all $t$. In fact $t$ must have the following upper bound: 
 $$
 \gamma t/K \le \max\left\{b\gamma+ u_* - u_0, \ u_0 - u_* - a\gamma\right\} \coloneqq \gamma t_0 \implies t \leq K t_0\,.
 $$
Now let us concentrate on the upper bound. As $f \ge 1/a_0$, we have: 
 $$
F\left(\frac{u_0 - u_*}{\gamma} + \frac{t}{K}\right) - F\left(\frac{u_0 - u_* }{\gamma} - \frac{t}{K}  \right) \ge \frac{2t}{a_0K} \,,
$$
and consequently: 
$$
1 - \left\{F\left(\frac{u_0 - u_*}{\gamma} + \frac{t}{K}\right) - F\left(\frac{u_0 - u_* }{\gamma} - \frac{t}{K}  \right)\right\} \le 1 - \frac{2t}{a_0K} \,.
$$
This bound only makes sense only if $t \le (a_0K)/2$. We now prove the lower bound on the probability,  where we use the upper bound on $f$, i.e., $f \le a_0$. This implies: 
$$
1 - \left\{F\left(\frac{u_0 - u_*}{\gamma} + \frac{t}{K}\right) - F\left(\frac{u_0 - u_* }{\gamma} - \frac{t}{K}  \right)\right\} \ge 1 - \frac{2a_0t}{K} \,.
$$
As before, this bound only makes sense when $t \le K/(2a_0)$. Therefore, it is concluded that
\begin{equation}\label{eqn:d1-lb-ub}
    \begin{split}
       \left\{1 - \frac{2a_0t}{K}\right\}^K \indc(t \le K/(2a_0)) \leq \Pr(d_{(1)}(u_0) \ge \gamma t/K) \leq  \left\{1 - \frac{2t}{a_0K} \right\}^K  \indc(t \le (a_0K)/2)\,.
    \end{split}
\end{equation}
Now consider the expectation:
 \begin{align*}
     \bbE\left[(K/\gamma)^{\beta}d_{(1)}(u_0)^{\beta}\right] & = \int_0^\infty \Pr\left((K/\gamma)^{\beta}d_{(1)}(u_0)^{\beta} \ge t\right) \ dt \\
     & = \int_0^\infty \Pr\left(d_{(1)}(u_0) \ge \frac{\gamma}{K}t^{\frac{1}{\beta}}\right) \ dt \,.
 \end{align*}
By the upper bound in Equation \eqref{eqn:d1-lb-ub}, 
\[\Pr\left(d_{(1)}(u_0) \ge \frac{\gamma}{K}t^{\frac{1}{\beta}}\right) \leq \left\{1 - \frac{2t^{\frac{1}{\beta}}}{a_0K} \right\}^K  \indc(t^{\frac{1}{\beta}} \le (a_0K)/2)\,\]
we obtain:
\begin{align*}
    \bbE\left[(K/\gamma)^{\beta}d_{(1)}(u_0)^{\beta}\right]  & \le \int_0^{\infty} \left(1 - \frac{2t^{1/\beta}}{a_0K} \right)^K \ dt \\
    & \le \int_0^{\infty}\exp\left(-\frac{2t^{1/\beta}}{a_0}\right) \ dt \\
    & \le \int_0^\infty \exp\left(-\frac{2t^{1/\beta}}{a_0}\right) \ dt = C_1 \,.
\end{align*}
We now prove the lower bound on the expectation,  where we use the lower bound in Equation \eqref{eqn:d1-lb-ub},
\[\Pr\left(d_{(1)}(u_0) \ge \frac{\gamma}{K}t^{\frac{1}{\beta}}\right) \geq   \left\{1 - \frac{2a_0t^{\frac{1}{\beta}}}{K}\right\}^K \indc(t^{\frac{1}{\beta}} \le K/(2a_0))\,\]
Define $t_1 = \min\{t_0, 1/(2a_0)\}$. We now have: 
\begin{align*}
    \bbE\left[(K/\gamma)^{\beta}d_{(1)}(u_0)^{\beta}\right]  & = \int_0^{(Kt_0)^{\beta}} \Pr\left(d_{(1)}(u_0) \ge \frac{\gamma}{K}t^{\frac{1}{\beta}}\right)  \ dt \\
    & \ge \int_0^{(Kt_1)^{\beta}} \Pr\left(d_{(1)}(u_0) \ge \frac{\gamma}{K}t^{\frac{1}{\beta}}\right)  \ dt  \\
    & \ge \int_0^{(Kt_1)^{\beta}}  \left(1 - \frac{2a_0t^{1/\beta}}{K}\right)^K \ dt \\
    & \ge \int_0^{(Kt_1)^{\beta}} \ \exp\left(-2a_0t^{1/\beta}\right)\left(1 - \frac{(2a_0t^{1/\beta})^2}{K}\right) \ dt \hspace{.3in} [\because (1 + x/n)^n \ge e^x(1 - x^2/n)] \\
    & \ge \int_0^{\frac{1}{C^\beta}K^{\beta/2}t_0^{\beta}} \ \exp\left(-2a_0t^{1/\beta}\right)\left(1 - \frac{(2a_0t^{1/\beta})^2}{K}\right)  \ dt  \hspace{.3in} [\because K \ge 1, C > 1 \text{ to be chosen later}] \\
    & \ge \left(1 - \left(\frac{2a_0 t_0}{C}\right)^2\right)\int_0^{\frac{1}{C^\beta}K^{\beta/2}t_1^{\beta}} \ \exp\left(-2a_0t^{1/\beta}\right) \ dt \\
    & \ge \left(1 - \left(\frac{2a_0 t_1}{C}\right)^2\right)\int_0^{\frac{1}{C^\beta}t_1^{\beta}} \ \exp\left(-2a_0t^{1/\beta}\right) \ dt \hspace{.3in} [\because K \ge 1] \\
    & \ge C_2 \,.
\end{align*}
Now it is immediate that any $C > \max\{1, 2a_0 t_0\}$ is a valid choice.

\subsection{Proof of part (2) of Lemma \ref{lem:gap}}\label{sec:proof:order-stat-dK}
We use the same notation as before. Assume that $f$ is supported on $[a, b]$ with $f \in [1/a_0,a_0]$ for some large $a_0 > 0$. We further assume without loss of generality that $(b - a) \le a_0/2$. As the density of $U$ is $f((u - u_*)/\gamma)/\gamma$, it is immediate that: 
$$
F_U(u) = \Pr(U \le u) = F\left(\frac{u - u_*}{\gamma}\right) \,.
$$
Hence $U$ is supported on $[a\gamma + u_*, b\gamma + u_*]$. Recall that $d_{(K)}(u_0) = \max_{k \in [K]} |U_k - u_0|$, it follows that
\begin{align*}
    \Pr(d_{(K)}(u_0) \ge \gamma t/K) = 1 -  \Pr\left(|U - u_0| \leq \gamma t/K\right)^K = 1 - \left\{F\left(\frac{u_0 - u_*}{\gamma} + \frac{t}{K}\right) - F\left(\frac{u_0 - u_* }{\gamma} - \frac{t}{K}  \right) \right\}^K
\end{align*}
Observe that the probability is not non-zero for all $t$. In fact $t$ must have the following upper bound: 
 $$
 \gamma t/K \le \max\left\{b\gamma+ u_* - u_0, \ u_0 - u_* - a\gamma\right\} \coloneqq \gamma t_0 \implies t \leq K t_0\,.
 $$
Now let us concentrate on the upper bound. As $f \ge 1/a_0$, we have: 
 $$
F\left(\frac{u_0 - u_*}{\gamma} + \frac{t}{K}\right) - F\left(\frac{u_0 - u_* }{\gamma} - \frac{t}{K}  \right) \ge \frac{2t}{a_0K} \,,
$$
and consequently: 
$$
1 - \left\{F\left(\frac{u_0 - u_*}{\gamma} + \frac{t}{K}\right) - F\left(\frac{u_0 - u_* }{\gamma} - \frac{t}{K}  \right)\right\} \le 1 - \frac{2t}{a_0K} \,.
$$
This bound only makes sense only if $t \le (a_0K)/2$. We now prove the lower bound on the probability,  where we use the upper bound on $f$, i.e., $f \le a_0$. This implies: 
$$
1 - \left\{F\left(\frac{u_0 - u_*}{\gamma} + \frac{t}{K}\right) - F\left(\frac{u_0 - u_* }{\gamma} - \frac{t}{K}  \right)\right\} \ge 1 - \frac{2a_0t}{K} \,.
$$
As before, this bound only makes sense when $t \le K/(2a_0)$. Therefore, it is concluded that
\begin{equation}\label{eqn:dK-lb-ub}
    \begin{split}
      \left[ 1 - \left\{\frac{2a_0t}{K}\right\}^K \right] \indc(t \le K/(2a_0)) \leq \Pr(d_{(K)}(u_0) \ge \gamma t/K) \leq \left[ 1 - \left\{\frac{2t}{a_0K} \right\}^K \right]  \indc(t \le (a_0K)/2)\,.
    \end{split}
\end{equation}
Now consider the expectation $\bbE\left[d_{(K)}(u_0)^{\beta}\right]$. Since $U$ is supported on $[a\gamma + u_*, b\gamma + u_*]$, whose Lebesgue measure $= (b-a)\gamma$, it is immediate that $\bbE\left[d_{(K)}(u_0)^{\beta}\right] \leq (b-a)\gamma^\beta $. 
Now consider the expectation:
 \begin{align*}
     \bbE\left[(K/\gamma)^{\beta}d_{(K)}(u_0)^{\beta}\right] & = \int_0^\infty \Pr\left((K/\gamma)^{\beta}d_{(K)}(u_0)^{\beta} \ge t\right) \ dt \\
     & = \int_0^\infty \Pr\left(d_{(K)}(u_0) \ge \frac{\gamma}{K}t^{\frac{1}{\beta}}\right) \ dt \,.
 \end{align*}

By the lower bound in Equation \eqref{eqn:dK-lb-ub}, 
\[\Pr\left(d_{(K)}(u_0) \ge \frac{\gamma}{K}t^{\frac{1}{\beta}}\right) \geq \left[1 - \left\{\frac{2a_0t^{\frac{1}{\beta}}}{K} \right\}^K \right] \indc(t^{\frac{1}{\beta}} \le K/(2a_0))\,.\]

Define $t_1 = \min\{t_0, 1/(4a_0)\}$. We now have: 
\begin{align*}
    \bbE\left[(K/\gamma)^{\beta}d_{(K)}(u_0)^{\beta}\right]  & = \int_0^{(Kt_0)^{\beta}} \Pr\left(d_{(K)}(u_0) \ge \frac{\gamma}{K}t^{\frac{1}{\beta}}\right)  \ dt \\
    & \ge \int_0^{(Kt_1)^{\beta}} \Pr\left(d_{(K)}(u_0) \ge \frac{\gamma}{K}t^{\frac{1}{\beta}}\right)  \ dt  \\
    & \ge \int_0^{(Kt_1)^{\beta}}  1 - \left(\frac{2a_0t^{1/\beta}}{K}\right)^K \ dt \\
    & = \int_0^{Kt_1} \left[1 - \left(\frac{2a_0s}{K}\right)^K\right] \beta s^{\beta-1} \ ds \hspace{.3in} [\because \text{substitution } s = t^{\frac{1}{\beta}}] \\
    & \ge \left[1 - \left(\frac{2a_0s}{K}\right)^K\right]\Bigg|_{s=K/(4a_0)} \int_0^{Kt_1}  \beta s^{\beta-1} \ ds\\
    &  \ge \frac{1}{2} \int_0^{Kt_1}  \beta s^{\beta-1} \ ds\\
    & = C_2 K^{\beta}\,.
\end{align*}
This means that $\bbE\left[d_{(K)}(u_0)^{\beta}\right] \geq C_2\gamma^\beta$ for some constant $C_2$.

\subsection{Proof of Lemma \ref{lem:bdd-Z}}\label{proof:lem:bdd-Z}

Let $\times$ denote Cartesian product of two sets. Recall the definition of $Z_{ki}$: 
$$
Z_{ki} = \bPhi_l\left(\frac{U_k - u_0}{h}\right) \otimes X_{ki} = \left\{X_{kij} \frac{1}{m!}\left(\frac{U_k - u_0}{h}\right)^m \right\}_{(j,m) \in   \{1,\ldots,p\} \times \{0,1,\ldots,l\}} \in \bbR^{p(l+1)} \,.
$$
Therefore, 
\begin{align*}
    \|Z_{ki}\|_2^2\indc\{|U_k-u_0|\le h\}& = \sum_{j = 1}^p \sum_{m = 0}^l X_{kij}^2\left( \frac{1}{m!}\left(\frac{U_k - u_0}{h}\right)^m\right)^2\indc\{|U_k-u_0|\le h\} \\
    & \le \sum_{m = 0}^l\left( \frac{1}{m!}\left(\frac{U_k - u_0}{h}\right)^m\right)^2\indc\{|U_k-u_0|\le h\} \hspace{.3in} [\because \|X_{ki}\|_2 \le 1] \\
    & \le \sum_{m = 0}^l\left( \frac{1}{m!}\right)^2 \\
    & \le \sum_{m = 0}^l\left( \frac{1}{m!}\right)\\
    & \le e\,.
\end{align*} 
Hence, 
    \[\|Z_{ki}\|_2^2\indc\{|U_k-u_0|\le h\} \leq e \implies \|Z_{ki}\|_2\indc\{|U_k-u_0|\le h\}\leq \sqrt{e} < 2\]

\subsection{Proof of Lemma \ref{lem:var-bd}}\label{proof:lem:var-bd}

By Assumption~\ref{Assump:balance-sample}, $n_k \ge b_0'\bar n$ for all $k=1,\ldots,K$. Hence
\[
S_h=\frac12\sum_{k=0}^K n_k\indc\{|U_k-u_0|\le h\}
\ge \frac{b_0'\bar n}{2}\sum_{k=1}^K \indc\{|U_k-u_0|\le h\}.
\]
Let $\cE:=\{d_{(1)}(u_0)\le h\}$ and pick an index $k_0$ such that $|U_{k_0}-u_0|=d_{(1)}(u_0)$. Then on $\cE$,
\[
S_h \ge \frac{b_0'\bar n}{2}\Big(1+\sum_{k\in[K]\setminus\{k_0\}}\indc\{|U_k-u_0|\le h\}\Big)
= \frac{b_0'\bar n}{2}\,(S+1),
\]
where $S:=\sum_{k\in[K]\setminus\{k_0\}}\indc\{|U_k-u_0|\le h\}$. Therefore,
\[
\bbE\left[S_h^{-1}\mid \cE\right] \le \frac{2}{b_0'\bar n}\,\bbE\left[\frac{1}{S+1}\,\middle|\,\cE\right].
\]
We now upper bound $\bbE[(S+1)^{-1}\mid \cE]$, using the inequality
\[
\bbE[(S+1)^{-1}\mid \cE]\le\frac{\bbE[(S+1)^{-1}]}{\Pr(\cE)}.
\]
Unconditionally, $S\sim\mathrm{Bin}(K-1,p)$ with
\[
p=\Pr(|U-u_0|\le h)\ \ge\ \frac{a_0'h}{\gamma}
\qquad\text{[by Assumption~\ref{Assump:VCM-UX} (a) and condition }h\le |\cU|].
\]
Since $S\sim\mathrm{Bin}(K-1,p)$,
\[
\bbE\left[\frac{1}{S+1}\right]
= \sum_{t=0}^{K-1}\frac{1}{t+1}\binom{K-1}{t}p^t(1-p)^{K-1-t}.
\]
Use the identity
\[
\frac{1}{t+1}\binom{K-1}{t}=\frac{1}{K}\binom{K}{t+1},
\]
to obtain
\[
\bbE\left[\frac{1}{S+1}\right]
= \frac{1}{K}\sum_{t=0}^{K-1}\binom{K}{t+1}p^t(1-p)^{K-1-t}.
\]
Reindex with $s=t+1$ (so $s=1,\dots,K$) and factor one $p$:
\[
\bbE\left[\frac{1}{S+1}\right]
= \frac{1}{pK}\sum_{s=1}^{K}\binom{K}{s}p^{\,s}(1-p)^{K-s}
= \frac{1}{pK}\Bigg(\sum_{s=0}^{K}\binom{K}{s}p^{\,s}(1-p)^{K-s}-(1-p)^K\Bigg).
\]
Notice that the bracket equals $1-(1-p)^K$, hence
\[
\bbE\left[\frac{1}{S+1}\right]
= \frac{1-(1-p)^K}{pK}.
\]
Moreover, $\Pr(\cE)=1-(1-p)^K$. Hence
\[
\bbE\left[\frac{1}{S+1}\,\middle|\,\cE\right]
\le \frac{\bbE[1/(S+1)]}{\Pr(\cE)}
= \frac{1}{pK}
\le \frac{\gamma}{a_0' K h}.
\]
Combining the bounds and using $n=(K+1)\bar n$,
\[
\bbE\left[S_h^{-1}\mid \cE\right]
\le \frac{2}{b_0'\bar n}\cdot \frac{\gamma}{a_0' K h}
\le C\,\frac{\gamma}{n h},
\]
for a constant $C>0$ depending only on $a_0',b_0'$.

\subsection{Proof of Lemma \ref{lem:mseA-ub}}\label{proof:lem:mseA-ub}

\textbf{Bias.}
Recall
\[
\hat\btheta_{\DVCM}(u_0)=\bA_l(\bZ^\top\bW\bZ)^{-1}\bZ^\top\bW\by .
\]
Conditioning on $\Gamma:=\{(U_k,X_{ki})\}$, write
\[
\by=\bZ\balpha(u_0)+\br+\beps,
\qquad \bbE[\beps\mid\Gamma]=0,
\]
where $\balpha(u_0)\in\mathbb{R}^{(l+1)p}$ collects the local polynomial coefficients at $u_0$
(with $\bA_l\balpha(u_0)=\btheta(u_0)$), and $\br\in\mathbb{R}^{n}$ is the stacked approximation error
with entries $r_{ki}=X_{ki}^\top\{\btheta(U_k)-\sum_{v=0}^l \btheta^{(v)}(u_0)(U_k-u_0)^v/v!\}$.
Then
\begin{align}
\bbE\!\left[\hat\btheta_{\DVCM}(u_0)-\btheta(u_0)\mid\Gamma\right]
&=\bA_l(\bZ^\top\bW\bZ)^{-1}\bZ^\top\bW\br .
\label{eq:bias-matrix-form}
\end{align}

We now bound the right-hand side in operator norm.
Using
Assumption \ref{Assump:VCM-UX}
\[
\big\|\bA_l(\bZ^\top\bW\bZ)^{-1}\bA_l^\top\big\|_2 \le \lambda_0^{-1},
\]
we obtain
\begin{align}
\Big\|\bbE\!\left[\hat\btheta_{\DVCM}(u_0)-\btheta(u_0)\mid\Gamma\right]\Big\|_2
&\le \big\|\bA_l(\bZ^\top\bW\bZ)^{-1}\bA_l^\top\big\|_2\;\big\|\bA_l\bZ^\top\bW\br\big\|_2 \notag\\
&\le \lambda_0^{-1}\;\big\|\bA_l\bZ^\top\bW\br\big\|_2 .
\label{eq:bias-step1}
\end{align}

Next, note that $\bA_l\bZ^\top$ extracts the first $p$ rows of $\bZ^\top$, which correspond to the part containing $X_{ki}$ only.
Equivalently,
\[
\bA_l\bZ^\top\bW\br=\sum_{k=0}^K\sum_{i\in\cI_k} w_k\,X_{ki}\,r_{ki} = \sum_{k=0}^K\sum_{i\in\cI_k} w_k\,X_{ki}\,X_{ki}^\top\Big\{\btheta(U_k)-\sum_{v=0}^l \btheta^{(v)}(u_0)(U_k-u_0)^v/v!\Big\},
\]
with $ w_k:=W\!\Big(\frac{U_k-u_0}{h}\Big)/S_h$. Using $\|X_{ki}\|_2\le 1$,
\[
\big\|\bA_l\bZ^\top\bW\br\big\|_2
\le \sum_{k=0}^K\sum_{i\in\cI_k} w_k\Big\|\btheta(U_k)-\sum_{v=0}^l \btheta^{(v)}(u_0)(U_k-u_0)^v/v!\Big\|_2.
\]

Finally, by Assumption~\ref{Assump:VCM-SN}(a) ($\btheta\in\cH(\beta,L)$) and the standard local-polynomial remainder,
on the support of $w_k$ we have $|U_k-u_0|\le h$ and hence
\[
\Big\|\btheta(U_k)-\sum_{v=0}^l \btheta^{(v)}(u_0)(U_k-u_0)^v/v!\Big\|_2\le C h^\beta .
\]
Since $\sum_{k,i} w_k = 1$ (because they are normalized by $S_h$), it follows that
\[
\big\|\bA_l\bZ^\top\bW\br\big\|_2\le C h^\beta.
\]
Plugging into \eqref{eq:bias-step1} yields
\[
\Big\|\bbE\!\left[\hat\btheta_{\DVCM}(u_0)-\btheta(u_0)\mid\Gamma\right]\Big\|_2
\le q_1' h^\beta,
\]
for some constant $q_1'>0$, and therefore
\[
\bbE^\top\!\left[\hat\btheta_{\DVCM}(u_0)-\btheta(u_0)\mid\Gamma\right]
A\,
\bbE\!\left[\hat\btheta_{\DVCM}(u_0)-\btheta(u_0)\mid\Gamma\right]
\le \lambda_{A,1}{q_1'}^2 h^{2\beta}
=: q_1 h^{2\beta},
\]
where $\lambda_{A,1}$ is the largest eigenvalue of $A$.

\textbf{Variance.}
Let $\beps$ be the stacked noise vector with $\bbE[\beps\mid\Gamma]=0$ and
$\Var(\beps\mid\Gamma)=\Sigma_\eps:=\diag(\sigma^2(U_k))\preceq \sigma_{\max}^2 I$.
Then
\[
\hat\btheta_{\DVCM}(u_0)-\bbE[\hat\btheta_{\DVCM}(u_0)\mid\Gamma]
=
\bA_l(\bZ^\top\bW\bZ)^{-1}\bZ^\top\bW\beps,
\]
so
\[
\Cov(\hat\btheta_{\DVCM}(u_0)|\Gamma)
=
\bA_l(\bZ^\top\bW\bZ)^{-1}\bZ^\top\bW\Sigma_\eps\bW\bZ(\bZ^\top\bW\bZ)^{-1}\bA_l^\top
\preceq
\sigma_{\max}^2\,\bA_l(\bZ^\top\bW\bZ)^{-1}\bZ^\top\bW^2\bZ(\bZ^\top\bW\bZ)^{-1}\bA_l^\top.
\]
Because $W(u)=\frac12\indc(|u|\le 1/2)$ and $\bW$ is normalized by $S_h$, we have $\bW^2=(2S_h)^{-1}\bW$, hence
\[
(\bZ^\top\bW\bZ)^{-1}\bZ^\top\bW^2\bZ(\bZ^\top\bW\bZ)^{-1}=\frac{1}{2S_h}(\bZ^\top\bW\bZ)^{-1}.
\]
Therefore,
\[
\Cov(\hat\btheta_{\DVCM}(u_0)\mid\Gamma)\preceq
\frac{\sigma_{\max}^2}{2S_h}\,\bA_l(\bZ^\top\bW\bZ)^{-1}\bA_l^\top,
\]
and
\begin{align*}
\bbE\!\left[\big\|\hat\btheta_{\DVCM}(u_0)-\bbE[\hat\btheta_{\DVCM}(u_0)\mid\Gamma]\big\|_A^2\mid\Gamma\right]
&=\tr\!\big(A\,\Cov(\hat\btheta_{\DVCM}(u_0)\mid\Gamma)\big) \\
&\le \frac{\sigma_{\max}^2}{2S_h}\tr\!\Big(A\,\bA_l(\bZ^\top\bW\bZ)^{-1}\bA_l^\top\Big) \\
&\le \frac{\sigma_{\max}^2\,\lambda_{A,1}}{2S_h}\tr\!\Big(\bA_l(\bZ^\top\bW\bZ)^{-1}\bA_l^\top\Big) \\
&\le \frac{\sigma_{\max}^2\,\lambda_{A,1}\,p}{2S_h}\big\|\bA_l(\bZ^\top\bW\bZ)^{-1}\bA_l^\top\big\|_2 \\
&\le \frac{\sigma_{\max}^2\,\lambda_{A,1}\,p}{2\lambda_0}\,S_h^{-1}
=:q_2 S_h^{-1}.
\end{align*}

\textbf{Conclusion.}
Combining the conditional squared bias and conditional variance bounds,
\[
\bbE\!\left[\|\hat\btheta_{\DVCM}(u_0)-\btheta(u_0)\|_A^2 \,\middle|\,\Gamma\right]
\le q_1 h^{2\beta} + q_2 S_h^{-1}.
\]

\subsection{Proof of Lemma \ref{lem:cond-normality-GDVCM}}\label{proof:lem:cond-normality-GDVCM}

Fix $k\in\{0\}\cup[K]$ and condition on $U_k=u_k$.  Let
\[
t_k\coloneqq \frac{u_k-u_0}{h},\qquad 
Z_{ki}\coloneqq \bPhi_l(t_k)\otimes X_{ki},\qquad
\eta_{ki}^*\coloneqq X_{ki}^\top\btheta(u_k),
\]
and define
\[
\delta_{ki}\coloneqq s_1\big(Z_{ki}^\top\bar\btheta(u_0),Y_{ki}\big)\,Z_{ki}\,W(t_k),
\qquad
\Delta_k\coloneqq n_k^{-1/2}\sum_{i\in\cI_k}\delta_{ki}.
\]
Throughout, $s_j(\eta,y)=\partial^j\ell(\eta,y)/\partial\eta^j$, and Assumptions
\ref{Assump:VCM-SN}–\ref{Assump:Unif-Kernel} are assumed. We work with a natural exponential family and canonical link so that
\begin{equation}\label{eq:exp-family-identity}
    \begin{split}
        &\bbE[s_1(\eta_{ki}^*,Y_{ki})\mid X_{ki},U_k]=0\\
        &\bbE[s_1^2(\eta_{ki}^*,Y_{ki})\mid X_{ki},U_k]=\nu(u_k)\,b''(\eta_{ki}^*)\\
        &s_j(\eta,Y_{ki})=b^{(j)}(\eta) \quad \text{for } j\geq 2, \text{ i.e. higher derivatives are independent of }Y,
    \end{split}
\end{equation}
with variance function $\nu(\cdot)$ bounded and continuous.

\noindent\textbf{Conditional mean.}
By Taylor expansion of $s_1(\cdot,Y_{ki})$ in $\eta$ around $\eta_{ki}^*$, write
\begin{equation}\label{eq:s1-taylor}
s_1\big(Z_{ki}^\top\bar\btheta(u_0),Y_{ki}\big)
= s_1(\eta_{ki}^*,Y_{ki})
  + s_2(\eta_{ki}^*,Y_{ki})\big(Z_{ki}^\top\bar\btheta(u_0)-\eta_{ki}^*\big)
  + \frac12\,s_3(\tilde\eta_{ki},Y_{ki})\big(Z_{ki}^\top\bar\btheta(u_0)-\eta_{ki}^*\big)^2,
\end{equation}
where $\eta_{ki}^*=X_{ki}^\top\btheta(u_k)$ and $\tilde\eta_{ki}$ lies between $\eta_{ki}^*$ and $Z_{ki}^\top\bar\btheta(u_0)$.
Using the identity $X_{ki}^\top=Z_{ki}^\top\bA_l^\top$, the local–polynomial bias satisfies
\begin{align}
Z_{ki}^\top\bar\btheta(u_0)-\eta_{ki}^*
&= X_{ki}^\top \frac{\btheta^{(l)}(u_0)-\btheta^{(l)}(\tilde u_k)}{l!}\,(u_k-u_0)^{l}\notag\\
&= Z_{ki}^\top\bA_l^\top \frac{\btheta^{(l)}(u_0)-\btheta^{(l)}(\tilde u_k)}{l!}\,(t_k h)^{l},
\label{eq:lp-bias}
\end{align}
uniformly for $|t_k|$ in the support of $W$ (hence $|u_0{-}u_k|\le h$), and $\tilde u_k$ is between $u_k$ and $u_0$.  Moreover, when $|u_0{-}u_k|\le h$, by H\"older continuity $\cH(L, \beta)$ and boundedness of $\|X_{ki}\|_2$, its absolute value is such that
\begin{align}\label{eq:bias-beta-rate}
    |Z_{ki}^\top\bar\btheta(u_0)-\eta_{ki}^*| \le \|X_{ki}\|_2 \frac{|\btheta^{(l)}(u_0)-\btheta^{(l)}(\tilde u_k)|}{l!}\,|u_k-u_0|^{l} =  O(h^\beta)\,.
\end{align}
Taking $\bbE[\cdot\mid X_{ki},U_k{=}u_k]$ in \eqref{eq:s1-taylor} and using the  identity 
$s_2(\eta_{ki}^*,Y_{ki})=b''(\eta_{ki}^*)$ gives
\begin{equation}\label{eq:s1-cond-mean}
\bbE\Big[s_1\big(Z_{ki}^\top\bar\btheta(u_0),Y_{ki}\big)\,\Big|\,X_{ki},U_k{=}u_k\Big]
= b''(\eta_{ki}^*)\big(Z_{ki}^\top\bar\btheta(u_0)-\eta_{ki}^*\big)
  + O\big(|Z_{ki}^\top\bar\btheta(u_0)-\eta_{ki}^*|^2\big).
\end{equation}
Multiplying \eqref{eq:s1-cond-mean} by $Z_{ki}W(t_k)$ and taking $\bbE[\cdot\mid U_k{=}u_k]$, we use
$\bbE[b''(\eta_{ki}^*)X_{ki}X_{ki}^\top\mid U_k=u_k]=\Psi(u_k)$ and \eqref{eq:lp-bias} to obtain
\begin{align*}
    \bbE[\delta_{ki}\mid U_k=u_k]
& =  \bbE\left[b''(\eta_{ki}^*)Z_{ki}\big(Z_{ki}^\top\bar\btheta(u_0)-\eta_{ki}^*\big)W(t_k)\mid U_k=u_k\right] +  R_k\\
& =  \bbE\left[b''(\eta_{ki}^*) Z_{ki}Z_{ki}^\top\bA_l^\top \frac{\btheta^{(l)}(u_0)-\btheta^{(l)}(\tilde u_k)}{l!}\,(t_k h)^{l}W(t_k)\mid U_k=u_k\right] + R_k
\\
& = \bPhi_l(t_k)^{\otimes 2}\otimes \Psi(u_k)
  \bA_l^\top \frac{\btheta^{(l)}(u_0)-\btheta^{(l)}(\tilde u_k)}{l!}\,(t_k h)^{l}\,W(t_k)
  + R_{k},
\end{align*}
where the remainder satisfies
\[
R_{k}
= \bbE\Big[\,O\big(|Z_{ki}^\top\bar\btheta(u_0)-\eta_{ki}^*|^2\big)\,\|Z_{ki}\|\,\Big|\,U_k{=}u_k\Big]\,W(t_k)
= O(h^{2\beta})=o(h^{\beta}),
\]
by boundedness of $\|X_{ki}\|_2$, result in \eqref{eq:bias-beta-rate}, and the bounded–moment assumption on $b^{(3)}$ (so that $\bbE[|s_3(\tilde\eta_{ki},Y_{ki})|^4\mid U_k]$ is $O(1)$).  
Therefore, it follows that
\[n_k^{-1/2}\,\bbE[\Delta_k\mid U_k{=}u_k] = \bPhi_l(t_k)^{\otimes 2}\otimes \Psi(u_k)\,
  \bA_l^\top \frac{\btheta^{(l)}(u_0)-\btheta^{(l)}(\tilde u_k)}{l!}\,(t_k h)^{l}\,W(t_k)\{1+o(1)\}\]
Moreover, by bounded spectral norm of $\bPhi_l(t_k)^{\otimes 2}, \Psi(u_k)$, $\bA_l$, and the same argument as in \eqref{eq:bias-beta-rate}, its absolute value is such that
\begin{align*}
    \Big |n_k^{-1/2}\,\bbE[\Delta_k\mid U_k{=}u_k]\Big |
& \le \Big\|\bPhi_l(t_k)^{\otimes 2}\otimes \Psi(u_k)\,
  \bA_l^\top \frac{\btheta^{(l)}(u_0)-\btheta^{(l)}(\tilde u_k)}{l!}\,(t_k h)^{l}\,W(t_k)\Big\|_2\,\{1+o(1)\}\\
  &= O(h^\beta)\,.
\end{align*}

\noindent\textbf{Conditional variance.}
By the same expansion as in Equation \eqref{eq:s1-taylor} and we square both sides to obtain
\[s_1^2\big(Z_{ki}^\top\bar\btheta(u_0),Y_{ki}\big) = s_1^2(\eta_{ki}^*,Y_{ki})+ O(h^{\beta})\]
The second identity in \eqref{eq:exp-family-identity} further implies
\[
\bbE\Big[s_1^2\big(Z_{ki}^\top\bar\btheta(u_0),Y_{ki}\big)\,\Big|\,X_{ki},U_k=u_k\Big]
= \nu(u_k)\,b''(\eta_{ki}^*) + O(h^{\beta}),
\]
where the $O(h^{\beta})$ term collects all the higher-order terms. Hence
\begin{align*}
\bbE\big[\delta_{ki}\delta_{ki}^\top\mid U_k=u_k\big]
&= \bbE\big[s_1^2\big(Z_{ki}^\top\bar\btheta(u_0),Y_{ki}\big)Z_{ki}Z_{ki}^\top\mid U_k=u_k\big]\,W(t_k)^2\\
&= \nu(u_k)\,\Big(\bPhi_l(t_k)^{\otimes 2}\otimes \Psi(u_k)\Big)\,W(t_k)^2
   + O(h^{\beta}).
\end{align*}
Since $\bbE[\delta_{ki}\mid U_k=u_k]=O(h^{\beta})$, the subtraction of
 $\bbE[\delta_{ki}\mid U_k=u_k]\,\bbE[\delta_{ki}\mid U_k=u_k]^\top$ affects the
variance only at order $O(h^{2\beta})$, so
\[
\Var[\delta_{ki}\mid U_k=u_k]
= \nu(u_k)\,\Big(\bPhi_l(t_k)^{\otimes 2}\otimes \Psi(u_k)\Big)\,W(t_k)^2
  + O(h^{\beta}).
\]
Because $\{\delta_{ki}\}_{i\in\cI_k}$ are conditionally i.i.d.,
\[
\Var[\Delta_k\mid U_k=u_k]
= \Var[\delta_{ki}\mid U_k=u_k]
= \nu(u_k)\,\Big(\bPhi_l(t_k)^{\otimes 2}\otimes \Psi(u_k)\Big)\,W(t_k)^2
  + O(h^{\beta}).
\]

\noindent\textbf{The matrix term $\Lambda_k$.}
Recall
\[
\Lambda_k
= n_k^{-1}\sum_{i\in\cI_k} s_2\big(Z_{ki}^\top\bar\btheta(u_0),Y_{ki}\big)\,Z_{ki}Z_{ki}^\top\,W(t_k).
\]
By a first–order Taylor expansion of $s_2(\cdot,Y_{ki})$ in $\eta$ around $\eta_{ki}^*$,
\[
s_2\big(Z_{ki}^\top\bar\btheta(u_0),Y_{ki}\big)
= s_2(\eta_{ki}^*,Y_{ki})
  + s_3(\tilde\eta_{ki},Y_{ki})\,
    \big(Z_{ki}^\top\bar\btheta(u_0)-\eta_{ki}^*\big),
\]
where $\tilde\eta_{ki}$ lies between $\eta_{ki}^*$ and $Z_{ki}^\top\bar\btheta(u_0)$.
Taking $\bbE[\cdot\mid X_{ki},U_k=u_k]$ and using the third identity in \eqref{eq:exp-family-identity} gives
\[
\bbE\big[s_2\big(Z_{ki}^\top\bar\btheta(u_0),Y_{ki}\big)\mid X_{ki},U_k=u_k\big]
= b''(\eta_{ki}^*) + \bbE\big[s_3(\tilde\eta_{ki},Y_{ki})\mid X_{ki},U_k=u_k\big]\,
    \big(Z_{ki}^\top\bar\btheta(u_0)-\eta_{ki}^*\big).
\]
By the same local polynomial approximation in \eqref{eq:lp-bias},
$Z_{ki}^\top\bar\btheta(u_0)-\eta_{ki}^*=O(h^{\,\beta})$ uniformly for
$|t_k|$ in the support of $W$. Now we marginalize $X_{ki}$. By the moment condition in Assumption~\ref{Assump:VCM-UX-Generalized},
$$
\bbE[\,|s_3(\tilde\eta_{ki},Y_{ki})|\mid U_k=u_k]= \bbE[|b^{(3)}(\tilde\eta_{ki})|\mid U_k=u_k] = O(1)\,.
$$
Hence
\begin{align*}
\bbE[\Lambda_k\mid U_k=u_k]
&= \bbE\Big[\big(b''(\eta_{ki}^*)+O(h^{\,\beta})\big)\,Z_{ki}Z_{ki}^\top \,\Big|\, U_k=u_k\Big]\,W(t_k) \\
&= \Big(\bPhi_l(t_k)^{\otimes 2}\otimes \bbE[\,b''(\eta_{ki}^*)\,X_{ki}X_{ki}^\top\mid U_k=u_k]\Big)\,W(t_k) + O(h^{\,\beta}) \\
&= \Big(\bPhi_l(t_k)^{\otimes 2}\otimes \Psi(u_k)\Big)\,W(t_k) + O(h^{\,\beta}).
\end{align*}

\subsection{Proof of Proposition \ref{prop:asymp-GLM-VCM}}\label{proof:prop:asymp-GLM-VCM}

\begin{proof}
We first note that the condition
\(\frac{\gamma}{K} \ll h \lesssim \Big(\frac{\gamma}{n}\Big)^{\frac{1}{2\beta+1}}\) implies a relationship we will repeatedly apply:
\begin{equation}\label{eq:terms-order}
    \frac{h^{2\beta-1}}{K} \ll \frac{h^{2\beta}}{\gamma} \lesssim \frac{1}{nh}\,.
\end{equation}
Let $s_j(\eta,y)=\partial^j\ell(\eta,y)/\partial\eta^j$ denote derivatives w.r.t.\ the first argument of $\ell$. Recall
\[
\hat\btheta_{\GDVCM}(u_0)=\bA_l\hat\balpha_{\GDVCM},\qquad
\hat\balpha_{\GDVCM}=\argmin_{\balpha\in\bbR^{(l+1)p}}
\sum_{k=0}^K\sum_{i\in\cI_k}\ell\big(Z_{ki}^\top\balpha,Y_{ki}\big)\,W\left(\frac{U_k-u_0}{h}\right),
\]
with $\bA_l=[I_p,\,0,\ldots,0]$, $Z_{ki}=\bPhi_l(t_k)\otimes X_{ki}$ and $t_k=(u_k-u_0)/h$.
Write
\[
\bar\btheta(u_0)^\top=\big[\btheta(u_0)^\top,\,h\btheta'(u_0)^\top,\,\ldots,\,h^l\btheta^{(l)}(u_0)^\top\big],
\qquad
r_{n}\coloneqq \left(\frac{nh}{\gamma}\right)^{-1/2}.
\]
We also use short hand notation for the density of $U_k$:
\[f_\gamma(u) \coloneqq \frac{1}{\gamma}f\left(\frac{u-u^*}{\gamma}\right)\,.\]
Due to Assumption \ref{Assump:balance-sample}, we assume w.l.o.g. that
\[ p_k\coloneqq \frac{n_k}{n} \asymp \frac{1}{K}.\]
The norm $\|\cdot\|$ used in the proof denotes the $\ell_2$-norm when applied to a vector, and the spectral norm when applied to a matrix.

\noindent\underline{\textbf{Taylor Expansion.}}
Define
\[
Q_n(\alpha)\coloneqq  \sum_{k=0}^K\sum_{i\in\cI_k}\ell\big(Z_{ki}^\top\alpha,Y_{ki}\big)\,W(t_k),
\qquad
D_n(\delta)\coloneqq Q_n\big(\bar\btheta(u_0) +r_{n}\delta\big)-Q_n(\bar\btheta(u_0) ),
\]
where \(r_{n}=(nh/\gamma)^{-1/2}\). Thus the minimizer $\hat\delta$ of $D_n(\delta)$ satisfy $\hat\delta = r_{n}^{-1}(\hat\balpha_{\GDVCM} - \bar\btheta(u_0))$. 
A second-order Taylor expansion of each summand at \(\eta_{ki}^*\coloneqq Z_{ki}^\top\bar\btheta(u_0) \) yields
\begin{align}\label{eq:ell-expand}
\ell(\eta_{ki}^*+r_{n}Z_{ki}^\top\delta,Y_{ki})
&=\ell(\eta_{ki}^*,Y_{ki})
+r_{n}s_1(\eta_{ki}^*,Y_{ki})Z_{ki}^\top\delta \notag\\
&\quad+\frac12 r_{n}^{2} s_2(\eta_{ki}^*,Y_{ki})(Z_{ki}^\top\delta)^2
+\frac16 r_{n}^{3} s_3(\tilde\eta_{ki},Y_{ki})(Z_{ki}^\top\delta)^3,
\end{align}
with \(\tilde\eta_{ki}\) between \(\eta_{ki}^*\) and \(\eta_{ki}^*+r_{n}Z_{ki}^\top\delta\). Define the following quantities
\[
\bar\Delta\coloneqq -\frac{1}{nh}\sum_{k=0}^K n_k^{1/2}\Delta_k,\quad
\bar\Lambda\coloneqq \frac{1}{nh}\sum_{k=0}^K n_k\Lambda_k,\quad
R_n(\delta)=\frac16\,r_{n}^3\sum_{k=0}^K\sum_{i\in\cI_k}
s_3(\tilde\eta_{ki},Y_{ki})\,(Z_{ki}^\top\delta)^3\,W(t_k),
\]
with (as in Lemma~\ref{lem:cond-normality-GDVCM})
\[
\Delta_k \coloneqq n_k^{-1/2}\sum_{i\in\cI_k}s_1\big(Z_{ki}^\top\bar\btheta(u_0),Y_{ki}\big)\,Z_{ki}\,W(t_k),\quad
\Lambda_k\coloneqq n_k^{-1}\sum_{i\in\cI_k}s_2\big(Z_{ki}^\top\bar\btheta(u_0),Y_{ki}\big)\,Z_{ki}Z_{ki}^\top\,W(t_k).
\]
Summing over \((k,i)\) for Equation \eqref{eq:ell-expand} and collecting terms gives
\begin{equation}\label{eq:Dn-expand}
    D_n(\delta)
=-(nh)\,r_{n}\,\delta^\top\bar\Delta
+\frac12 (nh)\,r_{n}^{2}\,\delta^\top\bar\Lambda\,\delta
+R_n(\delta).
\end{equation}

\noindent
\underline{\textbf{Marginalizing $U_k$.}}
From Lemma~\ref{lem:cond-normality-GDVCM}, for each $k$ and small enough $h$,
\begin{align}
n_k^{-1/2}\,\bbE[\Delta_k\mid U_k=u_k]
&=\bPhi_l(t_k)^{\otimes 2}\otimes \Psi(u_k)\,
  \bA_l^\top \frac{\btheta^{(l)}(u_0)-\btheta^{(l)}(\tilde u_k)}{l!}\,(t_k h)^{l}\,W(t_k)\{1+o(1)\} \lesssim h^\beta,
\label{eq:E-Delta-k}\\
\Var[\Delta_k\mid U_k=u_k]
&=\nu(u_k)\big(\bPhi_l(t_k)^{\otimes 2}\otimes\Psi(u_k)\big)W(t_k)^2+O(h^{\beta}),
\label{eq:V-Delta-k}\\
\bbE[\Lambda_k\mid U_k=u_k]
&=\big(\bPhi_l(t_k)^{\otimes 2}\otimes\Psi(u_k)\big)W(t_k)+O(h^{\beta}).
\label{eq:E-Lambda-k}
\end{align}

\noindent\emph{1) Mean of $\bar\Delta$.}
Since $\bar\Delta=-(nh)^{-1}\sum_{k=0}^K n_k^{1/2}\Delta_k$,
\[
\bbE\bar\Delta=-\sum_{k=0}^K p_k\Big\{\,h^{-1}n_k^{-1/2}\,\bbE\Delta_k\,\Big\}.
\]
For $k\in[K]$, integrating \eqref{eq:E-Delta-k} w.r.t. the density $f_\gamma(u) = \gamma^{-1} f\Big(\tfrac{u-u^*}{\gamma}\Big)$ of $U_k$
and writing $t=(u-u_0)/h$,
\begin{align*}
h^{-1}n_k^{-1/2}\,\bbE\Delta_k
&=h^{-1}\int\big(\bPhi_l(t)^{\otimes 2}\otimes\Psi(u)\big)\,
\bA_l^\top\frac{\btheta^{(l)}(u_0)-\btheta^{(l)}(\tilde u_k)}{l!}\,(th)^{l}W(t)
f_\gamma(u)\,du\,(1+o(1))\\
&\overset{t=(u-u_0)/h}{=}\int \bPhi_l(t)^{\otimes 2}t^{l}W(t)\ \otimes\ \Psi(u_0)
\bA_l^\top\bigg[\frac{\btheta^{(l)}(u_0)-\btheta^{(l)}(\tilde u_k)}{l!}\ h^{\,l}\,\gamma^{-1}\bigg]f\Big(\tfrac{u_0 - u^* + ht}{\gamma}\Big)\,dt\,(1+o(1)).
\end{align*}
Notice that due to H\"older continuity and the fact that $|\tilde u_k - u_0| \le h$, it holds that
\[\|\btheta^{(l)}(u_0)-\btheta^{(l)}(\tilde u_k)\| \lesssim h^{\beta-l}\]
and thus
\begin{equation}\label{eq:barDelta-mean}
h^{-1}n_k^{-1/2}\,\bbE\Delta_k = O(\gamma^{-1}h^\beta)\quad\implies \quad \bbE\bar\Delta = O(\gamma^{-1}h^\beta).
\end{equation}

\noindent\emph{2) Variance of $\bar\Delta$.}
By the law of total variance,
\[
\Var(\bar\Delta)
=(nh)^{-2}\sum_{k=0}^K n_k\,\Var(\Delta_k)
=(nh)^{-2}\sum_{k=0}^K n_k\left\{\bbE\big[\Var(\Delta_k\mid U_k)\big]
+\Var\big(\bbE[\Delta_k\mid U_k]\big)\right\}.
\]
For the leading term, using \eqref{eq:V-Delta-k} and the change of variables
$t=(u-u_0)/h$,
\begin{align*}
h^{-1}\,\bbE\big[\Var(\Delta_k\mid U_k)\big]
&=h^{-1}\int \nu(u)\,\big(\bPhi_l(t)^{\otimes 2}\otimes\Psi(u)\big)\,W(t)^2
\,f_\gamma(u)\,du\,(1+o(1))\\
&=\gamma^{-1}\Big[\textstyle\int \bPhi_l(t)^{\otimes 2}W(t)^2\,f\Big(\tfrac{u_0 - u^* + ht}{\gamma}\Big)dt\Big]\otimes\Psi(u_0)
\nu(u_0)\,(1+o(1))\\
&=\gamma^{-1}\zeta_{0,2}\otimes\Psi(u_0)\nu(u_0)\,(1+o(1)).
\end{align*}
Hence
\[
(nh)^{-2}\sum_{k=0}^K n_k\,\bbE\big[\Var(\Delta_k\mid U_k)\big]
=(\gamma nh)^{-1}\,\nu(u_0)\,\zeta_{0,2}\otimes\Psi(u_0)\,(1+o(1)).
\]
For the second term, by the bound \(\|\Var(Z)\|\le \bbE\|Z\|^2\) and \eqref{eq:E-Delta-k},
together with the same change of variables \(t=(u-u_0)/h\) and the upper-boundedness
of \(\|\btheta^{(l)}(u_0) - \btheta^{(l)}(\tilde u_k)\|\), \(\lambda_{\max}\{\Psi(u)\}\), and \(f(u)\),
\begin{align*}
    \big\|\Var\big(\bbE[\Delta_k\mid U_k]\big)\big\|
 & \lesssim n_k\,h^{2\beta}\int \|\bPhi_l(t)^{\otimes2}\|^2\,t^{2l}W(t)^2\,f_\gamma(u_0+ht)\,du\\
 & = \gamma^{-1}n_k\,h^{2\beta+1}\int \|\bPhi_l(t)^{\otimes2}\|^2\,t^{2l}W(t)^2\,f((u_0+ht)/\gamma)\,dt
\end{align*}
Since \(\bPhi_l(t)=(1,t,\ldots,t^l/l!)^\top\) implies
\(\|\bPhi_l(t)\|_2^2\indc(|t|\le1)\lesssim 1\) and \(W\) is the uniform kernel
(Assumption~\ref{Assump:Unif-Kernel}) with finite moments, we obtain
\[
\big\|\Var\big(\bbE[\Delta_k\mid U_k]\big)\big\|\lesssim\frac{n_k\,h^{2\beta+1}}{\gamma}.
\]
Therefore,
\[
(nh)^{-2}\sum_{k=0}^K n_k\,\big\|\Var\big(\bbE[\Delta_k\mid U_k]\big)\big\|
\lesssim (nh)^{-2}\sum_{k=0}^K \frac{n_k^2\,h^{2\beta+1}}{\gamma}
=\frac{h^{2\beta-1}}{\gamma}\sum_{k=0}^K\Big(\frac{n_k}{n}\Big)^2
\lesssim\frac{h^{2\beta-1}}{\gamma K},
\]
using \(n_k/n \lesssim 1/K\).
Thus combining the two components in the law of total variance and applying Equation \eqref{eq:terms-order} yield:
\begin{equation}
\Var(\bar\Delta)
=(\gamma nh)^{-1}\,\nu(u_0)\,\zeta_{0,2}\otimes\Psi(u_0)\,(1+o(1)).
\label{eq:barDelta-var}
\end{equation}

\noindent\emph{3) Mean of $\bar\Lambda$.}
Since $\bar\Lambda=(nh)^{-1}\sum_{k=0}^K n_k\Lambda_k$,
\[
\bbE\bar\Lambda=\sum_{k=0}^K p_k\Big\{\,h^{-1}\bbE\Lambda_k\,\Big\}.
\]
Integrating \eqref{eq:E-Lambda-k} as above gives
\begin{align*}
h^{-1}\bbE\Lambda_k
&=h^{-1}\int\big(\bPhi_l(t)^{\otimes 2}\otimes\Psi(u)\big)\,W(t)
f_\gamma(u)\,du\,(1+o(1))\\
&\overset{t=(u-u_0)/h}{=}\gamma^{-1} \left[\int \bPhi_l(t)^{\otimes 2}W(t)f\Big(\tfrac{u_0 - u^* + ht}{\gamma}\Big)\,dt\right]\otimes\Psi(u_0)\,(1+o(1))\\[0.2em]
&=\gamma^{-1} \zeta_{0,1}\ \otimes\ \Psi(u_0)\,(1+o(1)).
\end{align*}
Therefore,
\begin{equation}
\bbE\bar\Lambda
=\gamma^{-1}\zeta_{0,1}\otimes\Psi(u_0)\,(1+o(1)).
\label{eq:barLambda-mean}
\end{equation}

\noindent\underline{\textbf{Rate of minimizer.}}
Fix \(M<\infty\) and consider the sphere \(\{\|\delta\|=M\}\).
Assumption~\ref{Assump:VCM-UX-Generalized} implies bounded \(\bbE|s_3(\cdot,Y)|^4 = \bbE|b^{(3)}(\cdot)|^4\) and \(\bbE\|X\|^4\), while Assumption \ref{Assump:Unif-Kernel} (compactly supported bounded \(W\)) implies \(\bbE[W(t_k)]\asymp \int \indc(|u|\leq h) f_\gamma(u) du = O(h/\gamma)\).
Using the Taylor remainder above and \(|Z_{ki}^\top\delta|\le \|Z_{ki}\|\,\|\delta\|\), we get
\begin{equation}\label{eq:Rn-op1}
    \sup_{\|\delta\|\le M}|R_n(\delta)|
= O_p(1)\cdot M^3\,r_{n}^{3}\sum_{k=0}^K\sum_{i\in\cI_k}\|Z_{ki}\|^3W(t_k)
=O_p\Big(M^3\,r_{n}^{3}\,\frac{nh}{\gamma}\Big)
=O_p\Big(\frac{M^3}{\sqrt{nh/\gamma}}\Big)=o_p(1).
\end{equation}
Equation~\eqref{eq:barLambda-mean} implies
\(\gamma\bar\Lambda\xrightarrow{p}\zeta_{0,1}\otimes\Psi(u_0)\), so for some \(c>0\),
\[
\lambda_{\min}(\bar\Lambda)\ \ge\ \frac{c}{\gamma}\qquad\text{w.p.}\to1.
\]
Moreover, from Equations \eqref{eq:barDelta-mean} and \eqref{eq:barDelta-var}, as well as condition $nh^{2\beta+1}/\gamma \lesssim 1$,
\[
\|\bar\Delta\|=O_p\big((nh\,\gamma)^{-1/2}\big).
\]
Recall the decomposition in \eqref{eq:Dn-expand}:
\[
D_n(\delta)
=-(nh)\,r_{n}\,\delta^\top\bar\Delta
+\frac12\,(nh)\,r_{n}^{2}\,\delta^\top\bar\Lambda\,\delta
+R_n(\delta).
\]
Hence, on \(\{\|\delta\|=M\}\),
\[
\begin{aligned}
\inf_{\|\delta\|=M} D_n(\delta)
&\ge -(nh)\,r_{n}\,\|\bar\Delta\|\,M
+\frac12\,(nh)\,r_{n}^{2}\,\lambda_{\min}(\bar\Lambda)\,M^2
-\sup_{\|\delta\|\le M}|R_n(\delta)| \\[0.25em]
&\ge -\,M\cdot O_p(1)+\frac12\,\gamma\cdot\frac{c}{\gamma}\,M^2-o_p(1)
=\frac12\,c\,M^2-O_p(M)-o_p(1).
\end{aligned}
\]
Choosing \(M\) as a sufficiently large constant makes the RHS positive w.p.\(\to1\).
Since \(D_n(0)=0\) and \(D_n\) is convex, the (unique) minimizer
\(\hat\delta\coloneqq \arg\min_\delta D_n(\delta)\) lies in the ball \(\{\|\delta\|\le M\}\) w.p.\(\to1\), which justifies the rate $r_{n}$ and consistency.

\noindent\underline{\textbf{Linearization.}}
On \(\{\|\delta\|\le M\}\) we have the uniform expansion, hence
\[
\nabla D_n(\delta)=-(nh)\,r_{n}\bar\Delta+(nh)\,r_{n}^{2}\,\bar\Lambda\,\delta+\nabla R_n(\delta).
\]
\(\nabla D_n(\hat\delta)=0\), so
\[
\hat\delta
= r_{n}^{-1}\bar\Lambda^{-1}\bar\Delta - \gamma^{-1}\bar\Lambda^{-1}\nabla R_n(\hat\delta)
= r_{n}^{-1}\bar\Lambda^{-1}\bar\Delta + o_p(1),
\]
because \(\sup_{\|\delta\|\le M}\|\nabla R_n(\delta)\|=o_p(1)\) and \(\|\bar\Lambda^{-1}\|=O_p(\gamma)\) by Equations \eqref{eq:Rn-op1} and \eqref{eq:barLambda-mean}.
Finally, since \(\hat\balpha_{\GDVCM}=\bar\btheta(u_0) +r_{n}\hat\delta\) and \(\hat\btheta_{\GDVCM}(u_0)=\bA_l\hat\balpha_{\GDVCM}\),
\begin{equation}\label{eq:DVCM-linearization}
    \bA_l\hat\delta = r_{n}^{-1}\big(\hat\btheta_{\GDVCM}(u_0)-\btheta(u_0)\big)
= r_{n}^{-1}\bA_l\,\bar\Lambda^{-1}\bar\Delta + o_p(1).
\end{equation}

\noindent\underline{\textbf{CLT for $\bar\Delta$.}}
Fix a unit vector $v_0\in\bbR^{p(l+1)}$ and define
\[
T_K(v_0)\coloneqq v_0^\top(\bar\Delta-\bbE\bar\Delta)
=-\frac{1}{nh}\sum_{k=0}^K n_k^{1/2}\,\eta_k(v_0),
\qquad
\eta_k(v_0)\coloneqq v_0^\top\big(\Delta_k-\bbE\Delta_k\big).
\]
By construction, $\{\eta_k(v_0)\}_{k=0}^K$ are independent, centered random variables. From
\eqref{eq:V-Delta-k} and Lemma~\ref{lem:cond-normality-GDVCM},
\[
\Var\big(\eta_k(v_0)\mid U_k=u_k\big)
= v_0^\top\Big(\nu(u_k)\,\bPhi_l(t_k)^{\otimes 2}\otimes\Psi(u_k)\Big)v_0\,W(t_k)^2 + O(h^{\beta})\,.
\]
Thus, taking expectation w.r.t. $U_k$ over density $f_\gamma(u) = \gamma^{-1}f(\gamma^{-1}(u-u^*))$ yields
\[
\Var\big(T_K(v_0)\big)
=(nh)^{-2}\sum_{k=0}^K n_k\,\Var\big(\eta_k(v_0)\big)
=(\gamma nh)^{-1}\,\nu(u_0)
v_0^\top(\zeta_{0,2}\otimes\Psi(u_0))v_0\,(1+o(1)).
\]
 Moreover, since $n_k\asymp n/K$, it follows that
\[\bbE \left[\sum_{k} n_k^{3/2}W(t_k)\right]\asymp (Kh/\gamma)\,(n/K)^{3/2}\,.\]
Assumption~\ref{Assump:VCM-UX-Generalized}  (bounded moments of $Y$) implies the term
$\bbE|\eta_k(v_0)|^3$ is upper bounded.
Hence the Lyapunov condition is verified:
\[
\frac{\sum_{k=0}^K \bbE\left[\left|\frac{n_k^{1/2}}{nh}\eta_k(v_0)\right|^3\right]}
     {\Var(T_K(v_0))^{3/2}}
\lesssim
\frac{(Kh/\gamma)\,(n/K)^{3/2}/(nh)^3}
     {\big((nh)^{-1}\cdot \gamma^{-1}\big)^{3/2}}
\asymp \sqrt{\frac{\gamma}{K h}}\ \longrightarrow\ 0\,.
\]
Therefore, by Lyapunov’s CLT,
\[
\frac{T_K(v_0)}{\sqrt{\Var(T_K(v_0))}}\ \xrightarrow[]{d}\ \cN(0,1).
\]
By the Cramér--Wold device,
\begin{equation}\label{eq:CLT-barDelta}
\sqrt{\gamma nh}\big(\bar\Delta-\bbE\bar\Delta\big)
\ \xrightarrow[]{d}\ 
\cN\left(0,\ \zeta_{0,2}\otimes\Psi(u_0)
\nu(u_0)\right).
\end{equation}

\noindent
\underline{\textbf{Collecting results.}}
Combine Equations \eqref{eq:barDelta-mean}, \eqref{eq:barLambda-mean}, and \eqref{eq:CLT-barDelta}. The bias term is
\[
\bb_{\GDVCM}(u_0)
=\bA_l(\bbE\bar\Lambda)^{-1}\bbE\bar\Delta
= O(h^\beta)
\]
and the asymptotic variance of the centered estimator is
\[
\Omega_{\GDVCM}(u_0)
= \nu(u_0)\,
\be_{1,l+1}^\top\,\zeta_{0,1}^{-1}\zeta_{0,2}\zeta_{0,1}^{-1}\,\be_{1,l+1}\ \otimes\ \Psi(u_0)^{-1}.
\]
Therefore,
\[
\sqrt{\frac{nh}{\gamma}}\Big(\hat\btheta_{\GDVCM}(u_0)-\btheta(u_0)-\bb_{\GDVCM}(u_0)\Big)
\ \xrightarrow[]{d} \cN\big(0,\ \Omega_{\GDVCM}(u_0)\big).
\]
\end{proof}

\subsection{Proof of Proposition \ref{prop:asymp-LR-VCM}}\label{proof:prop:asymp-LR-VCM}

\begin{proof}
The asymptotic distribution for the linear DVCM is a \emph{special case} of the generalized DVCM result (see Appendix~\ref{proof:prop:asymp-GLM-VCM} for the proof under a more general setup). Below, we highlight only the differences from the general proof. Recall that the model is: 
$$
Y_{ki}=X_{ki}^\top\btheta(U_k)+\varepsilon_{ki}, \quad \bbE[\varepsilon_{ki}\mid X_{ki},U_k]=0, \quad \Var(\varepsilon_{ki}\mid X_{ki},U_k)=\sigma^2(U_k) \,.
$$
For linear regression, our loss function $\ell(\eta, y)$ is the squared-error loss, i.e. $\ell(\eta, y) = (y - \eta)^2/2$, where $\eta = X^\top \theta$. As a consequence, the derivatives of the loss function (w.r.t. $\eta$) satisfy: 
\[
s_1(\eta,y)=\eta-y,\quad s_2(\eta,y)=1,\quad s_3(\eta,y)= 0 \,.
\]
Furthermore, we have $b(x) = x^2/2$, which implies $b'(x) = x, b''(x) = 1$ and $b'''(x) = 0$. 

\noindent
\underline{\textbf{Main differences between the two proofs.}}
In case of linear regression, the term with $s_3$ in the proof of Proposition \ref{prop:asymp-GLM-VCM} vanishes as $s_3 \equiv 0$. Furthermore, the definition of the conditional variance $\Psi(u)$ simplifies to: 
$$
\Psi(u) = \bbE\left[XX^\top b''(X^\top \btheta(u)) \mid U = u\right] = \bbE[XX^\top \mid U = u] \,,
$$
and the variance term $\nu(u)$ becomes $\sigma^2(u)$. 

With the same idea in the proof of Proposition \ref{prop:asymp-GLM-VCM}, we construct the function $D_n(\delta)$ in a way such that the minimizer $\hat\delta$ of $D_n(\delta)$ satisfy $\bA_l\,\hat\delta = r_{n}^{-1}\big(\hat\btheta_{\DVCM}(u_0)-\btheta(u_0)\big)$, with $r_n = (nh/\gamma)^{-1/2}$. Because $s_3\equiv 0$, the Taylor remainder $R_n(\delta)$ in the generalized DVCM proof (Equation \eqref{eq:Dn-expand}) vanishes, and consequently we have:
\[
D_n(\delta)=-(nh)\,r_{n}\,\delta^\top\bar\Delta+\frac12(nh)\,r_{n}^2\,\delta^\top\bar\Lambda\,\delta,
\]
where
$$
\bar \Lambda \coloneqq \frac{1}{nh}\sum_{k=0}^K Z_{ki}Z_{ki}^\top\,W(t_k), \quad \bar \Delta \coloneqq \frac{1}{nh}\sum_{k=0}^K (Y_{ki}-Z_{ki}^\top\bar\btheta(u_0)) Z_{ki}\,W(t_k) \,.
$$
Using the fact that the optimal $\delta$ satisfies $\nabla D_n(\hat \delta) = 0$, we have: 
\begin{equation}\label{eq:linearization-linear-DVCM}
    \hat\delta=\bar\Lambda^{-1}\bar\Delta \implies 
r_{n}^{-1}\big(\hat\btheta_{\DVCM}(u_0)-\btheta(u_0)\big)
= r_{n}^{-1}\,\bA_l\,\bar\Lambda^{-1}\bar\Delta.
\end{equation}
where the RHS follows from the definition of $\delta$. 
The asymptotic limits of $\bbE[\Delta_k \mid U_k = u_k]$, $\Var[\Delta_k \mid U_k = u_k]$, and $\mathbb{E}[\Lambda_k \mid U_k = u_k]$ has been shown in Lemma~\ref{lem:cond-normality-DVCM}. 
Integrating the above quantities w.r.t. $U_k$ with density $f_\gamma(u)=(1/\gamma)f\big((u-u^*)/\gamma\big)$ and applying Equation \eqref{eq:terms-order}, the values of $\bbE[\bar\Delta]$,  $\Var(\bar \Delta)$, and $\bbE[\bar\Lambda]$ become: 
\begin{align*}
    & \bbE[\bar\Delta]= O(\gamma^{-1}h^\beta),\\
    & \Var(\bar\Delta)=(\gamma nh)^{-1}\,\sigma^2(u_0)\,\zeta_{0,2}\otimes\Psi(u_0)\,(1+o(1)),\\
    & \bbE\bar\Lambda=\gamma^{-1}\zeta_{0,1}\otimes\Psi(u_0)\,(1+o(1))\,.
\end{align*}
In particular, the change–of–variables and integration steps are worked out in detail in Appendix~\ref{proof:prop:asymp-GLM-VCM} for the generalized DVCM case, and the same calculations apply here exactly.
Note that the quantities $\zeta_{r, s}$ does not depend on the loss function and consequently remain fixed. Since $\Delta_k$ are sums of $\varepsilon_{ki}Z_{ki}W(t_k)$ with mean zero conditional on $(X_{ki},U_k)$,
the Lyapunov (or Lindeberg–Feller) CLT applies, yielding
\[
\sqrt{\gamma nh}\big(\bar\Delta-\bbE\bar\Delta\big)\ \xrightarrow{d}\
\cN\big(0,\ \zeta_{0,2}\otimes\Psi(u_0)\,\sigma^2(u_0)\big).
\]
Now we have established the asymptotic normality of $\bar \Delta$. Going back to Equation \eqref{eq:linearization-linear-DVCM}, it can be easily verified that its linear transformation $ r_{n}^{-1}\,\bA_l\,\bar\Lambda^{-1}\bar\Delta$ is such that
\[ r_n^{-1}\left(\hat\btheta_{\DVCM}(u_0) - \btheta(u_0) - \bb_{\DVCM}(u_0)\right) 
    \xrightarrow[]{d} \cN\left( 0, \Omega_{\DVCM}(u_0)\right)\]
where
\begin{align*}
    \bb_{\DVCM}(u_0)
= O(h^\beta),\qquad
\Omega_{\DVCM}(u_0)
= \sigma^2(u_0)\,
\be_{1,l+1}^\top\,\zeta_{0,1}^{-1}\zeta_{0,2}\zeta_{0,1}^{-1}\,\be_{1,l+1}
\ \otimes\ \Psi(u_0)^{-1}\,.
\end{align*}
This completes the proof.
\end{proof}

\subsection{Proof of Lemma \ref{lem:Dconv-GLM-TF-0-infty}}\label{proof:lem:Dconv-GLM-TF-0-infty}

Recall the rates of $\hat\btheta_{\GLR}(u_0)$, $\hat\btheta_{\GDVCM}(u_0)$, and their ratio
\[
\hat\btheta_{\GLR}(u_0) - \btheta(u_0) =O_p(r_{\GLR}), \quad \hat\btheta_{\GDVCM}(u_0) - \btheta(u_0)=O_p(r_{\GDVCM}), \qquad \rho\coloneqq \frac{r_{\GLR}}{r_{\GDVCM}},
\]
and assume $Q\asymp_p\rho^2I$, i.e.\ $c\rho^2\le\lambda_{\min}(Q)\le\lambda_{\max}(Q)\le C\rho^2$
w.p. $\to1$. Set $r_{\TL}\coloneqq r_{\GLR}\wedge r_{\GDVCM}$.
Define
\[
D_n(\delta)\coloneqq L_n(\btheta(u_0)+r_{\TL}\delta)-L_n(\btheta(u_0)) \,,
\]
and 
\[
L_n(\alpha)\coloneqq \frac{1}{n_0}\sum_{i\in\cI_0^*}\ell(X_{0i}^\top\alpha,Y_{0i})
+\frac12\|\alpha-\hat\btheta_{\GDVCM}(u_0)\|_Q^2 \,.
\]

In this proof, we use $\|\cdot\|$ to denote $\ell_2$-norm of a vector and spectral norm of a matrix.

\centerline{\underline{\textbf{Part 1: Rate and consistency of TL.}}}

The key step is to show that for any \(\varepsilon>0\) there exists a large constant \(M =  M(\eps)\) such that
\begin{equation}\label{eq:goal-TL-rate}
    \Pr\left\{\ \inf_{\|\delta\|=M} D_n(\delta) > 0\ \right\} = \Pr\left\{\ \inf_{\|\delta\|=M} L_n(\btheta(u_0)+r_{\TL}\delta)>L_n(\btheta(u_0))\ \right\} \ge\ 1-\varepsilon.
\end{equation}
Since \(D_n(0)=0\) and \(D_n(\cdot)\) is convex, this implies that with probability at
least \(1-\varepsilon\) a global minimizer lies inside the ball \(\{\btheta(u_0)+r_{\TL}\delta:\ \|\delta\|\le M\}\). Consequently,
\[
\|\hat\btheta_{\TL}(u_0)-\btheta(u_0)\|=O_p(r_{\TL}). 
\]
\noindent\underline{\textbf{Expanding $D_n(\delta)$.}} A Taylor expansion of $L_n(\btheta(u_0)+r_{\TL}\delta)$ around $\btheta(u_0)$ gives
\begin{equation}
\label{eq:Dn_def}
D_n(\delta)=r_{\TL}\,\delta^\top G_n+\frac12 r_{\TL}^2\,\delta^\top H_n\delta+R_n(\delta),
\end{equation}
with
\begin{align*}
    G_n & = \frac{1}{n_0}\sum_{i\in\cI_0^*} s_1(X_{0i}^\top\btheta(u_0),Y_{0i})X_{0i} + Q(\btheta(u_0)-\hat\btheta_{\GDVCM}(u_0)) \,, \\
    H_n & = \frac{1}{n_0}\sum_{i\in\cI_0^*} s_2(X_{0i}^\top\btheta(u_0),Y_{0i})X_{0i}X_{0i}^\top+Q \,, \\
    R_n(\delta) & = \frac{r_{\TL}^3}{6n_0}\sum_{i\in\cI_0^*}
s_3(\tilde\eta_{0i},Y_{0i})\,(X_{0i}^\top\delta)^3 \,,
\end{align*}
with $\tilde\eta_{0i}$ being some intermediate point between $X_{0i}^\top\btheta(u_0)$ and $X_{0i}^\top(\btheta(u_0)+r_{\TL}\delta)$.  

\noindent\underline{\textbf{Order of $G_n$, $H_n$, and $R_n$.}} It is immediate from the definition of the remainder term $R_n$ that for any $M > 0$: 
\begin{equation}
\label{eq:Rn_UB}
\sup_{\|\delta\|\le M}|R_n(\delta)|
\le \frac{r_{\TL}^3 M^3}{6n_0}\sum_{i\in\cI_0^*}\big|s_3(\tilde\eta_{0i},Y_{0i})\big|\,\|X_{0i}\|^3
=O_p(r_{\TL}^3 M^3),
\end{equation}
as $\bbE|s_3(\cdot,Y)|^4<\infty$ and $\|X\|<\infty$ by Assumption \ref{Assump:VCM-UX-Generalized}. Our next goal is to obtain the order of $G_n$ and $H_n$. Note that $G_n$ involves the term $Q(\btheta(u_0)-\hat\btheta_{\GDVCM}(u_0))$. From our definition of $r_{\GDVCM}$, we have: 
$$
\btheta(u_0)-\hat\btheta_{\GDVCM}(u_0)=O_p(r_{\GDVCM}) \,,
$$
and by our assumption on $Q$, we have $Q \asymp_p \rho^2 I$. Furthermore, from the standard property of the (centered) score function on the target domain, we have: 
$$
\frac{1}{n_0}\sum_{i\in\cI_0^*} s_1(X_{0i}^\top\btheta(u_0),Y_{0i})X_{0i} = O_p(n_0^{-1/2}) = O_p(r_{\GLR}) \,.
$$
Therefore, combining these orders, we have: 
\begin{equation}\label{eq:Gn-rate}
G_n=O_p\big(r_{\GLR}+\rho^2 r_{\GDVCM}\big).
\end{equation}
Furthermore, we have $\lambda_{\min}(\Psi(u_0)) \geq c_0'$ by Assumption \ref{Assump:VCM-UX-Generalized}. Therefore, by the law of large numbers: 
\begin{align}
\label{eq:Hn-curv}
    \lambda_{\min}(H_n) & \ge \lambda_{\min}(\Psi(u_0) +  Q) + o_p(1) \notag \\
    & \ge \lambda_{\min}(\Psi(u_0)) + \lambda_{\min}(Q) + o_p(1) \notag \\
    & = O_p(1 + \rho^2) 
\end{align}
\noindent\underline{\textbf{Equivalent characterizations of $r_{\TL}$}.}
Observe that
\[
\frac{r_{\GLR}+\rho^2 r_{\GDVCM}}{1+\rho^2}
=\begin{cases}
r_{\GLR}\dfrac{1+\rho}{1+\rho^2}\le 2\,r_{\GLR}, & \rho\le1,\\
r_{\GDVCM}\dfrac{\rho+\rho^2}{1+\rho^2}\le 2\,r_{\GDVCM}, & \rho\ge1,
\end{cases}
\]
and thus we obtain lower bound on $r_{\TL}$
\begin{equation}\label{eq:rates-inequality-1}
    \frac{r_{\GLR}+\rho^2 r_{\GDVCM}}{1+\rho^2} \le 2( r_{\GLR}\wedge r_{\GDVCM}) \coloneqq 2r_{\TL}\ \implies\  r_{\TL} \geq \frac{r_{\GLR}+\rho^2 r_{\GDVCM}}{2(1+\rho^2)} \,.
\end{equation}
Furthermore, we also have: 
\[
r_{\TL} = \min\{r_{\GLR}, r_{\GDVCM}\}\le\frac{r_{\GLR}+\rho^2\, r_{\GDVCM}}{1+\rho^2}.
\]
Combining this with \eqref{eq:rates-inequality-1} yields the bound
\begin{equation}\label{eq:rates-inequality-2}
    \frac{r_{\GLR}+\rho^2 r_{\GDVCM}}{2(1+\rho^2)}
\ \le\ r_{\TL}\ \le\
\frac{r_{\GLR}+\rho^2 r_{\GDVCM}}{1+\rho^2} \ \implies \ r_{\TL} \asymp \frac{r_{\GLR}+\rho^2 r_{\GDVCM}}{1+\rho^2}\,.
\end{equation}

\noindent\underline{\textbf{Verifying Objective \eqref{eq:goal-TL-rate}.}} Recall the definition of $D_n(\delta)$ from Equation \eqref{eq:Dn_def}. We have on the sphere $\{\|\delta\| = M\}$: 
\begin{equation}\label{eq:Dn-lb-1}
    \begin{split}
        \inf_{\|\delta\|=M} D_n(\delta) & \ge - \underbrace{r_{\TL}\|G_n\|}_{\coloneqq A_n \geq 0}\,M+ \underbrace{\frac12 r_{\TL}^2 \lambda_{\min}(H_n)}_{\coloneqq B_n \geq 0} M^2 - \sup_{\|\delta\|\le M}|R_n(\delta)| \\
        & = -A_n M + B_n M^2 - \sup_{\|\delta\|\le M}|R_n(\delta)|\,.
    \end{split}
\end{equation}
The order of $A_n$ and $B_n$ can be derived from the order of $G_n$ and $\lambda_{\min}(H_n)$ respectively. In particular, using the order of $G_n$, as established in Equation \eqref{eq:Gn-rate}, we have: 
$$
A_n = r_{\TL}\|G_n\| = O_p\left(r_{\TL}\left(r_{\GLR}+\rho^2 r_{\GDVCM}\right)\right).
$$
Moreove, the lower bound on the minimum eigenvalue of $H_n$ (Equation \eqref{eq:Hn-curv}) and equation on $r_{\TL}$ in Equation \eqref{eq:rates-inequality-2} yield the following order of $B_n$:
\[
B_n = \frac12 r_{\TL}^2 \lambda_{\min}(H_n) = \frac12 r_{\TL}\cdot \left(\frac{r_{\GLR}+\rho^2 r_{\GDVCM}}{2(1+\rho^2)}\right)\, \Omega_p(1+\rho^2) = \Omega_p\left(r_{\TL}\left(r_{\GLR}+\rho^2 r_{\GDVCM}\right)\right) \,.
\]
Treating $M$ as a constant, the upper bound on the remainder term (Equation \eqref{eq:Rn_UB}) yields
$$
\sup_{\|\delta\|\le M}|R_n(\delta)| =  O_p\Big(r_{\TL}^3\Big)\,,
$$
and by Equation \eqref{eq:rates-inequality-2}, as well as the immediate relation $1+\rho^2 \gtrsim 1$,
\[
O_p(r_{\TL}^3) = o_p(r_{\TL}^2) = o_p\left(r_{\TL}\frac{r_{\GLR}+\rho^2 r_{\GDVCM}}{1+\rho^2}\right) = o_p\left(r_{\TL} (r_{\GLR}+\rho^2 r_{\GDVCM})\right)\,.
\]
Hence the remainder $\sup_{\|\delta\|\le M}|R_n(\delta)|$ is dominated by both $A_n$ and $B_n$, so Equation \eqref{eq:Dn-lb-1} is written as
\[\inf_{\|\delta\|=M} D_n(\delta) \geq  -A_n M + B_n M^2 + o\left(A_n \wedge B_n\right)\,.\]
Both $A_n$ and $B_n$ are of order $O_p\left(r_{\TL}\left(r_{\GLR}+\rho^2 r_{\GDVCM}\right)\right)$ so by choosing large enough $M$, $B_nM^2$ dominates $-A_nM$ uniformly in $M$. Therefore we have verified that by setting $M$ to be large enough, the condition in \eqref{eq:goal-TL-rate} is verified, and this further implies
\[
\big\|\hat\btheta_{\TL}(u_0)-\btheta(u_0)\big\|
= O_p(r_{\TL})\,.
\]

\centerline{\underline{\textbf{Part 2: Asymptotic distribution.}}}

Let $s_j(\eta,y)=\partial^j\ell(\eta,y)/\partial\eta^j$,
$\btau_{\TL}\coloneqq \hat\btheta_{\TL}(u_0)-\btheta(u_0)$ and
$\btau_{\GDVCM}\coloneqq \hat\btheta_{\GDVCM}(u_0)-\btheta(u_0)$.
The first–order condition is
\begin{equation}\label{eq:TL-FOC-again}
\frac{1}{n_0}\sum_{i\in\cI_0^*} s_1\big(X_{0i}^\top\hat\btheta_{\TL}(u_0),Y_{0i}\big)X_{0i}
+Q\big(\hat\btheta_{\TL}(u_0)-\hat\btheta_{\GDVCM}(u_0)\big)=0.
\end{equation}
A second–order Taylor expansion of $s_1$ at first argument around $\btheta(u_0)$ gives
\[
s_1\big(X_{0i}^\top\hat\btheta_{\TL}(u_0),Y_{0i}\big)
=s_1\big(X_{0i}^\top\btheta(u_0),Y_{0i}\big)
+s_2\big(X_{0i}^\top\btheta(u_0),Y_{0i}\big)\,X_{0i}^\top\btau_{\TL}
+\frac12\,s_3(\bar\eta_{0i},Y_{0i})\big(X_{0i}^\top\btau_{\TL}\big)^2,
\]
where $\bar\eta_{0i}$ lies between $X_{0i}^\top\btheta(u_0)$ and $X_{0i}^\top\hat\btheta_{\TL}(u_0)$.
Plug into \eqref{eq:TL-FOC-again} and collect terms to obtain
\begin{equation}\label{eq:FOC-linearized}
\frac{1}{n_0}\bX_0^\top S_1
+\Big(\frac{1}{n_0}\bX_0^\top S_2\bX_0+Q\Big)\btau_{\TL}
+\frac{1}{2n_0}\sum_{i\in\cI_0^*} s_3(\bar\eta_{0i},Y_{0i})\,X_{0i}\big(X_{0i}^\top\btau_{\TL}\big)^2
=Q\btau_{\GDVCM},
\end{equation}
where $\bX_0 \in \reals^{n_0 \times p}$ with $X_{0i}$ being its $i^{th}$ row, $S_1 \in \reals^{n_0}$ with entries being $s_1(X_{0i}^\top\btheta(u_0),Y_{0i})$, and $S_2 \in \reals^{n_0 \times n_0}$ is a diagonal matrix with entries $s_2(X_{0i}^\top\btheta(u_0),Y_{0i})$. By $\bbE|s_3(\cdot,Y)|^4<\infty$ and $\|X\|<\infty$
\[
\frac{1}{n_0}\sum_{i\in\cI_0^*}\big|s_3(\bar\eta_{0i},Y_{0i})\big|\,\|X_{0i}\|^3
=O_p(1),
\]
and thus
\begin{equation}\label{eq:TL-reminder-Op}
    \Big\|\frac{1}{2n_0}\sum_{i} s_3(\bar\eta_{0i},Y_{0i})\,X_{0i}\big(X_{0i}^\top\btau_{\TL}\big)^2\Big\|
\le \frac{1}{n_0}\sum_i \big|s_3(\bar\eta_{0i},Y_{0i})\big|\,\|X_{0i}\|^3\,\|\btau_{\TL}\|^2
= O_p\big(\|\btau_{\TL}\|^2\big).
\end{equation}
By the law of large numbers and Assumption~\ref{Assump:VCM-UX-Generalized},
\begin{equation}\label{eq:S2}
    \frac{1}{n_0}\bX_0^\top S_2\bX_0=\Psi(u_0)+o_p(1).
\end{equation}
Recall $\btau_{\GLR}\coloneqq \hat\btheta_{\GLR}(u_0)-\btheta(u_0)$. The score equation of generalized linear model gives
\[
0=\frac{1}{n_0}\sum_{i\in\cI_0^*} s_1(X_{0i}^\top\hat\btheta_{\GLR},Y_{0i})X_{0i}
= \frac{1}{n_0}\bX_0^\top S_1+\Big(\frac{1}{n_0}\bX_0^\top S_2\bX_0\Big)\btau_{\GLR}+R_{n_0}.
\]
Reorganizing terms yields
\[
-\frac{1}{n_0}\bX_0^\top S_1=\Big(\frac{1}{n_0}\bX_0^\top S_2\bX_0\Big)\btau_{\GLR}+R_{n_0}.
\]
By standard theory of generalized linear model, it follows that \(\frac{1}{n_0}\bX_0^\top S_2\bX_0=\Psi(u_0)+o_p(1)\) and $\|R_{n_0}\|=o_p(r_{\GLR})$. Therefore,
\begin{equation}\label{eq:S1}
    -\frac{1}{n_0}\bX_0^\top S_1=\Psi(u_0)\btau_{\GLR}+o_p(r_{\GLR}).
\end{equation}
Rearranging \eqref{eq:FOC-linearized} yields
\[
r_{\TL}^{-1}\btau_{\TL} = r_{\TL}^{-1}\Big(\frac{1}{n_0}\bX_0^\top S_2\bX_0+Q\Big)^{-1} \Big(Q\btau_{\GDVCM}-\frac{1}{n_0}\bX_0^\top S_1\Big) + \tilde R_n
\]
with 
\[\tilde R_n = -  r_{\TL}^{-1}\Big(\frac{1}{n_0}\bX_0^\top S_2\bX_0+Q\Big)^{-1} \frac{1}{2n_0}\sum_{i\in\cI_0^*} s_3(\bar\eta_{0i},Y_{0i})\,X_{0i}\big(X_{0i}^\top\btau_{\TL}\big)^2\,.\]
Substituting the Equations \eqref{eq:TL-reminder-Op} and \eqref{eq:S2}, it follows from the consistency of $\hat\btheta_{\TL}(u_0)$ that
\[
\tilde R_n =r_{\TL}^{-1}\Big(\Psi(u_0)+Q\Big)^{-1}O_p(\|\btau_{\TL}\|^2) =  O_p(\|\btau_{\TL}\|) = o_p(1)\,,
\]
and thus
\[
r_{\TL}^{-1}\btau_{\TL} = r_{\TL}^{-1}\Big(\frac{1}{n_0}\bX_0^\top S_2\bX_0+Q\Big)^{-1} \Big(Q\btau_{\GDVCM}-\frac{1}{n_0}\bX_0^\top S_1\Big)+o_p(1)\,.
\]
Now we substitute the Equations \eqref{eq:S2} and \eqref{eq:S1}, and obtain
\begin{equation}\label{eq:TL-linearization-0}
    r_{\TL}^{-1}\btau_{\TL}
=r_{\TL}^{-1}\Big(\Psi(u_0)+Q\Big)^{-1}\Big\{\,Q\,\btau_{\GDVCM}+\Psi(u_0)\btau_{\GLR}\Big\}+o_p(1).
\end{equation}
By sample splitting, $\btau_{\GLR}$ is independent of $\btau_{\GDVCM}$, based on which we will derive the regime–specific asymptotic normality.

\noindent\underline{\emph{Regime $\rho\to0$ (GLR–dominated).}}
From \(Q=O_p(\rho^2)\), we have \((\Psi(u_0)+Q)^{-1}=\Psi^{-1}(u_0)+o_p(1)\) and
\[
\|\Psi^{-1}(u_0)Q\btau_{\GDVCM}\|
=O_p(\rho^2\,\|\btau_{\GDVCM}\|)
=O_p(\rho^2 r_{\GDVCM})
=O_p(r_{\GLR}\,\rho)=o_p(r_{\GLR}).
\]
Hence \(\btau_{\TL}=\btau_{\GLR}+o_p(r_{\GLR})\). By Slutsky's theorem and part 1 of Proposition \ref{prop:asymp-GLM-VCM}
\[
r_{\GLR}^{-1}\big(\hat\btheta_{\TL}(u_0)-\btheta(u_0)\big)
\ \xrightarrow[]{d}\ \cN\big(0,\ \Omega_{\GLR}(u_0)\big).
\]

\noindent\underline{\emph{Regime $\rho\to\infty$ (DVCM–dominated).}}
Starting from
\[
\btau_{\TL}
=(\Psi(u_0)+Q)^{-1}\left\{\,Q\,\btau_{\GDVCM}+\Psi(u_0)\,\btau_{\GLR}\right\}+o_p(r_{\TL}),
\]
use the identity
\[
(\Psi(u_0)+Q)^{-1}-Q^{-1}
=-(\Psi(u_0)+Q)^{-1}\,\Psi(u_0)\,Q^{-1}.
\]
Since $\lambda_{\min}(Q)\asymp_p \rho^2$ and $\Psi(u_0)\succeq c_0'I$, 
\[
\|(\Psi(u_0)+Q)^{-1}\|\le \|Q^{-1}\|=O_p(\rho^{-2}),
\qquad
\|\Psi(u_0)\|=O(1),
\]
hence
\begin{equation}\label{eq:Qinv-diff}
    \big\|(\Psi(u_0)+Q)^{-1}-Q^{-1}\big\|
\le \|(\Psi(u_0)+Q)^{-1}\|\,\|\Psi(u_0)\|\,\|Q^{-1}\|
=O_p(\rho^{-4})
=o_p\big(\|Q^{-1}\|\big).
\end{equation}
Next, apply the identity $(\Psi(u_0)+Q)^{-1}Q=I-(\Psi(u_0)+Q)^{-1}\Psi(u_0)$ to get
\[
\btau_{\TL}
=\btau_{\GDVCM}-(\Psi(u_0)+Q)^{-1}\Psi(u_0)\,\btau_{\GDVCM}
+(\Psi(u_0)+Q)^{-1}\Psi(u_0)\,\btau_{\GLR}+o_p(r_{\GDVCM}).
\]
The second and third terms are negligible by Equation \eqref{eq:Qinv-diff}:
\[
\big\|(\Psi(u_0)+Q)^{-1}\Psi(u_0)\,\btau_{\GDVCM}\big\|
\le \|Q^{-1}\|\,\|\Psi(u_0)\|\,\|\btau_{\GDVCM}\|\{1+o_p(1)\}
=O_p(\rho^{-2} r_{\GDVCM})
=o_p(r_{\GDVCM}),
\]
and, using $r_{\GLR}=\rho\,r_{\GDVCM}$ and Equation \eqref{eq:Qinv-diff},
\[
\big\|(\Psi(u_0)+Q)^{-1}\Psi(u_0)\,\btau_{\GLR}\big\|
\le \|Q^{-1}\|\,\|\Psi(u_0)\|\,\|\btau_{\GLR}\|\{1+o_p(1)\}
=O_p(\rho^{-2} r_{\GLR})
=O_p(r_{\GDVCM}/\rho)
=o_p(r_{\GDVCM}).
\]
Therefore,
\[
\btau_{\TL}=\btau_{\GDVCM}+o_p(r_{\GDVCM}),
\]
 and  by Slutsky's theorem and part 2 of Proposition \ref{prop:asymp-GLM-VCM}
\[
r_{\GDVCM}^{-1}\Big(\hat\btheta_{\TL}(u_0)-\btheta(u_0)-\bb_{\GDVCM}(u_0)\Big)
\ \xrightarrow[]{d}\ \cN\big(0,\ \Omega_{\GDVCM}(u_0)\big).
\]

\section{Discussion of Assumption \ref{Assump:VCM-UX}}\label{sec:eigen-assum}

In this section, we examine the plausibility and technical strength of Assumption~\ref{Assump:VCM-UX}, which requires boundedness of $\|\bA_l(\bZ^\top\bW\bZ)^{-1}\bA_l^\top\|_2$. Our analysis considers two asymptotic regimes based on the number of tasks $K$:
(i) a fixed-$K$ regime with $n_0\to\infty$, and
(ii) a growing-task regime with $K\to\infty$.

The corresponding results are established in Proposition~\ref{thm:ZWZ-n0} and Proposition~\ref{thm:ZWZ-K}, respectively. In both regimes, we show that Assumption~\ref{Assump:VCM-UX} is satisfied in the sense of almost sure convergence under mild conditions, thereby providing theoretical justification for its use in our framework.

\begin{proposition}
\label{thm:ZWZ-n0}
Assume that
\[
\lambda_{\min}\!\big(\bbE[X_{0i}X_{0i}^\top]\big)\ \ge\ \kappa,
\qquad
\frac{n_0}{\sum_{k=0}^K n_k}\ \ge\ c_0
\]
for some constants $\kappa,c_0>0$. Then it holds that
\[
\Pr\left(\limsup_{n_0 \to \infty}\big\|\bA_l(\bZ^\top\bW\bZ)^{-1}\bA_l^\top\big\|_2 \le C\right) = 1
\]
for some constant $C>0$.
\end{proposition}

\begin{proof}
Recall
\[
\bZ^\top\bW\bZ
=
w_0\sum_{i\in\cI_0} \bZ_{0i}\bZ_{0i}^\top
+
\sum_{k=1}^K w_k\sum_{i\in\cI_k} \bZ_{ki}\bZ_{ki}^\top,
\qquad
w_k=\frac{W\!\big((U_k-u_0)/h\big)}{S_h},\ 
S_h=\sum_{k=0}^K n_k W\!\Big(\frac{U_k-u_0}{h}\Big).
\]
Using $\bZ_{0i}=\bA_l^\top X_{0i}$, define
\[
B := w_0\sum_{i\in\cI_0} X_{0i}X_{0i}^\top,
\qquad
M := \sum_{k=1}^K w_k\sum_{i\in\cI_k} \bZ_{ki}\bZ_{ki}^\top,
\]
so that $\bZ^\top\bW\bZ=\bA_l^\top B\bA_l+M$.

Let $T:=\bA_l M^{-1}\bA_l^\top\succeq 0$. The Woodbury identity yields
\[
\bA_l(\bZ^\top\bW\bZ)^{-1}\bA_l^\top
=
T - T(B^{-1}+T)^{-1}T
\preceq B^{-1},
\]
hence
\begin{equation}\label{eq:block-by-Cinv}
\big\|\bA_l(\bZ^\top\bW\bZ)^{-1}\bA_l^\top\big\|_2
\le \|B^{-1}\|_2
= \lambda_{\min}(B)^{-1}.
\end{equation}

Since $w_0=W(0)/S_h$ and $S_h\le W(0)\sum_{k=0}^K n_k$, we have
\[
w_0=\frac{W(0)}{S_h}\ \ge\ \frac{1}{\sum_{k=0}^K n_k}.
\]
Therefore,
\[
\lambda_{\min}(B)
=
w_0\,\lambda_{\min}\!\left(\sum_{i\in\cI_0}X_{0i}X_{0i}^\top\right)
\ge
\frac{1}{\sum_{k=0}^K n_k}\,
\lambda_{\min}\!\left(\sum_{i\in\cI_0}X_{0i}X_{0i}^\top\right).
\]
Using $\sum_{k=0}^K n_k \le n_0/c_0$, this becomes
\begin{equation}\label{eq:C-lmin-reduce}
\lambda_{\min}(B)
\ge
c_0\,\lambda_{\min}\!\left(\frac{1}{n_0}\sum_{i\in\cI_0}X_{0i}X_{0i}^\top\right).
\end{equation}

Let $\Sigma_0:=\bbE[X_{0i}X_{0i}^\top]$ with $\lambda_{\min}(\Sigma_0)\ge \kappa$.
Since $X_{0i}X_{0i}^\top$ are i.i.d.\ PSD matrices with
$\lambda_{\max}(X_{0i}X_{0i}^\top)=\|X_{0i}\|_2^2\le 1$,
the matrix Chernoff bound \citep{tropp2015} gives that for any $\eta\in(0,1)$,
\[
\Pr\!\left\{
\lambda_{\min}\!\left(\sum_{i\in\cI_0}X_{0i}X_{0i}^\top\right)
\le (1-\eta)\,n_0\,\kappa
\right\}
\le
p\exp\!\left(-\frac{\eta^2\,n_0\,\kappa}{2}\right).
\]
Taking $\eta=1/2$ yields
\begin{equation}\label{eq:chernoff-cov}
\Pr\!\left\{
\lambda_{\min}\!\left(\frac{1}{n_0}\sum_{i\in\cI_0}X_{0i}X_{0i}^\top\right)
\le \frac{\kappa}{2}
\right\}
\le
p\exp\!\left(-c\,n_0\,\kappa\right)
\end{equation}
for a constant $c>0$.

On the complement of the event in \eqref{eq:chernoff-cov}, we have
$\lambda_{\min}(\frac{1}{n_0}\sum X_{0i}X_{0i}^\top)\ge \kappa/2$ and thus by
\eqref{eq:C-lmin-reduce},
\[
\lambda_{\min}(B)\ge c_0\cdot \frac{\kappa}{2}.
\]
Plugging into \eqref{eq:block-by-Cinv} gives
\[
\big\|\bA_l(\bZ^\top\bW\bZ)^{-1}\bA_l^\top\big\|_2
\le \frac{2}{c_0\kappa}=:C.
\]
Therefore, define $B_{n_0} = \Big\{\big\|\bA_l(\bZ^\top\bW\bZ)^{-1}\bA_l^\top\big\|_2 > C\Big\}$ we have
\[
\Pr(B_{n_0})
\le p\exp(-C'n_0\kappa),
\]
which implies for some $k_0 > 0$
\[
\sum_{n_0=k_0}^\infty \Pr(B_{n_0})
\le \sum_{n_0=k_0}^\infty p\exp(-C'n_0\kappa) < \infty.
\]
Applying Borel–Cantelli lemma proves the result.
\end{proof}

\begin{proposition}
\label{thm:ZWZ-K}
Assume that $\lambda_{\min}\!\big(\bbE[XX^\top\mid U=u]\big)\ge \lambda_0$ for all $u\in\cU$ and $Kh/\gamma \gtrsim \log(K)$, then it holds 
\[
\Pr\left(\limsup_{K \to \infty}\big\|\bA_l(\bZ^\top\bW\bZ)^{-1}\bA_l^\top\big\|_2 \le C\right) = 1
\]
for some constant $C>0$.
\end{proposition}

\begin{proof}
For a fixed $K$, define the quantities
\[
\bY_k
:=\sum_{i=1}^{n_k}\bY_{ki}
=\frac{W_k}{S_h}\Big(\sum_{i=1}^{n_k}X_{ki}X_{ki}^\top\Big)\otimes(r_kr_k^\top),
\qquad
\bS^{(K)}=\bZ^\top\bW\bZ = \sum_{k=1}^K\bY_k,
\]
where $W_k:=W((U_k-u_0)/h)$ is the uniform kernel indicator,
$S_h:= \sum_{k=1}^K n_k W_k$, and $r_k = \Phi_l((U_k-u_0)/h)$.

\medskip
\noindent\textbf{1. Spectral bound.}
Condition on $U_{1:K}$. Then $W_k,S_h,r_k$ are deterministic and $\{\bY_k\}_{k=1}^K$ are conditionally independent.
Moreover, by boundedness of $\|X_{ki}\|_2^2$ in Assumption \ref{Assump:VCM-UX} and $\|r_k\|_2^2$ on the kernel window,
there exists a constant $R>0$ such that
\[
\Big\|\sum_{i=1}^{n_k}X_{ki}X_{ki}^\top\Big\|_2\le R n_k,
\qquad
\|r_kr_k^\top\|_2=\|r_k\|_2^2\le R.
\]
Therefore, applying $n_k/\bar n \le b_0$ in  Assumption \ref{Assump:VCM-UX},
\[
\|\bY_k\|_2
\le \frac{W_k}{S_h}\Big\|\sum_{i=1}^{n_k}X_{ki}X_{ki}^\top\Big\|_2\|r_kr_k^\top\|_2
\le R^2\frac{n_k W_k}{S_h}
\le R^2\frac{n_k}{\sum_{s=1}^K n_s}
=\frac{R^2}{K}\frac{n_k}{\bar n}
\le \frac{R^2 b_0}{K}.
\]
Define a quantity (to be used later)
\begin{equation}\label{eq:Rblock}
R_0:=\frac{R^2 b_0}{K}.
\end{equation}

\medskip
\noindent\textbf{2. Conditional matrix Chernoff.}
Let $\bM(U):=\bbE[\bS^{(K)}\mid U_{1:K}]$ and $\mu(U):=\lambda_{\min}(\bM(U))$.
Applying the matrix multiplicative Chernoff bound \citep{tropp2015} conditional on $U_{1:K}$
yields, for any $\varepsilon\in(0,1)$,
\begin{equation}\label{eq:chernoff-S-cond}
\Pr\!\left(\lambda_{\min}(\bS^{(K)})\le (1-\varepsilon)\mu(U)\ \Big|\ U_{1:K}\right)
\le d_0\exp\!\left(-\frac{\varepsilon^2\,\mu(U)}{2R_0}\right).
\end{equation}
Now we want to lower bound $\mu(U)$
Using $\bbE[XX^\top\mid U=u]\succeq \lambda_0 I_p$,
\[
\bbE\!\left[\sum_{i=1}^{n_k}X_{ki}X_{ki}^\top\ \Big|\ U_k\right]
=n_k\,\bbE[XX^\top\mid U_k]\succeq n_k\lambda_0 I_p.
\]
Hence,
\begin{align*}
\bM(U)
&=\bbE[\bS^{(K)}\mid U_{1:K}]
=\sum_{k=1}^K \frac{W_k}{S_h}\,
\bbE\!\left[\sum_{i=1}^{n_k}X_{ki}X_{ki}^\top\Big|\ U_k\right]\otimes(r_kr_k^\top)\\
&\succeq
\lambda_0\, I_p\otimes
\sum_{k=1}^K \frac{n_k W_k}{S_h}\, r_kr_k^\top.
\end{align*}
Therefore,
\begin{equation}\label{eq:muU-lower}
\mu(U)\ \ge\ \lambda_0\,
\lambda_{\min}\!\left(\sum_{k=1}^K \alpha_k r_kr_k^\top\right),
\qquad
\alpha_k:=\frac{n_kW_k}{S_h}.
\end{equation}

\medskip
\noindent\textbf{3. Concentration of the moment matrix.}
Define the unnormalized $(l+1)\times(l+1)$ matrix
\[
\bG:=\sum_{k=1}^K n_k W_k\, r_kr_k^\top,
\qquad\text{so that}\qquad
\sum_{k=1}^K \alpha_k r_kr_k^\top=\frac{\bG}{S_h}\,.
\]
Hence
\begin{equation}\label{eq:ratio-min}
\lambda_{\min}\!\left(\sum_{k=1}^K \alpha_k r_kr_k^\top\right)=\frac{\lambda_{\min}(\bG)}{S_h}.
\end{equation}

We first lower bound $\lambda_{\min}(\bG)$.
Let $\bT_k:=n_k W_k r_kr_k^\top\succeq 0$, so $\bG=\sum_{k=1}^K \bT_k$.
Because $W_k$ and $r_k$ depend only on $U_k$, the matrices $\{\bT_k\}$ are independent.
Moreover, on the window $\|r_k\|_2^2\le R$ and $n_k\le b_0\bar n$, so
\[
\lambda_{\max}(\bT_k)\le n_k\|r_k\|_2^2\le b_0\bar n\,R=:R_A.
\]
Let $\bM_G:=\bbE[\bG]=\sum_{k=1}^K n_k\,\bbE[W_k r_kr_k^\top]$ and
$\mu_G:=\lambda_{\min}(\bM_G)$.
Using the density lower bound on $[u_0-h,u_0+h]$ and the change of variables $t=(u-u_0)/h$,
\[
\bbE[W_k r_kr_k^\top]
=\int_{u_0-h}^{u_0+h}\Phi_l\!\Big(\frac{u-u_0}{h}\Big)\Phi_l\!\Big(\frac{u-u_0}{h}\Big)^\top f_U(u)\,du
\succeq
\frac{a_0'h}{\gamma}\int_{-1}^1 \Phi_l(t)\Phi_l(t)^\top\,dt.
\]
There exists $\kappa_0$ such that $\int_{-1}^1\Phi_l(t)\Phi_l(t)^\top\,dt \ge \kappa_0>0$.
Then
\begin{equation}\label{eq:muG}
\mu_G\ \ge\ \frac{a_0'h}{\gamma}\,\kappa_0 \sum_{k=1}^K n_k
=\frac{a_0'h}{\gamma}\,\kappa_0\,K\bar n.
\end{equation}
Applying matrix Chernoff to $\bG=\sum_{k=1}^K \bT_k$ yields, for any $\eta\in(0,1)$,
\begin{equation}\label{eq:chernoff-G}
\Pr\!\left\{\lambda_{\min}(\bG)\le (1-\eta)\mu_G\right\}
\le (l+1)\exp\!\left(-\frac{\eta^2\mu_G}{2R_A}\right)
\le (l+1)\exp\!\left(-c_1\,\frac{Kh}{\gamma}\right),
\end{equation}
for some constant $c_1>0$.

\medskip
\noindent\textbf{4. Upper bound $S_h$}
We also control $S_h=\sum_{k=1}^K n_k W_k$ from above. Note that $0\le n_kW_k\le n_k\le b_0\bar n$
and $\{W_k\}$ are independent. Moreover, by the density upper bound in Assumption \ref{Assump:VCM-UX},
\[
\bbE[W_k]=\Pr(|U_k-u_0|\le h)=\int_{u_0-h}^{u_0+h}f_U(u)\,du\ \le\ \frac{2a_0h}{\gamma}.
\]
Thus
\[
\bbE[S_h]=\sum_{k=1}^K n_k\bbE[W_k]\ \le\ \frac{2a_0h}{\gamma}\sum_{k=1}^K n_k
=\frac{2a_0h}{\gamma}\,K\bar n.
\]
A Chernoff bound gives that for any $t\in(0,1)$,
\begin{equation}\label{eq:chernoff-Sh}
\Pr\!\left\{S_h\ge (1+t)\bbE[S_h]\right\}
\le \exp\!\left(-c_2\,\frac{Kh}{\gamma}\right)
\end{equation}
for some constant $c_2>0$.

\medskip
\noindent\textbf{5. probability on ``good event".}
Let
\[
\cE:=\Big\{\lambda_{\min}(\bG)\ge (1-\eta)\mu_G\Big\}\cap\Big\{S_h\le (1+t)\bbE[S_h]\Big\}.
\]
On $\cE$, using \eqref{eq:ratio-min}, \eqref{eq:muG}, and $\bbE[S_h]\le (2a_0h/\gamma)K\bar n$,
\[
\lambda_{\min}\!\left(\sum_{k=1}^K \alpha_k r_kr_k^\top\right)
=\frac{\lambda_{\min}(\bG)}{S_h}
\ge
\frac{(1-\eta)\mu_G}{(1+t)\bbE[S_h]}
\ge
\frac{1-\eta}{1+t}\cdot \frac{a_0'}{2a_0}\,\kappa_0
=: \tilde\kappa_0.
\]
Combining with \eqref{eq:muU-lower} gives
\begin{equation}\label{eq:muU-good}
\mu(U)\ \ge\ \lambda_0\,\tilde\kappa_0
\qquad\text{on }\cE.
\end{equation}
Moreover, by the union bound and \eqref{eq:chernoff-G}--\eqref{eq:chernoff-Sh},
\begin{equation}\label{eq:Eprob}
\Pr(\cE^c)\le\ C_0\exp\!\left(-c_0\,\frac{Kh}{\gamma}\right).
\end{equation}

\medskip
\noindent\textbf{6. Conclusion.}
Let $A:=\{\lambda_{\min}(\bS^{(K)})\le (1-\varepsilon)\lambda_0\tilde\kappa_0\}$. By the tower property,
\[
\Pr(A)=\bbE\big[\Pr(A\mid U_{1:K})\big]
\le \bbE\big[\Pr(A\mid U_{1:K})\indc(\cE)\big]+\Pr(\cE^c).
\]
On $\cE$, $\mu(U)\ge \lambda_0\tilde\kappa_0$ by \eqref{eq:muU-good}, hence by
\eqref{eq:chernoff-S-cond},
\[
\Pr(A\mid U_{1:K})\indc(\cE)
\le d_0\exp\!\left(-\frac{\varepsilon^2\,\mu(U)}{2R_0}\right)\indc(\cE)
\le d_0\exp\!\left(-\frac{\varepsilon^2\,\lambda_0\tilde\kappa_0}{2R_0}\right)\indc(\cE).
\]
Recalling $R_0=R^2 b_0/K$ from \eqref{eq:Rblock}, the exponent equals $-c_3 K$
for some $c_3>0$.
Therefore,
\[
\Pr(A)
\le d_0 e^{-c_3K}+\Pr(\cE^c)
\le d_0 e^{-c_3K}+C_0\exp\!\left(-c_0\,\frac{Kh}{\gamma}\right)
\le \exp\!\left(-C'\,\frac{Kh}{\gamma}\right),
\]
and we have shown that there exist constants $C,C'>0$ such that
\[
\Pr\!\left\{\lambda_{\min}\!\big(\bS^{(K)}\big)\le C\right\}
\ \le\
\exp\!\left(-C'\,\frac{Kh}{\gamma}\right).
\]
Let the RHS exponent $\frac{Kh}{\gamma} \ge \tfrac{M_0}{C'} \log(K)$ for some constant $M_0 > 1$ then it follows that, for some constant $k_0$,
\[\sum_{K=k_0}^\infty \Pr\!\left\{\lambda_{\min}\!\big(\bS^{(K)}\big)\le C\right\}
\ \le\
\sum_{K=k_0}^\infty \exp\!\left(-C'\,\frac{Kh}{\gamma}\right) \le \sum_{K=k_0}^\infty \exp\!\left(- M_0 \log(K)\right) = \sum_{K=k_0}^\infty K^{-M_0} < \infty\]
Using Borel–Cantelli lemma proves the almost sure result
\[\lambda_{\min}\!\big(\bZ^\top\bW\bZ\big) \ge C\]
which further implies
\[
\big\|\bA_l(\bZ^\top\bW\bZ)^{-1}\bA_l^\top\big\|_2 \le \tilde C.
\]
\end{proof}

\section{Additional Simulation Results for Section \ref{sec:sim}}\label{sec:add-sim}

This section presents supplementary simulation results corresponding to Section~\ref{sec:sim}. The MSE $\pm$ standard deviation for the analysis in Section~\ref{subsec:SLID}, computed across different random seeds, is reported in Table~\ref{tab:linear-real-data}. The cross-entropy loss $\pm$ standard deviation for the analysis in Section~\ref{sec:adult}, also computed across different random seeds, is reported in Table~\ref{tab:logistic-real-data}. 

We observe that the transfer learning (TL) estimator exhibits adaptivity, effectively selecting the better-performing estimator between the two baselines across different settings.

\begin{table}[ht]
\centering
\caption{Estimation performance across $U$ values for experiments in Section~\ref{subsec:SLID}. Entries are mean $\pm$ std.}
\label{tab:linear-real-data}
\begin{tabular}{c|ccc}
\toprule
$U$ & DVCM & GLR & TL \\
\midrule
0.05 & $0.1173 \pm 0.0146$ & $0.1084 \pm 0.0158$ & $0.1092 \pm 0.0151$ \\
0.15 & $0.1441 \pm 0.0102$ & $0.1438 \pm 0.0116$ & $0.1450 \pm 0.0103$ \\
0.25 & $0.1710 \pm 0.0085$ & $0.1599 \pm 0.0091$ & $0.1599 \pm 0.0095$ \\
0.35 & $0.1630 \pm 0.0130$ & $0.1561 \pm 0.0142$ & $0.1581 \pm 0.0132$ \\
0.45 & $0.1771 \pm 0.0176$ & $0.1765 \pm 0.0200$ & $0.1757 \pm 0.0196$ \\
0.55 & $0.1371 \pm 0.0086$ & $0.1390 \pm 0.0092$ & $0.1392 \pm 0.0093$ \\
0.65 & $0.1516 \pm 0.0224$ & $0.1538 \pm 0.0199$ & $0.1537 \pm 0.0232$ \\
0.75 & $0.1827 \pm 0.0347$ & $0.1840 \pm 0.0357$ & $0.1864 \pm 0.0351$ \\
0.85 & $0.1810 \pm 0.0234$ & $0.1835 \pm 0.0237$ & $0.1807 \pm 0.0220$ \\
0.95 & $0.0910 \pm 0.0448$ & $0.1258 \pm 0.0615$ & $0.0961 \pm 0.0535$ \\
\bottomrule
\end{tabular}
\end{table}

\begin{table}[ht]
\centering
\caption{Estimation performance across $U$ values for experiments in Section~\ref{sec:adult}. Entries are mean $\pm$ std.}
\label{tab:logistic-real-data}
\begin{tabular}{c|ccc}
\toprule
$U$ & DVCM & GLR & TL \\
\midrule
0.05 & $0.1302 \pm 0.0033$ & $0.0246 \pm 0.0046$ & $0.0258 \pm 0.0050$ \\
0.15 & $0.2811 \pm 0.0052$ & $0.2222 \pm 0.0077$ & $0.2237 \pm 0.0076$ \\
0.25 & $0.4453 \pm 0.0069$ & $0.4444 \pm 0.0089$ & $0.4439 \pm 0.0080$ \\
0.35 & $0.5184 \pm 0.0073$ & $0.5120 \pm 0.0071$ & $0.5123 \pm 0.0074$ \\
0.45 & $0.5486 \pm 0.0089$ & $0.5405 \pm 0.0093$ & $0.5404 \pm 0.0091$ \\
0.55 & $0.5549 \pm 0.0062$ & $0.5406 \pm 0.0066$ & $0.5402 \pm 0.0066$ \\
0.65 & $0.5355 \pm 0.0136$ & $0.5331 \pm 0.0114$ & $0.5333 \pm 0.0121$ \\
0.75 & $0.4708 \pm 0.0214$ & $0.4645 \pm 0.0231$ & $0.4662 \pm 0.0237$ \\
0.85 & $0.4124 \pm 0.0217$ & $0.3934 \pm 0.0246$ & $0.3967 \pm 0.0266$ \\
0.95 & $0.3462 \pm 0.0229$ & $0.3340 \pm 0.0290$ & $0.3308 \pm 0.0351$ \\
\bottomrule
\end{tabular}
\end{table}

\end{appendix}


\bibliography{ref}
\bibliographystyle{IEEEtranN}


\end{document}